\documentclass[12pt]{amsart}

\usepackage{tikz,amssymb,amsmath,fullpage,listings,amsthm,graphicx,enumitem,hyperref,tikz-cd}
\usetikzlibrary{patterns,decorations.pathreplacing}

\newcommand{\area}{\mathsf{area}}
\newcommand{\dinv}{\mathsf{dinv}}
\newcommand{\bounce}{\mathsf{bounce}}

\newtheorem{theorem}{Theorem}[section]
\newtheorem{proposition}[theorem]{Proposition}
\newtheorem{conjecture}[theorem]{Conjecture}
\newtheorem{corollary}[theorem]{Corollary}
\newtheorem{lemma}[theorem]{Lemma}

\theoremstyle{remark}
\newtheorem{remark}[theorem]{Remark}

\numberwithin{equation}{section}

\begin{document}

\title{Decorated Dyck paths, the Delta conjecture,\\ and a new $q,t$-square}

\author[M. D'Adderio]{Michele D'Adderio}
\address{Universit\'e Libre de Bruxelles (ULB)\\D\'epartement de Math\'ematique\\ Boulevard du Triomphe, B-1050 Bruxelles\\ Belgium}\email{mdadderi@ulb.ac.be}

\author[A. Vanden Wyngaerd]{Anna Vanden Wyngaerd}
\address{Universit\'e Libre de Bruxelles (ULB)\\D\'epartement de Math\'ematique\\ Boulevard du Triomphe, B-1050 Bruxelles\\ Belgium}\email{anvdwyng@ulb.ac.be}

\begin{abstract}
We discuss the combinatorics of the decorated Dyck paths appearing in the Delta conjecture framework in \cite{haglundremmelwilson} and \cite{zabrocki}, by introducing two new statistics, $\mathsf{bounce}$ and $\mathsf{bounce}'$. We then provide plethystic formulae for their $q,t$-enumerators, by proving in this way a decorated version of Haglund's $q,t$-Schr\"{o}der theorem, answering a question in \cite{haglundremmelwilson}. In particular we provide both an algebraic and a combinatorial explanation of a symmetry conjectured in \cite{haglundremmelwilson} and \cite{zabrocki}. Finally, we show a link between these results and the $q,t$-square theorem of Can and Loehr \cite{canloehr} (conjectured by Loehr and Warrington \cite{loehrwarringtonqtsquare}), getting a new $q,t$-square, whose specialization at $t=1/q$ is identical to the classical one, and hinting to a broader phenomenon of this type.
\end{abstract}

\maketitle

\tableofcontents

\section*{Introduction}

Motivated by the problem of proving the Schur positivity of the modified Macdonald polynomials, Garsia and Haiman introduced the $\mathfrak{S}_n$-module of diagonal harmonics, and they conjectured that its Frobenius characteristic is $\nabla e_n$, where the operator $\nabla$ is just a special case of a whole family of so called delta operators $\Delta_f$, where $f$ is a symmetric function.

In \cite{hhlru} a combinatorial formula for $\nabla e_n$ has been conjectured, in terms of labelled Dyck paths (i.e. parking functions). This formula has been known as the \emph{shuffle conjecture}, and it has been recently proved by Carlsson and Mellit in \cite{carlssonmellit}, where they actually proved a refinement of this formula, known as the \emph{compositional shuffle conjecture}, proposed in \cite{haglundmorsezabrocki}.

A generalization of the shuffle conjecture, known as the \emph{delta conjecture}, has been proposed in \cite{haglundremmelwilson}, in terms of decorated labelled Dyck paths. Some special cases and consequences of this conjecture have been proved: see \cite{haglundremmelwilson} for a survey of these partial results. We should mention here also the work of Romero in \cite{romero}, who proved the conjecture at $q=1$, as it does not appear in \cite{haglundremmelwilson}. To this date, the general problem remains open.

\medskip

In \cite{zabrocki}, Zabrocki established one of these consequences, the so called \emph{$4$-variable Catalan conjecture}. In fact, this conjecture, stated in \cite{haglundremmelwilson}, makes other predictions that have not been explained in \cite{zabrocki}. One of them is the symmetry 
\begin{equation} \label{eq:HRW_symmetry}
\langle\Delta_{h_{\ell}}\nabla e_{n-\ell},s_{k+1,1^{n-\ell-k-1}}\rangle=\langle\Delta_{h_{k}}\nabla e_{n-k},s_{\ell+1,1^{n-k-\ell-1}}\rangle\qquad \text{for }n>k+\ell.
\end{equation}

In \cite{haglundremmelwilson} and \cite{wilsonPhD}, the authors asked for a proof of the decorated version of the famous $q,t$-Schr\"{o}der theorem of Haglund \cite{haglundschroeder}. This is a combinatorial interpretation of the formula
\begin{equation} \label{eq:qtSchroederLHS}
\langle \Delta_{e_{a+b-k-1}}'e_{a+b},e_a h_{b}\rangle
\end{equation}
as a sum of weights $q^{\mathsf{dinv}(D)}t^{\mathsf{area}(D)}$ of decorated Dyck paths.
 
\medskip
In the present article we provide a solution of the last two problems. In particular we give both an algebraic and a combinatorial explanation of the symmetry in \eqref{eq:HRW_symmetry}, and we prove three combinatorial interpretations of \eqref{eq:qtSchroederLHS}.

In fact, we introduce two more statistics on decorated Dyck paths, $\mathsf{bounce}$ and $\mathsf{bounce}'$, which both extend Haglund's bounce statistic on Dyck paths. We prove that the zeta map of Haglund (also known as sweep map) switches the bistatistics $(\mathsf{dinv},\mathsf{area})$ and $(\mathsf{area},\mathsf{bounce})$ and the numbers of decorated rises and decorated peaks. Also, we provide a recursion for the $q,t$-enumerators of these two bistatistics, which matches a recursion in \cite{wilsonPhD}.

We prove a recursion also for the $q,t$-enumerator of the bistatistic $(\mathsf{area},\mathsf{bounce}')$, and we define a bijection $\psi$ on decorated Dyck paths that preserves the bistatistic $(\mathsf{area},\mathsf{bounce}')$ but switches the numbers of decorated rises and decorated peaks: this will eventually explain combinatorially the symmetry in \eqref{eq:HRW_symmetry}.

On the symmetric function side, we introduce three families of plethystic formulae, and we provide recursions for each one of them. The key ingredient to prove these recursions is the following summation formula, that generalizes a result of Haglund.
\begin{theorem}
	For $m,k\geq 1$ and $\ell\geq 0$, we have
	\begin{equation} 
	\sum_{\gamma\vdash m}\frac{\widetilde{H}_\gamma[X]}{w_\gamma} h_k[(1-t)B_\gamma]e_\ell[B_\gamma]=\qquad \qquad \qquad \qquad  \qquad \qquad  \qquad \qquad 
	\end{equation}
	\begin{align*} 
	=\sum_{j=0}^{\ell} t^{\ell-j}\sum_{s=0}^{k}q^{\binom{s}{2}} \begin{bmatrix}
	s+j\\
	s
	\end{bmatrix}_q \begin{bmatrix}
	k+j-1\\
	s+j-1
	\end{bmatrix}_qh_{s+j}\left[\frac{X}{1-q}\right] h_{\ell-j}\left[\frac{X}{M}\right] e_{m-s-\ell}\left[\frac{X}{M}\right].
	\end{align*}
\end{theorem}
Moreover, we establish an identity that will provide both an explanation of the symmetry in \eqref{eq:HRW_symmetry} and a crucial step in the proof of the decorated $q,t$-Schr\"{o}der.
\begin{theorem} 
	For all $a,b,k\in \mathbb{N}$, with $a\geq 1$, $b\geq 1$ and $1\leq k\leq a$, we have
	\begin{equation}  \label{eq:intro1}
	\langle \Delta_{e_a}'\Delta_{e_{a+b-k-1}}'e_{a+b},h_{a+b}\rangle = \langle \Delta_{h_k}\Delta_{e_{a-k}}' e_{a+b-k},e_{a+b-k}\rangle.
	\end{equation}
\end{theorem}

Combining the algebraic recursions with the combinatorial ones, we are able to establish combinatorial interpretations of each family of plethystic formulae in terms of $q,t$-enumerators of bistatistics involving $\mathsf{area}$, $\mathsf{dinv}$, $\mathsf{bounce}$ and $\mathsf{bounce}'$. Combining these with \eqref{eq:intro1}, we are able to prove the three combinatorial interpretations of \eqref{eq:qtSchroederLHS} mentioned above (see Section~2 for the relevant definitions):
\begin{theorem}
For $a,b,k\in \mathbb{N}\cup\{0\}$, $a+b\geq k+1$, we have
	\begin{align}
	\langle \Delta_{e_{a+b-k-1}}'e_{a+b},e_a h_{b}\rangle & = \sum_{D\in  \mathcal{D}_{a+b}^{(b,k)}}q^{\mathsf{dinv}(D)}t^{\mathsf{area}(D)}\\
	 & =\sum_{D\in \mathcal{D}_{a+b}^{(b,k)}}q^{\mathsf{area}(D)}t^{\mathsf{bounce}(D)}\\
	  & =\sum_{D\in\mathcal{D}_{a+b}^{(b,k)}}q^{\mathsf{area}(D)}t^{\mathsf{bounce'}(D)}.
	\end{align}
\end{theorem}

Finally, we establish a connection between the delta conjecture for $\Delta_{e_{n-1}}e_n$ and
the so called \emph{square conjecture} of Loehr and Warrington \cite{loehrwarringtonqtsquare}, recently proved in \cite{leven} after the breakthrough in \cite{carlssonmellit}. In fact, we show that, under an easy translation, the delta conjecture for $\Delta_{e_{n-1}}e_n$ can be interpreted as a new square conjecture, i.e. as a weighted sum over labelled square paths. In particular, the analogue of the $q,t$-square proved by Can and Loehr in \cite{canloehr} corresponds to a special case of our aforementioned results, providing in this way the new $q,t$-square mentioned in the title. This $q,t$-square will have many properties in common with the original one: in particular, it involves an extension of the classical statistics $\mathsf{area}$, $\mathsf{dinv}$ and $\mathsf{bounce}$ on Dyck paths to square paths, and it will have precisely the same evaluation at $t=1/q$. In fact, we prove a symmetric function identity that suggests that this phenomenon is not isolated.

\bigskip

This article is organized in the following way.

In Section~1 we provide some background on symmetric functions. In particular we fix some notation, and we collect some identities from the literature that we are going to need later in the text.

In Section~2 we discuss the combinatorics of decorated Dyck paths. In particular we discuss four statistics: $\mathsf{area}$ and $\mathsf{dinv}$, already introduced in \cite{haglundremmelwilson}, and two new ones, $\mathsf{bounce}$ and $\mathsf{bounce}'$. We show that Haglund's $\zeta$ map (also known as sweep map) behaves nicely with $\mathsf{area}$, $\mathsf{dinv}$ and $\mathsf{bounce}$. We introduce a $\psi$ map that behaves nicely with $\mathsf{area}$ and $\mathsf{bounce}'$. Finally we prove some combinatorial recursions for the related $q,t$-enumerators.

In Section~3 we prove a summation formula, which generalizes a formula of Haglund. This will be the key for the proofs in the next section.

In Section~4 we introduce three families of plethystic formulae, and we use the result in Section~3 to provide recursions among these families.

In Section~5 we prove another identity, which will provide in particular an algebraic proof of a symmetry conjectured in \cite{haglundremmelwilson} (cf. also \cite{zabrocki}).

In Section~6 we will draw the consequences of all the results in previous sections: we will provide combinatorial interpretations of our three families of plathystic formulae in terms of our $q,t$-enumerators coming from decorated Dyck paths. In particular we will provide a decorated version of Haglund's $q,t$-Schr\"{o}der theorem, and we will show how our $\psi$ map provide a combinatorial explanation of the symmetry conjectured in \cite{haglundremmelwilson} (cf. also \cite{zabrocki}).

In Section~7 we show the relation between our results and the square conjecture of Loehr and Warrington.

\section{Symmetric functions}

\subsection{Symmetric function notation}
The main references that we will use for symmetric functions
are \cite{macdonald} and \cite{stanleybook}. 

The standard bases of the symmetric functions that will appear in our
calculations are the complete $\{h_{\lambda}\}_{\lambda}$, elementary $\{e_{\lambda}\}_{\lambda}$, power $\{p_{\lambda}\}_{\lambda}$ and Schur $\{s_{\lambda}\}_{\lambda}$ bases.

\medskip

\emph{We will use implicitly the usual convention that $e_0=h_0=1$ and $e_k=h_k=0$ for $k<0$.}

\medskip

The ring $\Lambda$ of symmetric functions can be thought of as the polynomial ring in the power
sum generators $p_1, p_2, p_3,\dots$. This ring has a grading $\Lambda=\bigoplus_{n\geq 0}\Lambda^{(k)}$ given by assigning degree $i$ to $p_i$ for all $i\geq 1$. As we are working with Macdonald symmetric functions
involving two parameters $q$ and $t$, we will consider this polynomial ring over the field $\mathbb{Q}(q,t)$.
We will make extensive use of the \emph{plethystic notation} (see \cite{loehrremmel} for a nice exposition). 

With this notation we will be able to add and subtract alphabets, which will be represented as sums of monomials $X = x_1 + x_2 + x_3+\cdots $. Then, given a symmetric function $f$, and thinking of it as an element of $\Lambda$, we denote by $f[X]$ the expression $f$ with $p_k$ replaced by $x_{1}^{k}+x_{2}^{k}+x_{3}^{k}+\cdots$, for all $k$. 

We have for example the addition formulas 
\begin{equation}
p_k[X+Y]=p_k[X]+p_k[Y]\quad \text{ and } \quad p_k[X-Y]=p_k[X]-p_k[Y],
\end{equation}
and
\begin{equation} \label{eq:e_h_sum_alphabets}
e_n[X+Y]=\sum_{i=0}^ne_{n-i}[X]e_i[Y]\quad \text{ and } \quad  h_n[X+Y]=\sum_{i=0}^nh_{n-i}[X]h_i[Y].
\end{equation}
Notice in particular that $p_k[-X]$ equals $-p_k[X]$ and not $(-1)^kp_k[X]$. As the latter sort of negative sign can be also useful, it is customary to use the notation $\epsilon$ to express it: we will have $p_k[\epsilon X] = (-1)^k p_k[X]$, so that, in general, 
\begin{equation} \label{eq:minusepsilon}
f[-\epsilon X] = \omega f[X]
\end{equation} 
for any symmetric function $f$, where $\omega$ is the fundamental algebraic involution which sends $e_k$ to $h_k$, $s_{\lambda}$ to $s_{\lambda'}$ and $p_k$ to $(-1)^kp_k$.

We denote by $\langle\, , \rangle$ the \emph{Hall scalar product} on symmetric functions, which can be defined by saying that the Schur functions form an orthonormal basis. With this definition, we have the orthogonality
\begin{equation}
\langle p_{\lambda},p_{\mu}\rangle=z_{\mu}\chi(\lambda=\mu)
\end{equation}
which defines the integers $z_{\mu}$, where $\chi(\mathcal{P})=1$ if the statement $\mathcal{P}$ is true, and $\chi(\mathcal{P})=0$ otherwise.

Recall also the \emph{Cauchy identities}
\begin{equation} \label{eq:Cauchy_identities}
e_n[XY]=\sum_{\lambda\vdash n}s_{\lambda}[X]s_{\lambda'}[Y]\quad \text{ and } \quad h_n[XY]=\sum_{\lambda\vdash n}s_{\lambda}[X]s_{\lambda}[Y].
\end{equation}

With the symbol ``$\perp$'' we denote the operation of taking the adjoint of an operator with respect to the Hall scalar product, i.e.
\begin{equation}
\langle f^\perp g,h\rangle=\langle g,fh\rangle\quad \text{ for all }f,g,h\in \Lambda.
\end{equation}

We introduce also the operator
\begin{equation}
\tau_z f[X]:=f[X+z]\qquad \text{for all }f[X]\in \Lambda,
\end{equation}
so for example
\begin{equation}
\tau_{-\epsilon} f[X]=f[X-\epsilon]\qquad \text{for all }f[X]\in \Lambda.
\end{equation}
The operator $\tau_z$ can be computed using the following formula (see \cite[Theorem~1.1]{garsiahaimantesler} for a proof):
\begin{equation}
\tau_z =\sum_{r\geq 0}z^rh_r^\perp.
\end{equation}

We record here also the following well-known identity (see \cite[Lemma~2.1]{garsiahicksstout} for a proof).

For all $\mu\vdash n$ we have 
\begin{equation} \label{eq:s_plethystic_eval}
s_{\mu}[1-u]=\left\{\begin{array}{ll}
(-u)^r(1-u) & \text{ if }\mu=(n-r,1^r)\text{ for some }r\in \{0,1,2,\dots,n-1\},\\
0 & \text{ otherwise}.
\end{array}\right.
\end{equation}

We refer also to \cite{haglundbook} for more informations on this topic.

\subsection{Macdonald symmetric function toolkit}

For a partition $\mu\vdash n$, we denote by
\begin{equation}
\widetilde{H}_{\mu}:=\widetilde{H}_{\mu}[X]=\widetilde{H}_{\mu}[X;q,t]=\sum_{\lambda\vdash n}\widetilde{K}_{\lambda \mu}(q,t)s_{\lambda}
\end{equation}
the \emph{(modified) Macdonald polynomials}, where 
\begin{equation}
\widetilde{K}_{\lambda \mu}:=\widetilde{K}_{\lambda \mu}(q,t)=K_{\lambda \mu}(q,1/t)t^{n(\mu)}\quad \text{ with }\quad n(\mu)=\sum_{i\geq 1}\mu_i(i-1)
\end{equation}
are the \emph{(modified) Kostka coefficients} (see \cite[Chapter~2]{haglundbook} for more details). 

The set $\{\widetilde{H}_{\mu}[X;q,t]\}_{\mu}$ is a basis of the ring of symmetric functions $\Lambda$ with coefficients in $\mathbb{Q}(q,t)$. This is a modification of the basis introduced by Macdonald \cite{macdonald}, and they are the Frobenius characteristic of the so called Garsia-Haiman modules (see \cite{garsiahaimanPNAS}).

If we identify the partition $\mu$ with its Ferrers diagram, i.e. with the collection of cells $\{(i,j)\mid 1\leq i\leq \mu_i, 1\leq j\leq \ell(\mu)\}$, then for each cell $c\in \mu$ we refer to the \emph{arm}, \emph{leg}, \emph{co-arm} and \emph{co-leg} (denoted respectively as $a_\mu(c), l_\mu(c), a_\mu(c)', l_\mu(c)'$) as the number of cells in $\mu$ that are strictly to the right, above, to the left and below $c$ in $\mu$, respectively (see Figure~\ref{fig:notation}).

\begin{figure}[h]
	\centering
	\begin{tikzpicture}[scale=.4]
	\draw[gray,opacity=.4](0,0) grid (15,10);
	\fill[white] (1,10)|-(3,9)|- (5,7)|-(9,5)|-(13,2)--(15.2,2)|-(1,10.2);
	\draw[gray]  (1,10)|-(3,9)|- (5,7)|-(9,5)|-(13,2)--(15,2)--(15,0)-|(0,10)--(1,10);
	\fill[blue, opacity=.2] (0,3) rectangle (9,4) (4,0) rectangle (5,7); 
	\fill[blue, opacity=.5] (4,3) rectangle (5,4);
	\draw (7,4.5) node {\tiny{Arm}} (3.25,5.5) node {\tiny{Leg}} (6.25, 1.5) node {\tiny{Co-leg}} (2,2.5) node {\tiny{Co-arm}} ;
	\end{tikzpicture}
	\caption{}
	\label{fig:notation}
\end{figure}

We set
\begin{equation}
M:=(1-q)(1-t),
\end{equation}
and we define for every partition $\mu$
\begin{align}
B_{\mu} & :=B_{\mu}(q,t)=\sum_{c\in \mu}q^{a_{\mu}'(c)}t^{l_{\mu}'(c)}\\
D_{\mu} & :=MB_{\mu}(q,t)-1\\
T_{\mu} & :=T_{\mu}(q,t)=\prod_{c\in \mu}q^{a_{\mu}'(c)}t^{l_{\mu}'(c)}\\
\Pi_{\mu} & :=\Pi_{\mu}(q,t)=\prod_{c\in \mu}(1-q^{a_{\mu}'(c)}t^{l_{\mu}'(c)})\\
w_{\mu} & :=w_{\mu}(q,t)=\prod_{c\in \mu}(q^{a_{\mu}(c)}-t^{l_{\mu}(c)+1})(t^{l_{\mu}(c)}-q^{a_{\mu}(c)+1}).
\end{align}

Notice that
\begin{equation} \label{eq:Bmu_Tmu}
B_{\mu}=e_1[B_{\mu}]\quad \text{ and } \quad T_{\mu}=e_{|\mu|}[B_{\mu}].
\end{equation}

It is useful to introduce the so called \emph{star scalar product} on $\Lambda$ given by
$$
\langle p_{\lambda},p_{\mu} \rangle_*=(-1)^{|\mu|-|\lambda|}\prod_{i=1}^{\ell(\mu)}(1-q^{\mu_i})(1-t^{\mu_i}) z_{\mu}\chi(\mu=\lambda).
$$
For every symmetric function $f[X]$ and $g[X]$ we have (see \cite[Proposition~1.8]{garsiahaimantesler})
\begin{equation}
\langle f,g\rangle_*= \langle \omega \phi f,g\rangle=\langle \phi \omega f,g\rangle
\end{equation}
where 
\begin{equation}
\phi f[X]:=f[MX]\qquad \text{ for all } f[X]\in \Lambda.
\end{equation}
For every symmetric function $f[X]$ we set
\begin{equation}
f^*=f^*[X]:=\phi^{-1}f[X]=f\left[\frac{X}{M}\right].
\end{equation}
Then for all symmetric functions $f,g,h$ we have
\begin{equation} \label{eq:hperp_estar_adjoint}
\langle h^\perp f,g\rangle_*=\langle h^\perp f,\omega \phi g\rangle=\langle f,h\omega \phi g\rangle=\langle f,\omega \phi ( (\omega h)^* \cdot g)\rangle =\langle f,  (\omega h)^* \cdot g \rangle_*,
\end{equation}
so the operator $h^\perp$ is the adjoint of the multiplication by $(\omega h)^*$ with respect to the star scalar product.

It turns out that the Macdonald polynomials are orthogonal with respect to the star scalar product: more precisely
\begin{equation} \label{eq:H_orthogonality}
\langle \widetilde{H}_{\lambda},\widetilde{H}_{\mu}\rangle_*=w_{\mu}(q,t)\chi(\lambda=\mu).
\end{equation}
These orthogonality relations give the following Cauchy identities
\begin{equation} \label{eq:Mac_Cauchy}
e_n\left[\frac{XY}{M}\right]=\sum_{\mu\vdash n} \frac{\widetilde{H}_\mu[X]\widetilde{H}_\mu[Y]}{w_\mu}\quad \text{ for all }n.
\end{equation}

The following linear operators were introduced in \cite{bergerongarsiasf,bergerongarsiashaimantesler}, and they are at the basis of the conjectures relating symmetric function coeffcients and $q,t$-combinatorics in this area. 

We define the \emph{nabla} operator on $\Lambda$ by
\begin{equation}
\nabla  \widetilde{H}_{\mu}=T_{\mu} \widetilde{H}_{\mu}\quad \text{ for all }\mu,
\end{equation}
and we define the \emph{delta} operators $\Delta_f$ and $\Delta_f'$ on $\Lambda$ by
\begin{equation}
\Delta_f \widetilde{H}_{\mu}=f[B_{\mu}(q,t)]\widetilde{H}_{\mu}\quad \text{ and } \quad 
\Delta_f' \widetilde{H}_{\mu}=f[B_{\mu}(q,t)-1]\widetilde{H}_{\mu},\quad \text{ for all }\mu.
\end{equation}
Observe that on the vector space of symmetric functions homogeneous of degree $n$, denoted by $\Lambda^{(n)}$, the operator $\nabla$ equals $\Delta_{e_n}$. Moreover, for every $1\leq k\leq n$,
\begin{equation} \label{eq:deltaprime}
\Delta_{e_k}=\Delta_{e_k}'+\Delta_{e_{k-1}}'\quad \text{ on }\Lambda^{(n)},
\end{equation}
and for any $k>n$, $\Delta_{e_k}=\Delta_{e_{k-1}}'=0$ on $\Lambda^{(n)}$, so that $\Delta_{e_n}=\Delta_{e_{n-1}}'$ on $\Lambda^{(n)}$.

We will use the following form of \emph{Macdonald-Koornwinder reciprocity} (see \cite{macdonald} p. 332 or \cite{garsiahaimantesler}): for all partitions $\alpha$ and $\beta$
\begin{equation} \label{eq:Macdonald_reciprocity}
\frac{\widetilde{H}_{\alpha}[MB_{\beta}]}{\Pi_{\alpha}}=\frac{\widetilde{H}_{\beta}[MB_{\alpha}]}{\Pi_{\beta}}.
\end{equation}

One of the most important identities in this theory is the following one (see \cite[Theorem~I.2]{garsiahaimantesler}): for every symmetric function $f[X]$ and every partition $\mu$, we have
\begin{equation} \label{eq:glenn_formula}
\langle f[X], \widetilde{H}_\mu[X+1]\rangle_*= \left.\nabla^{-1}\tau_{-\epsilon} f[X]\right|_{X=D_\mu}.
\end{equation}

\subsection{Pieri rules and summation formulae}

For a given $k\geq 1$, we define the Pieri coefficients $c_{\mu \nu}^{(k)}$ and $d_{\mu \nu}^{(k)}$ by setting
\begin{align}
\label{eq:def_cmunu} h_{k}^\perp \widetilde{H}_{\mu}[X] & =\sum_{\nu \subset_k \mu} c_{\mu \nu}^{(k)}\widetilde{H}_{\nu}[X],\\
\label{eq:def_dmunu} e_{k}\left[\frac{X}{M}\right] \widetilde{H}_{\nu}[X] & =\sum_{\mu \supset_k \nu} d_{\mu \nu}^{(k)}\widetilde{H}_{\mu}[X].
\end{align}

The following identity is \cite[Proposition~5]{bergeronhaiman}, written in the notation of \cite{garsiahaglundxinzabrocki}, which is coherent with ours:
\begin{equation} \label{eq:cmunu_recursion}
c_{\mu \nu}^{(k+1)}=\frac{1}{B_{\mu/\nu}}\sum_{\nu\subset_1 \alpha\subset_k\mu}c_{\mu \alpha}^{(k)}c_{\alpha \nu}^{(1)}\frac{T_{\alpha}}{T_{\mu}}\quad \text{ with }\quad B_{\mu/\nu}:=B_{\mu}-B_{\nu},
\end{equation}
where $\nu\subset_k \mu$ means that $\nu$ is
contained in $\mu$ (as Ferrers diagrams) and $\mu/\nu$ has $k$ lattice cells, while the
symbol $\mu\supset_k \nu$ is analogously defined. It follows from \eqref{eq:H_orthogonality} that
\begin{equation} \label{eq:rel_cmunu_dmunu}
c_{\mu \nu}^{(k)}=\frac{w_{\mu}}{w_{\nu}}d_{\mu \nu}^{(k)}.
\end{equation}
For every $m\in \mathbb{N}$ with $m\geq 1$, and for every $\gamma\vdash m$, we have 
\begin{align*}
B_{\gamma} & = e_{1}[B_{\gamma}]\\
\text{(using \eqref{eq:Mac_hook_coeff})}& = \langle \widetilde{H}_{\gamma}, e_1h_{m-1}\rangle \\
& = \langle \widetilde{H}_{\gamma}, h_1h_{m-1}\rangle \\
& = \langle h_{1}^{\perp} \widetilde{H}_{\gamma}, h_{m-1}\rangle\\
\text{(using \eqref{eq:def_dmunu})} & = \sum_{\delta\subset_1 \gamma}c_{\gamma \delta}^{(1)}\langle \widetilde{H}_{\delta}, h_{m-1}\rangle\\ 
\text{(using \eqref{eq:Mac_hook_coeff})} & = \sum_{\delta\subset_1 \gamma}c_{\gamma \delta}^{(1)},
\end{align*}	
so we get the well-known summation formula
\begin{equation} \label{eq:Pieri_sum1}
B_{\gamma}=\sum_{\delta\subset_1 \gamma}c_{\gamma \delta}^{(1)}.
\end{equation}

\subsection{$q$-notation}

We recall here some standard notations for $q$-analogues.
For $n,k\in \mathbb{N}$, we set
\begin{equation}
[0]_q:=0,\quad \text{ and }\quad [n]_q:=\frac{1-q^n}{1-q}=1+q+q^2+\cdots+q^{n-1} \quad \text{ for } n\geq 1,
\end{equation}
\begin{equation}
[0]_q!:=1\quad \text{ and }\quad [n]_q!:=[n]_q[n-1]_q\cdots [2]_q[1]_q\quad \text{ for }n\geq 1,
\end{equation}
and
\begin{equation}
\begin{bmatrix}
n\\
k
\end{bmatrix}_q:=\frac{[n]_q!}{[k]_q![n-k]_q!}\quad \text{ for }n\geq k\geq 0,\quad \text{ and }\quad  
\begin{bmatrix}
n\\
k
\end{bmatrix}_q:=0\quad \text{ for }n< k.
\end{equation}
Recall the well-known recursion
\begin{equation} \label{eq:qbin_recursion}
\begin{bmatrix}
n\\
k
\end{bmatrix}_q=q^k\begin{bmatrix}
n-1\\
k
\end{bmatrix}_q+\begin{bmatrix}
n-1\\
k-1
\end{bmatrix}_q=\begin{bmatrix}
n-1\\
k
\end{bmatrix}_q+q^{n-k}\begin{bmatrix}
n-1\\
k-1
\end{bmatrix}_q.
\end{equation}

Recall also the standard notation for the $q$-\emph{rising factorial}
\begin{equation}
(a;q)_s:=(1-a)(1-qa)(1-q^2a)\cdots (1-q^{s-1}a).
\end{equation}

It is well-known (cf. \cite[Theorem~7.21.2]{stanleybook}) that (recall that $h_0=1$)
\begin{equation} \label{eq:h_q_binomial}
h_k[[n]_q]=\frac{(q^{n};q)_k}{(q;q)_k}=\begin{bmatrix}
k+n-1\\
k
\end{bmatrix}_q\quad \text{ for } n\geq 1\text{ and }k\geq 0,
\end{equation}
and (recall that $e_0=1$)
\begin{equation} \label{eq:e_q_binomial}
e_k[[n]_q]=q^{\binom{k}{2}}\begin{bmatrix}
n\\
k
\end{bmatrix}_q\quad \text{ for all } n, k\geq 0.
\end{equation}
Also (cf. \cite[Corollary~7.21.3]{stanleybook})
\begin{equation} \label{eq:h_q_prspec}
h_k\left[\frac{1}{1-q}\right]= \frac{1}{(q;q)_k}=\prod_{i=1}^k\frac{1}{1-q^i}\quad \text{ for } k\geq 0,
\end{equation}
and
\begin{equation} \label{eq:e_q_prspec}
e_k\left[\frac{1}{1-q}\right]= \frac{q^{\binom{k-1}{2}}}{(q;q)_k}=q^{\binom{k-1}{2}}\prod_{i=1}^k\frac{1}{1-q^i}\quad \text{ for } k\geq 0.
\end{equation}

\subsection{Useful identities}

In this section we collect some results from the literature that we are going to use later in the text.

\bigskip

The symmetric functions $E_{nk}$ were introduced in \cite{garsiahaglundqtcatalan} by means of the following expansion:
\begin{equation} \label{eq:def_Enk}
e_n\left[X\frac{1-z}{1-q}\right]=\sum_{k=1}^n \frac{(z;q)_k}{(q;q)_k}E_{nk}.
\end{equation}

Notice that setting $z=q^j$ in \eqref{eq:def_Enk} we get
\begin{equation} \label{eq:en_q_sum_Enk}
e_n\left[X\frac{1-q^j}{1-q}\right]=\sum_{k=1}^n \frac{(q^j;q)_k}{(q;q)_k}E_{nk} =\sum_{k=1}^n \begin{bmatrix}
k+j-1\\
k
\end{bmatrix}_qE_{nk} .
\end{equation}
In particular, for $j=1$, we get
\begin{equation} \label{eq:en_sum_Enk}
e_n=E_{n1}+E_{n2}+\cdots +E_{nn}.
\end{equation}
The following identity is \cite[Proposition~2.2]{garsiahaglundqtcatalan}:
\begin{equation} \label{eq:garsia_haglund_eval}
\widetilde{H}_\mu[(1-t)(1-q^j)]=(1-q^j)\Pi_\mu h_j[(1-t)B_\mu].
\end{equation}
So, using \eqref{eq:Mac_Cauchy} with $Y=[j]_q=\frac{1-q^j}{1-q}$, we get
\begin{align} \label{eq:qn_q_Macexp}
e_n\left[X\frac{1-q^j}{1-q}\right] & =\sum_{\mu\vdash n}\frac{\widetilde{H}_\mu[X]\widetilde{H}_\mu[(1-t)(1-q^j)]}{w_\mu}\\
\notag & =(1-q^j)\sum_{\mu\vdash n}\frac{\Pi_\mu \widetilde{H}_\mu[X]h_j[(1-t)B_\mu]}{w_\mu}.
\end{align}

\bigskip

For $\mu\vdash n$, Macdonald proved (see \cite{macdonald} p. 362) that
\begin{equation} \label{eq:Mac_hook_coeff_ss}
\langle \widetilde{H}_{\mu},s_{(n-r,1^r)}\rangle=e_r[B_{\mu}-1],
\end{equation}
so that, since by Pieri rule $e_rh_{n-r}=s_{(n-r,1^r)}+s_{(n-r+1,1^{r-1})}$,
\begin{equation} \label{eq:Mac_hook_coeff}
\langle \widetilde{H}_{\mu},e_rh_{n-r}\rangle=e_r[B_{\mu}].
\end{equation}

The following well-known identity is an easy consequence of \eqref{eq:Mac_hook_coeff}.
\begin{lemma} \label{lem:Mac_hook_coeff}
For any symmetric function $f\in \Lambda^{(n)}$,
\begin{equation} \label{eq:lem_e_h_Delta}
\langle \Delta_{e_{d}}f,h_n \rangle = \langle f,e_dh_{n-d} \rangle.
\end{equation}
\end{lemma}
\begin{proof}
Checking it on the Macdonald basis elements $\widetilde{H}_\mu\in \Lambda^{(n)}$, using \eqref{eq:Mac_hook_coeff}, we get
$$
\langle \Delta_{e_{d}}\widetilde{H}_\mu,h_n \rangle=e_d[B_\mu] \langle \widetilde{H}_\mu,h_n \rangle =e_d[B_\mu] =\langle \widetilde{H}_\mu,e_dh_{n-d} \rangle.
$$
\end{proof}
The following lemma is due to Haglund. 
\begin{lemma}[Corollary 2 in \cite{haglundschroeder}]
	For positive integers $d,n$ and any symmetric function $f\in \Lambda^{(n)}$,
	\begin{equation}
		\langle \Delta_{e_{d-1}}e_n,f \rangle = \langle \Delta_{\omega f}e_d,h_d \rangle.
	\end{equation}
\end{lemma}
The following theorem is due to Haglund.
\begin{theorem}[Theorem 2.11 in \cite{haglundschroeder}] \label{thm:Haglund_formula}
For $n,k\in \mathbb{N}$ with $1\leq k\leq n$,
$$
\langle \nabla E_{n,k},h_n\rangle=\chi(n=k).
$$
In addition, if $m>0$ and $\lambda \vdash m$,
\begin{equation}\langle \Delta_{s_\lambda}\nabla E_{n,k},h_n\rangle =t^{n-k}\langle \Delta_{h_{n-k}}e_m\left[ X\frac{1-q^k}{1-q}\right], s_{\lambda'}\rangle,
\end{equation}
or equivalently
\begin{equation}
\langle \Delta_{s_\lambda}\nabla E_{n,k},h_n\rangle =t^{n-k}\sum_{\mu\vdash m} \frac{(1-q^k)h_k[(1-t)B_\mu]h_{n-k}[B_\mu]\Pi_{\mu} \widetilde{K}_{\lambda' \mu}}{w_{\mu}} .
\end{equation}
\end{theorem}
The following expansions are also well-known, and they can be deduced from Cauchy identities \eqref{eq:Mac_Cauchy}.
\begin{proposition} 
For $n\in \mathbb{N}$ we have
\begin{equation} \label{eq:en_expansion}
e_n[X]=e_n\left[\frac{XM}{M}\right]=\sum_{\mu\vdash n}\frac{M B_\mu \Pi_{\mu} \widetilde{H}_\mu[X]}{w_\mu}.
\end{equation}
Moreover, for all $k\in \mathbb{N}$ with $0\leq k\leq n$, we have
\begin{equation} \label{eq:e_h_expansion}
h_k\left[\frac{X}{M}\right] e_{n-k}\left[\frac{X}{M}\right] =\sum_{\mu\vdash n}\frac{e_k[B_\mu]\widetilde{H}_\mu[X]}{w_\mu}.
\end{equation}
\end{proposition}

\section{Combinatorics of decorated Dyck paths}

\subsection{Basic definitions}

We will use the notation $\mathcal D_n$ to denote the set of Dyck paths of size $n$, i.e. the set of lattice paths starting from $(0,0)$ and ending at $(n,n)$, using only north and east steps and staying weakly above the diagonal $x=y$, also called the \emph{main diagonal}. Each Dyck path $D\in \mathcal D_n$ can be described uniquely by a sequence $w_1(D)w_2(D)\cdots w_n(D)$ of $n$ integers called the \emph{area word} of $D$, where $w_i(D)$ is the number of whole squares on the $i$-th row between $D$ and the main diagonal. A sequence $w_1w_2\cdots w_n$ of $n$ nonnegative integers is an area word if and only if $w_1=0$ and $w_{i+1}\leq w_i+1$ for $i=1,...,n-1$.  

We distinguish two types of vertical steps of a path $D\in \mathcal D_n$: a \emph{peak} is a vertical step followed by a horizontal step, and a \emph{rise} is a vertical step followed by another vertical step. Given a subset of all the vertical steps of $D$, we decorate its peaks with the symbol $\circ$ and its rises with the symbol $\ast$. By $\mathcal D_n^{(a,b)}$ we denote the set of Dyck paths of size $n$ with $a$ decorated peaks and $b$ decorated rises. See Figure~\ref{fig:example_deco_Dyck} for an example.

\begin{figure}[h!] 
	\begin{tikzpicture}
	\draw[gray](0,0)grid(5,5);
	\draw[gray](0,0)--(5,5);
	\draw[blue, opacity=.4,ultra thick](0,0)--(0,2)--(1,2)--(1,3)--(3,3)|-(5,5);
	\draw (-.5,1.4) circle (3pt) (2.5,4.5) circle (3pt) ;
	\draw (2.5,3.5) node {$\ast$};
	\end{tikzpicture}
	\caption{Example of an element in $\mathcal D_5^{(2,1)}$}
	\label{fig:example_deco_Dyck}
\end{figure}

\begin{remark}
	Note that if we call \emph{fall} a horizontal step that is followed by a another horizontal step, then there is a natural way of mapping rises into falls bijectively. Indeed the endpoint of a rise is a point where $D$ vertically crosses a certain diagonal parallel to the main diagonal. Since the path must end at the main diagonal, it must cross the same diagonal horizontally with a fall at least once. We map our rise in the first of such falls: this clearly yields a bijective map (see Figure~\ref{rises vs falls}). Therefore we might equivalently decorate falls instead of the corresponding rises. 
\end{remark}

\begin{figure}[h!]
	\centering
	\begin{tikzpicture}[rotate=0, scale=0.6]
	\draw[gray] (0,0) grid (11,11)(0,0) to (11,11)(-1,1) to (10,12);
	\draw[blue, ultra thick, opacity=0.4] (0,0) to (0,2) to (1,2) to (1,3) to (2,3) to (2,6) to (4,6) to (4,9) to (5,9) to (5,10) to (7,10) |-(11,11);
	\draw[ultra thick] (2,3) to (2,4) (8,11) to (9,11) ;
	\draw[];
	\fill[pattern=north west lines, pattern color=gray] (2,4) rectangle (4,5) (8,11) rectangle (9,9);
	\draw  (1.5,3.5) node {$\ast$};
	\end{tikzpicture} 
	\caption{Correspondence of rises and falls}
	\label{rises vs falls}
\end{figure}

\subsection{The statistics $\mathsf{area}$, $\mathsf{dinv}$, $\mathsf{bounce}$ and $\mathsf{bounce}'$.} 

We introduce here four statistics on the set $\mathcal D_n^{(a,b)}$, all of which are generalisations of the usual statistics on undecorated Dyck paths. The $\area$ and $\dinv$ statistics are the ``unlabelled'' versions of the statistics defined on labelled decorated Dyck paths in \cite{haglundremmelwilson} (cf. also \cite{zabrocki}). The $\mathsf{bounce}$ and $\mathsf{bounce}'$ are new.

\medskip

Take $D\in \mathcal{D}_n^{(a,b)}$, and let $w_1(D)w_2(D)\cdots w_n(D)$ be its area word. 

Let $\mathsf{Rise}(D)\subseteq \{1,2,\dots,n-1\}$ be the set of indices such that $i\in \mathsf{Rise}(D)$ if the $i$-th vertical step of $D$ is a decorated rise. We define the \emph{area} of $D$ as
\begin{equation} 
\mathsf{area}(D)= \sum_{i=1} ^n w_i(D)- \sum_{j\in \mathsf{Rise(D)}} w_{j+1}(D).
\end{equation}
For a more visual definition, the area is the number of whole squares that lie between the path and the main diagonal, except for the ones in the rows containing the vertical steps following a decorated rise.

Equivalently, we could define the area of $D$ as the number of whole squares that lie between the path and the main diagonal, except the ones in the columns containing falls corresponding to decorated rises. For example, the path in Figure~\ref{statdef} has area equal to $6$ (grey in the picture). 

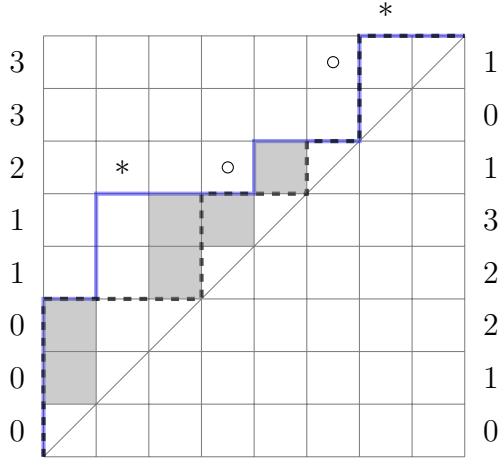
\begin{figure}
\centering
\begin{tikzpicture}[scale=.7]
			\draw (8.5,.5) node {0} (8.5,1.5) node {1} (8.5,2.5) node {2} (8.5,3.5) node {2} (8.5,4.5) node {3} (8.5,5.5) node {1} (8.5,6.5) node {0} (8.5,7.5) node {1};
			\draw (-.5,.5) node {0} (-.5,1.5) node {0} (-.5,2.5) node {0} (-.5,3.5) node {1} (-.5,4.5) node {1} (-.5,5.5) node {2} (-.5,6.5) node {3} (-.5,7.5) node {3};
			\draw[gray](0,0) grid (8,8) (0,0)--(8,8);
			\draw[blue, ultra thick, opacity=0.4](0,0)--(0,3)--(1,3)--(1,5)--(4,5)|-(6,6)|-(8,8);
			\draw  
			(3.5,5.5) circle (3pt)
			(1.5,5.5) node {$\ast$}
			(6.5,8.5) node {$\ast$}
			(5.5,7.5) circle (3pt) ;
			\fill[pattern=north west lines, pattern color=gray];
			\draw[ultra thick, dashed, opacity=.7](0,0)|-(3,3)|-(5,5)|-(6,6)|-(8,8);
			\fill[gray, opacity=.4](0,1) rectangle (1,3) (2,3) rectangle (3,5)rectangle (4,4) (4,5) rectangle(5,6);
\end{tikzpicture}
\caption{$D\in \mathcal {D}_{8}^{(2,2)}$, its first bounce word (left) and its area word (right).}\label{statdef}
\end{figure}

Similarly let $\mathsf{Peak}(D)\subseteq \{1,2,\dots,n\}$ be the set of indices such that $i\in \mathsf{Peak}(D)$ if the $i$-th vertical step of $D$ is a decorated peak. We define the \emph{dinv} of $D$, denoted $\mathsf{dinv}(D)$, to be the number of pairs $(i,j)\in \{1,...,n\}\times \{1,...,n\}$ with $i<j$ such that either \begin{enumerate}
	\item $w_i(D)=w_j(D)$ and $i\not \in \mathsf{Peak}(D)$
\end{enumerate}
or, \begin{enumerate}[resume]
	\item  $w_i(D)=w_j(D)+1$ and $j\not \in \mathsf{Peak}(D)$.
\end{enumerate} We refer to the first and second kind of pairs as \emph{primary} and \emph{secondary dinv}, respectively. For example, the path in Figure \ref{statdef} has dinv equal to 5: 4 primary and 1 secondary.

The definitions of our two bounce statistics both start from the usual definition of the bounce of undecorated Dyck paths, but they are modified (in two different ways) by the decorations on the peaks. 

\medskip

The following definition was suggested by a conversation with Alessandro Iraci \cite{iraci}. The \emph{first bounce path} of a decorated Dyck path $D\in \mathcal{D}_n^{(a,b)}$ starts in $(0,0)$ and travels north until it encounters the beginning of an east step of $D$, then it turns east until it hits the main diagonal, then it turns north again, and so on. Thus it continues until it arrives at $(n,n)$. We label the vertical steps of the first bounce path as follows: all the steps before the first ``bounce", i.e. the first time that the bounce path changes direction from north to east, are labelled with a $0$; next the vertical steps are labelled with a $1$, again until the path ``bounces"; then with a $2$, and so on. The labels of the bounce path from bottom to top create a sequence of integers $b_1(D)b_2(D)\cdots b_n(D)$ which we call the \emph{first bounce word}. We define the \emph{first bounce} of $D$ as 
\begin{equation}
\mathsf{bounce}(D):=\sum_{i=1}^nb_i(D)-\sum_{j\in \mathsf{Peak}(D) }b_j(D)=\sum_{i\in \{1,...,n\}\setminus \mathsf{Peak(D)}}b_i(D) .
\end{equation}
The path in Figure~\ref{statdef} has bounce equal to 5.

We now describe how to compute our \emph{second bounce} statistic of a decorated Dyck path $D\in \mathcal{D}_n^{(a,b)}$. Delete all the horizontal steps that directly follow decorated peaks. When deleting the horizontal steps, move the rest of the path one square to the left, together with part of the main diagonal so that the number of squares between the path and the main diagonal on each row remains the same.  In this way we end up with a path from $(0,0)$ to $(n-a,n)$ and some line segments that form the new \emph{main diagonal}. This is what is called the corresponding \emph{leaning stack} in \cite{haglundremmelwilson}. See Figure~\ref{fig:ex_second_bounce} for an example. 

Now we draw the \emph{second bounce path} in this new object as described before: start from $(0,0)$, bounce at the beginning of horizontal steps of the path and against the main diagonal. Then label the vertical steps as before to get a \emph{second bounce word} $b'_1(D)b_2'(D)\cdots b_n'(D)$. The \emph{second bounce} of $D$ is then defined as
\begin{equation}
\mathsf{bounce}'(D):=\sum_{i=1}^nb'_i(D).
\end{equation}
In the example $D\in \mathcal{D}_{11}^{(2,0)}$ of Figure~\ref{fig:ex_second_bounce} its second bounce word is $00111111122$, so that its second bounce is $\mathsf{bounce}'(D)=11$.

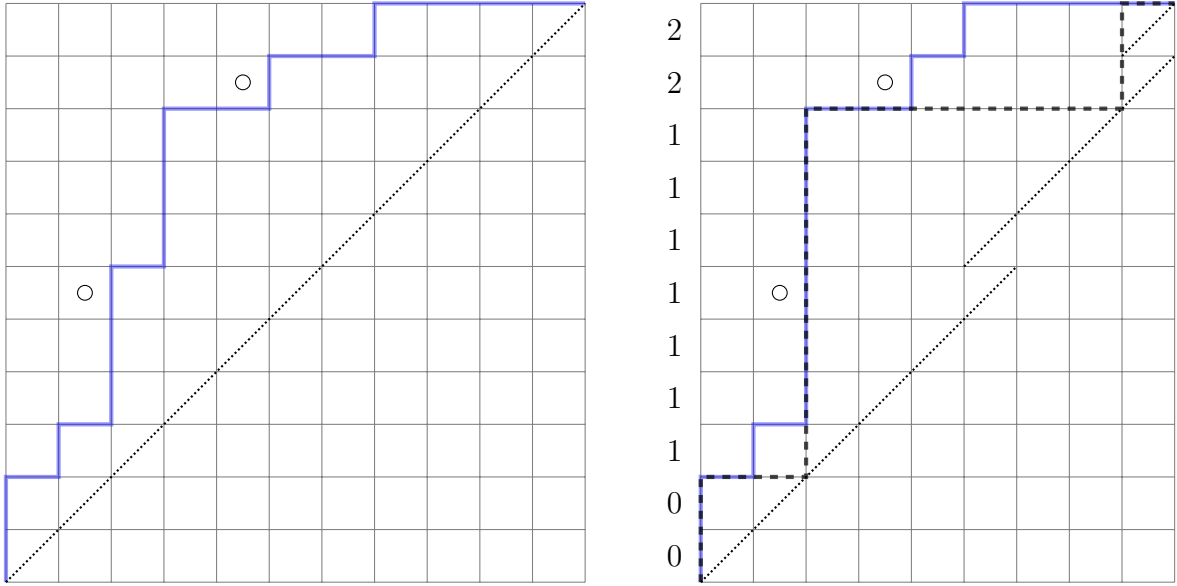
\begin{figure}[h]
	\centering
	\begin{minipage}{.55\textwidth}
		\centering
		\begin{tikzpicture}[scale=0.7]
		\draw[gray] (0,0) grid (11,11);
		\draw[blue, ultra thick, opacity=0.4] (0,0) to (0,2) to (1,2) to (1,3) to (2,3) to (2,6) to (3,6) to (3,9) to (5,9) to (5,10) to (7,10) |-(11,11);
		\draw[thick,densely dotted] (0,0) to (11,11);
		\draw (1.5, 5.5) circle (4pt) (4.5,9.5) circle (4pt);
		\end{tikzpicture} 
		
	\end{minipage}%
	\begin{minipage}{.45\textwidth}
		\centering
		\begin{tikzpicture}[scale=0.7]
		\draw[gray] (0,0) grid (9,11);
		\draw[blue, ultra thick, opacity=0.4] (0,0) to (0,2) to (1,2) to (1,3) to (2,3) to (2,6) to (2,9) to (4,9) to (4,10) to (5,10) |-(9,11);
		\draw (1.5, 5.5) circle (4pt) (3.5,9.5) circle (4pt);
		\draw[ultra thick, dashed,opacity=0.7](0,0)|-(2,2)|-(8,9)|-(9,11);
		\draw[thick,densely dotted] (0,0) to (6,6) (5,6) to (9,10) (8,10) to (9,11);
		\draw (-.5,.5) node {0} (-.5,1.5) node {0}(-.5,2.5) node {1}(-.5,3.5) node {1}(-.5,4.5) node {1}(-.5,5.5) node {1}(-.5,6.5) node {1}(-.5,7.5) node {1}(-.5,8.5) node {1}(-.5,9.5) node {2}(-.5,10.5) node {2} ;
		\end{tikzpicture} 
	\end{minipage}
	\caption{Construction of the second bounce path}
	\label{fig:ex_second_bounce}
\end{figure}

\begin{remark}
	Since a peak is a vertical step followed by a horizontal step, one could think of these two steps together as a diagonal step. So $\mathcal D_n^{(a,b)}$ can be interpreted as the set of \emph{decorated Schr\"oder paths} with $a$ diagonal steps and $b$ decorated rises. The definitions of $bounce$ and $dinv$ given in \cite[Chapter~4]{haglundbook} coincide with our definitions of $\mathsf{bounce'}$  and $\mathsf{dinv}$. When $b=0$ the definition of the $area$ in \cite[Chapter~4]{haglundbook} and our $\mathsf{area}$ also coincide.
\end{remark}

Let us conclude this section by introducing and motivating some notation. For each of the three statistics that depend on the decorations on the peaks, i.e. $\mathsf{dinv}, \mathsf{bounce}$ and $\mathsf{bounce'}$, there is one specific peak that, when decorated, does not alter this statistic. Often, when dealing with one of these statistics, it will be useful to consider not the whole set of decorated Dyck paths, but rather the set of decorated Dyck paths where this specific peak is not decorated. We will use the following notations:
\begin{center}
	\begin{itemize}
		\item $ \widetilde {\mathcal D}_n^{(a,b)} \subseteq {\mathcal D}_n^{(a,b)}$  is the subset where the rightmost highest peak is never decorated;
		\item $\widehat {\mathcal D}_n^{(a,b)}\subseteq {\mathcal D}_n^{(a,b)}$ is the subset where the first peak, i.e. the peak in the leftmost column, is never decorated; 
		\item $ \overline {\mathcal D}_n^{(a,b)}\subseteq {\mathcal D}_n^{(a,b)}$ is the subset where the last peak, i.e. the peak in the highest row is never decorated. 
	\end{itemize}
\end{center}

\subsection{Haglund's $\zeta$ (sweep) map}

The first bounce behaves well under Haglund's $\zeta$ map.

\begin{theorem} \label{thm:zeta_map}
	There exists a bijective map $$\zeta: \widetilde{\mathcal D}_n^{(a,b)} \rightarrow \widehat{\mathcal D}_n^{(b,a)} $$ such that for all $D\in \widetilde{\mathcal D}_n^{(a,b)}$ \begin{align*}
		\mathsf{area}(D)&=\mathsf{bounce}(\zeta(D))\\ \mathsf{dinv}(D)&=\mathsf{area}(\zeta(D)).
	\end{align*}
\end{theorem}

\begin{proof}
	
	The map is Haglund's $\zeta$ map on undecorated Dyck paths (Theorem 3.15 in \cite{haglundbook}), we just have to specify what happens to the decorations. 
	
	Take $D\in\widetilde{\mathcal D}_n^{(a,b)}$ and rearrange its area word in ascending order. This new word, call it $u$, will be the first bounce word of $\zeta(D)$. We construct $\zeta(D)$ as follows. First draw the first bounce path corresponding to $u$. The first vertical stretch and last horizontal stretch of $\zeta(D)$ are fixed by this path. For the section of the path  between the peaks of the first bounce path we apply the following procedure: place a pen on the top of the $i$-th peak of the first bounce path and scan the area word of $D$ from left to right. Every time we encounter a letter equal to $i-1$ we draw an east step and when we encounter a letter equal to $i$ we draw a north step. By construction of the first bounce path, we end up with our pen on top of the $(i+1)$-th peak of the bounce path. Note that in an area word a letter equal to $i\neq 0$ cannot appear unless it is preceded by a letter equal to $i-1$. This means that starting from the $i$'th peak, we always start with a horizontal step which explains why $u$ is indeed the first bounce word of $\zeta(D)$.
	
	Let us now deal with the decorated rises. Let $j$ be such that the $j$-th vertical step of $D$ is a decorated rise. It follows that  $w_{j+1}(D)=w_{j}(D)+1$. Take $i$ such that $w_{j}(D)=i-1$. While scanning the area word to construct the path between the $i$-th and $(i+1)$-th peak of the first bounce path, we will encounter $w_{j}(D)=i-1$, directly followed by $w_{j+1}(D)=i$. This will correspond to a horizontal step followed by a vertical step in $\zeta(D)$, sometimes called a \emph{valley}. Decorate the first peak following this valley. This peak must be in the section of the path between the $i$-th and $(i+1)$-th peak of the first bounce path so the label of the vertical step of the first bounce path contained in the same line must be equal to $i$. It now follows from the definition of the statistics that $\mathsf{area}$ gets sent into $\mathsf{bounce}$. Note that the first peak never gets decorated in this way, since it is not preceded by a valley.

	Finally, we determine where to send the decorated peaks. The most natural thing to do is to send decorated peaks into decorated falls. Let $j$ be the index of a decorated peak. It follows that $w_{j+1}(D)\leq w_j(D)$. Take once more $i$ such that $w_{j}(D)=i-1$. When scanning the area word to construct the path between the $i$-th and $(i+1)$-th peak of the bounce path, we will encounter $w_j=i-1$ which corresponds to a horizontal step, which we call $h$. Since $w_{j+1}(D)\leq i-1$ we must encounter a letter equal to $i-1$ before we might encounter a letter equal to $i$. It follows that our scanning continues with the encounter of a letter equal to $i-1$ or that we encounter no more letters equal to $i-1$ or $i$. In the first case it is clear that $h$ is followed by a horizontal step. In the second case and when $i$ is not the last peak of the bounce path then $h$ is also followed by a horizontal step since the next portion of the path must start with a horizontal step. Only when the $i$-th peak is the last peak of the bounce word and $w_j(D)$ is the last occurrence of $i-1$ in the area word, $h$ is not followed by a horizontal step. But this means that the $j$-th vertical step was the rightmost highest peak of $D$ and so $D\not \in \widetilde{\mathcal D}_n^{(a,b)}$.

	Thus, $h$ is always a fall, which we decorate. So in the area of $\zeta(D)$ the squares directly under $h$ and above the main diagonal are not counted. We distinguish two types of squares under $h$, those that are under the first bounce path of $\zeta(D)$ and those above it. The number of squares of the first kind equals the number of letters equal to $i-1$ that occur in the area word after its $j$-th letter. The number of squares of the second kind is equal to the number of letters equal to $i$ that occur in the area word before its $j$-th letter. Since $j\in \mathsf{Peak}(D)$ the pairs formed by $j$ and the indices of these letters are exactly the pairs that are not counted as dinv. The pairs that do contribute to the dinv correspond to the squares under undecorated horizontal steps. One concludes that $\mathsf{dinv}$ gets sent into $\mathsf{area}$. 
	
	One readily checks that this map is invertible: the bounce path of $\zeta(D)$ provides the letters of the area word of $D$, and the path the interlacing of the letters. The decorations of $D$ can be obtained applying the argument above backwards.  
	
	\begin{figure*}[h]
		\begin{minipage}{.5\textwidth}
			\centering
			\begin{tikzpicture}[scale=.7]
			\draw (8.5,.5) node {0} (8.5,1.5) node {1} (8.5,2.5) node {2} (8.5,3.5) node {2} (8.5,4.5) node {3} (8.5,5.5) node {1} (8.5,6.5) node {0} (8.5,7.5) node {1};
			\draw[gray](0,0) grid (8,8) (0,0)--(8,8);
			\draw[blue, ultra thick, opacity=0.4](0,0)--(0,3)--(1,3)--(1,5)--(4,5)|-(6,6)|-(8,8);
			\draw (-.5,.5) node {$\ast$} 
			(3.5,5.5) circle (3pt)
			(.5,3.5) node {$\ast$}
			(5.5,6.5) node {$\ast$}
			(5.5,7.5) circle (3pt) ;
			\fill[pattern=north west lines, pattern color=gray](0,1) rectangle (1,2) (1,4) rectangle (4,5)(6,7) rectangle (7,8);
			\end{tikzpicture}
			\caption{$D\in \widetilde{\mathcal D}_8^{(2,3)}$ and its area word.}
		\end{minipage}%
		\begin{minipage}{.5 \textwidth}
			\centering
			\begin{tikzpicture}[scale=.7]
			\draw[gray](0,0) grid (8,8) (0,0)--(8,8);
			\draw[blue, ultra thick, opacity=0.4](0,0)|-(1,2)|-(2,4)|-(3,5)|-(7,7)|-(8,8);
			\draw[ultra thick, dashed,opacity=0.7](0,0)|-(2,2)|-(5,5)|-(7,7)|-(8,8);
			\draw (8.5,.5) node {0} (8.5,1.5) node {0} (8.5,2.5) node {1} (8.5,3.5) node {1} (8.5,4.5) node {1} (8.5,5.5) node {2} (8.5,6.5) node {2} (8.5,7.5) node {3};
			\draw (.5,3.5) circle (3pt) (1.5,4.5) circle (3pt) (3.5, 7.5) node {$\ast$} (4.5,7.5) node {$\ast$}(6.5,7.5) circle (3pt);
			\fill[pattern=north west lines, pattern color=gray] (3,4) rectangle (4,7) (4,5) rectangle (5,7) ; 
			\end{tikzpicture}
			\caption{$\zeta(D)\in \widetilde{\mathcal D}_8^{(3,2)}$ and its bounce word.}
			
		\end{minipage}
		
	\end{figure*}
	
\end{proof}
\subsection{The $\psi$ map exchanging peaks and falls} 

We define a map that explains an interesting symmetry.

\begin{theorem} \label{thm:psi_map}
	There exists a bijective map $$\psi: \overline{\mathcal D}_n^{(a,b)} \rightarrow  \overline{\mathcal D}_n^{(b,a)}$$ such that for all $D\in \mathcal D_n^{(a,b)}$ \begin{align*}
		\mathsf{area}(D)&=\mathsf{area} (\psi(D)) \\ \mathsf{bounce'}(D)&=\mathsf{bounce'}(\psi(D)).
	\end{align*}
\end{theorem}

\begin{proof}Let us first look at a simpler map, $\psi_0$, that transforms one decorated fall into a decorated peak, conserving the said statistics. Call the endpoint of the decorated fall $x$. We travel southward from $x$ until we hit the main diagonal then travel west until we hit the endpoint of a north step of the path. Call this point $y$. Now we modify the portion of the path between $x$ and $y$ by deleting the decorated east step and by adding an east step right after the point $y$. See Figure~\ref{fig:psi0}. Since $y$ was the endpoint of a north step and we added an east step after $y$, we have created a peak which we decorate. 
	
	\begin{figure}[h!]
		\centering
		\begin{minipage}{.5 \textwidth}
			\centering
			\begin{tikzpicture}
			\filldraw (2,4) circle (1.5pt) (0,2) circle (1.5pt);
			\draw (2.2,4.2)node {$x$} (-.2,2.2) node {$y$};
			\draw[gray] grid (0,0) grid (5,5) (0,0)--(5,5);
			\draw[blue, opacity=.4,ultra thick](0,0)--(0,3)-|(1,4)--(3,4)|-(5,5);
			\draw(1.5,4.5) node {$\ast$};
			\fill[pattern=north west lines, pattern color=gray](1,2) rectangle (2,4);
			\draw[thick, dashed] (2,4)--(2,2)--(0,2);
			\end{tikzpicture}
		\end{minipage}%
		\begin{minipage}{.5\textwidth}
			\centering
			\begin{tikzpicture}
			\draw (2.2,4.2)node {$x$} (-.2,2.2) node {$y$};
			\draw[gray] grid (0,0) grid (5,5) (0,0)--(5,5);
			\filldraw (2,4) circle (1.5pt) (0,2) circle (1.5pt);
			\draw[thick, dashed] (2,4)--(2,2)--(0,2);
			\draw[blue, opacity=.4,ultra thick](0,0)--(0,2)--(1,2)|-(2,3)|-(3,4)|-(5,5);
			\draw (-.5,1.5) circle (3pt);
			\end{tikzpicture}
		\end{minipage}
		\caption{The map $\psi_0$.} \label{fig:psi0}
	\end{figure}
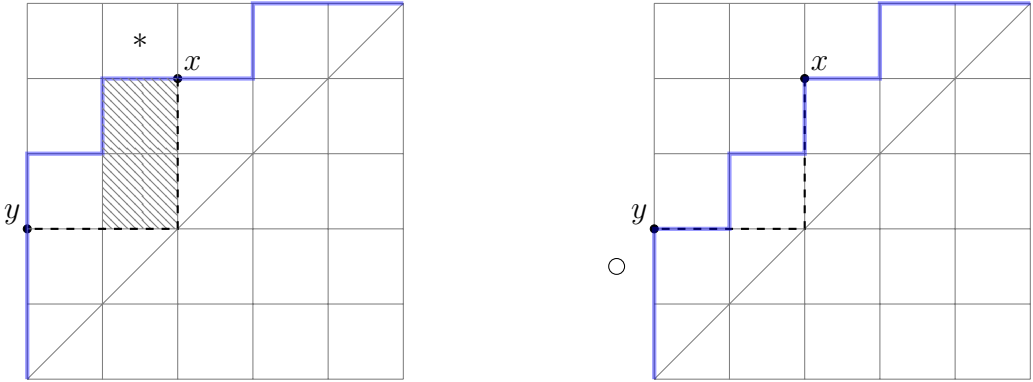
	
	Since $x$ is the endpoint of a fall, it is never on the main diagonal. It follows that the image is still a Dyck path. Furthermore, the area has not been altered since the number of squares beneath the path are exactly diminished by the number of squares under the decorated fall. Finally, recalling the definition of $\mathsf{bounce'}$, it is easy to see that the second bounce path is not altered, so that this statistic is not altered either. Clearly, $\psi_0$ is invertible. 
	
	Notice that $\psi_0^{-1}$ would not work with a decoration on the peak in the highest row: the added horizontal step would be the last of the path, so not a fall. This explains why we defined $\psi$ on $\overline {\mathcal D}_n^{(a,b)}$. 
	
	The general map $\psi$ relies on $\psi_0$ as follows: apply $\psi_0$ to all the decorated falls and $\psi_0^{-1}$ to all the decorated peaks, one by one. However, we must be careful, because the order in which we transform the decorations matters. We use the following order:
	\begin{enumerate}
		\item apply $\psi_0$, from left to right, to all decorated falls. We end up with a path all of whose decorations are on peaks;
		\item apply $\psi_0^{-1}$, from top to bottom, to all the decorated peaks that did not come from decorated falls in step (1). 
	\end{enumerate}

See Figure~\ref{fig:psi_definition} for an example.
	
	Notice that if we do not respect our rule, then the area of the final path might not be the same. Indeed $\psi_0$ requires pushing a portion of the path to the right. If there happens to be a decorated fall in this portion, then one decreases the number of squares under this step and so increases the overall area. Similarly, when applying $\psi_0^{-1}$, a portion of the path gets pushed to the left, so when there is a decorated fall in this portion the overall area decreases. These two situations are always avoided by respecting the given order. Since every step of $\psi$ preserves $\area$ and $\bounce'$, so does $\psi$. 
	
	With an obvious abuse of notation, it is clear that $\psi \circ \psi: \overline{\mathcal D}_n^{(a,b)} \to \overline{\mathcal D}_n^{(a,b)}$ is the identity map, in other words, $\psi$ is an involution, so it is bijective.

	\begin{figure}[h!]
		\centering
		\begin{minipage}{.5 \textwidth}
			\centering
			\begin{tikzpicture}[scale=.7]
			\draw[gray](0,0) grid (8,8) (0,0)--(8,8);
			\draw[blue, ultra thick, opacity=0.4](0,0)--(0,4)--(2,4)--(2,5)--(4,5)|-(5,6)|-(6,7)|-(8,8);
			\draw (2.5,5.5) node {$\ast$}
			(3.5,5.5) circle (3pt) 
			(.5,4.5) node {$\ast$} 
			(4.5,6.5) circle (3pt)
			(1.5,4.5) circle (3pt) ;
			\fill[pattern=north west lines, pattern color=gray](0,1) rectangle (1,4) (2,3) rectangle (3,5);
			\end{tikzpicture}	
		\end{minipage}%
		\begin{minipage}{.5 \textwidth}
			\centering
			\begin{tikzpicture}[scale=.7]
			\draw[gray](0,0) grid (8,8) (0,0)--(8,8);
			\draw[blue, ultra thick, opacity=0.4](0,0)|-(1,1)--(1,3)--(2,3)|-(3,4)|-(4,5)|-(5,6)|-(6,7)|-(8,8);
			\draw [ultra thick] (-.5,.5) circle (3pt) (1.5,3.5) circle (3pt);
			\draw (2.5,4.5) circle (3pt)
			(4.5,6.5) circle (3pt) 
			(3.5,5.5) circle (3pt)  ;
			\fill[pattern=north west lines, pattern color=gray](0,1);
			\end{tikzpicture}		
		\end{minipage}
		\vspace{.5cm}
		
		\begin{minipage}{\textwidth}
			\centering
			\begin{tikzpicture}[scale=.7]
			\draw[gray](0,0) grid (8,8) (0,0)--(8,8);
			\draw[blue, ultra thick, opacity=0.4](0,0)|-(1,1)--(1,3)--(2,3)|-(3,4)|-(8,8);
			\draw [ultra thick] (-.5,.5) circle (3pt) (1.5,3.5) circle (3pt)(4.5,8.5) node {$\boldsymbol{\ast}$} (5.5,8.5) node {$\boldsymbol{\ast}$} (6.5,8.5) node {$\boldsymbol{\ast}$} ;
			\draw  ;
			\fill[pattern=north west lines, pattern color=gray] (5,5)|-(6,6)|-(7,7)|-(8,8)--(4,8)--(4,5)--(5,5);
			\end{tikzpicture}	
		\end{minipage}
		\caption{The two steps in the definition of the map $\psi$}
		\label{fig:psi_definition}
	\end{figure}
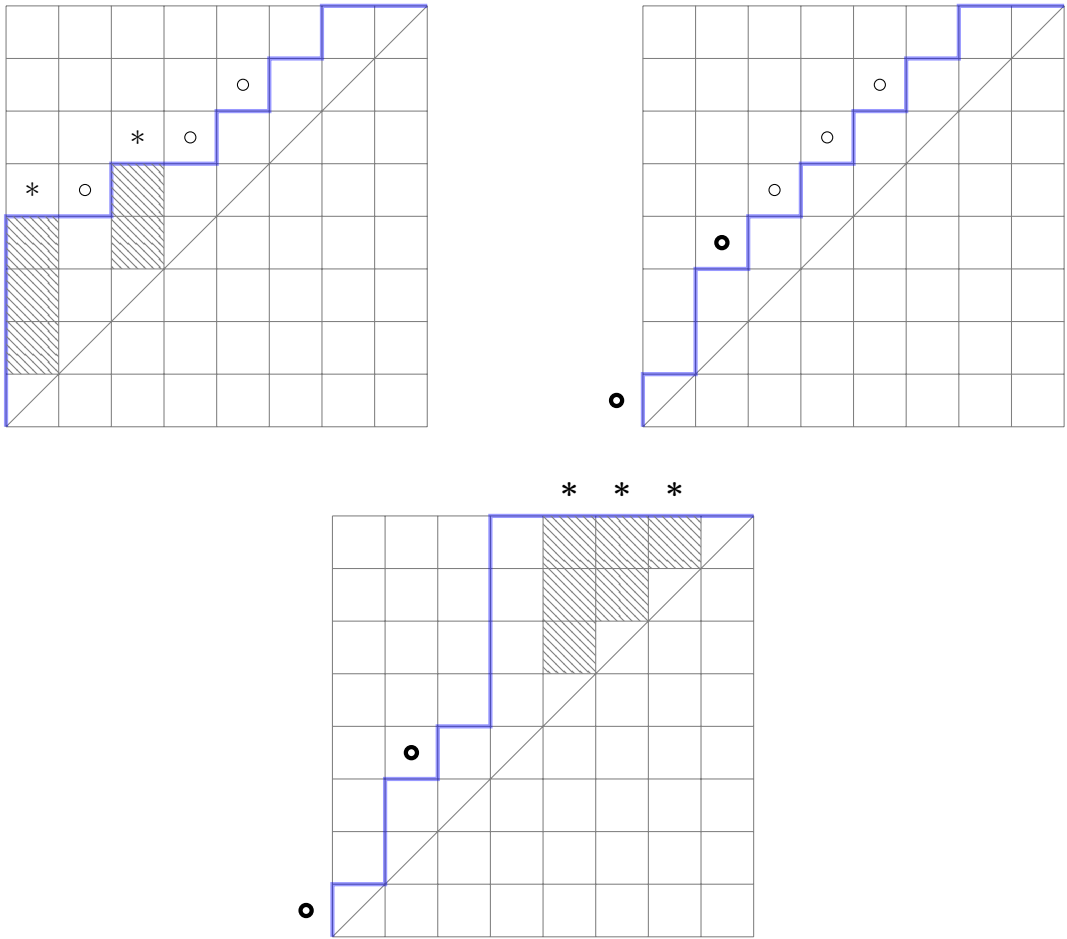
\end{proof}

\subsection{Combinatorial recursions} Let $S\subseteq \mathcal D_{n}^{(a,b)}$. We introduce some notation: 
\begin{align*}
	&\mathbf D S =\mathbf D S(q,t):= \sum_{D\in S}q^{\dinv(D)}t^{\area(D)} \\
	&\mathbf B S = \mathbf B S(q,t):= \sum_{D\in S}q^{\area(D)}t^{\bounce(D)}\\ 
	&\mathbf B' S = \mathbf B' S(q,t):= \sum_{D\in S}q^{\area(D)}t^{\bounce'(D)}.
\end{align*} In this section, we study the polynomials of this kind by breaking them into pieces and finding a recursive way to construct these pieces. 

\subsubsection{First bounce} 
Let $a,b,n,k\in \mathbb N\cup\{0\}$, with $n\geq k\geq 1$. We introduce two new sets 
\begin{equation}
\widehat{R}^{(a,b)}_{n,k}\subseteq \widehat{S}^{(a,b)}_{n,k} \subseteq \widehat{\mathcal D}_{n}^{(a,b)}:
\end{equation}
\begin{itemize}\item a path in $ \widehat{\mathcal D}_{n}^{(a,b)}$ is a path in $ \widehat{S}^{(a,b)}_{n,k}$ if its first bounce word contains exactly $k$ zeroes ; 
	\item a path in $ \widehat{S}_{n,k}^{(a,b)}$ is a path in $ \widehat{R}^{(a,b)}_{n,k}$ if its first $k$ columns do not contain any decorated fall.   
\end{itemize}

Since $\widehat{\mathcal D}_n^{(a,b)}=\sqcup_{k=1}^n \widehat{S}_{n,k}^{(a,b)}$, we have
\begin{equation}
\mathbf{B}\widehat{\mathcal D}_n^{(a,b)} = \sum_{k=1}^n \mathbf{B}\widehat{S}^{(a,b)}_{n,k} .
\end{equation}

We prove two theorems about recursions for the polynomials $\mathbf B \widehat{R}_{n,k}^{(a,b)}$ and $\mathbf B \widehat{S}_{n,k}^{(a,b)}$.

\begin{theorem} \label{thm:reco_first_bounce_S_R}
	
	Let $a,b,n,k\in \mathbb N\cup\{0\}$, $n\geq k\geq 1$. Then 
	\begin{align} 
	\mathbf B \widehat{S}_{n,n}^{(a,b)} & = \delta_{a,0} q^{\binom{n-b}{2}}\begin{bmatrix}
	n-1\\
	b
	\end{bmatrix}_q  ,
	\end{align}
	and for $n>k$
	\begin{equation}
	\mathbf B \widehat{S}_{n,k}^{(a,b)}=\sum_{s=0}^{\min(k-1,b)}{k\brack s}_q \mathbf B \widehat{R}_{n-s,k-s}^{(a,b-s)}.
	\end{equation}
\end{theorem}

\begin{proof}

First of all, consider the case $k=n$. The paths in $\widehat{S}_{n,n}^{(a,b)}$ are the paths that consist of $n$ north steps followed by $n$ east steps and so the bounce must be zero. There can be no decorations of the first and only peak of these paths so $a$ must equal zero. Of the $n$ consecutive east steps, $b$ are decorated falls. The last east step must be undecorated since it is not a fall. Choose an interlacing between the remaining $n-1-b$ undecorated and $b$ decorated east steps. The area is given by $q^{n-b\choose 2} {n-1\brack b}_q. $

	Take a path in $P \in\widehat{R}_{n-s,k-s}^{(a,b-s)}$. The first $k-s$ columns on $P$ do not contain any decorated falls. We will add $s$ falls to these first columns and prepend $s$ vertical steps to the $P$, thus obtaining a  path $P'$ in $\widehat{S}_{n,k}^{(a,b)}$ (as in Figure \ref{fig R and S}). We obtain all the possible paths of $\widehat{S}_{n,k}^{(a,b)}$  by summing over all the possible values of $s$. The bounce remains unchanged since one obtains the first bounce word of $P'$ by prepending $s$ zeroes to the first bounce word of $P$ and the decorated peaks stay in the same position. Choose an interlacing between the $k-s$ first horizontal steps of $P$ and the $s$ falls that are to be added. Reading the interlacing left to right, we prepend the decorated falls to the following undecorated horizontal steps. We may also end the interlacing with a decorated fall since $P$ has $k-s$ zeroes in its first bounce word and so the portion of $P$ in the first $k-s$ columns is followed by a horizontal step. The squares under the decorated falls do not contribute to the area but every time an undecorated horizontal step precedes a decorated fall in the interlacing, one square of area is added. This explains the factor ${k-s+s\brack s}_q={k\brack s}_q$ and proves the proposition.

	\begin{figure}[h!]
		\centering
		\begin{minipage}{.3\textwidth}
		\centering
			\begin{tikzpicture}[scale=.7]
				\draw[gray](3,3) grid (10,10) (3,3)--(10,10);
				\draw[ultra thick, dashed,opacity=0.7](3,3)|-(6,6)|-(9,9)|-(10,10);
				\draw[blue, ultra thick, opacity=.4](3,3)|-(4,6)|-(5,7)|-(6,8)--(6,9)--(8,9)|-(10,10);
				\draw(2.5,5.5) node{$\circ$} (6.5,9.5) node {$\ast$};
				\fill[gray, opacity=.4] (3,4) rectangle (4,6)(4,5) rectangle (5,7)(5,6) rectangle (6,8);
			\end{tikzpicture}
		\end{minipage}%
		\begin{minipage}{.7 \textwidth}
		\centering	
		\begin{tikzpicture}[scale=.7]
		\draw[gray, opacity=.5] (0,0) grid (10,10) (3,3) grid (10,10) (0,0)--(10,10);
		\draw[thick, opacity=.7](2.9,2.9) rectangle (10.1,10.1); 
		\draw[blue, ultra thick, opacity=.4](6,9)--(8,9)|-(10,10) (0,3)--(0,6) (1,6)-|(2,7) -|(3,8)-|(4,9);
		\draw[ultra thick, red, opacity=.9](0,0)--(0,3)(0,6)--(1,6)(4,9)--(6,9);
		\draw (.5,6.5) node {$\ast$} (4.5,9.5) node {$\ast$} (5.5,9.5) node {$\ast$} (8.5,10.5) node {$\ast$} (-.5, 5.5) node {$\circ$} (6.5,9.5) node {$\ast$} ;
		\fill[pattern=north west lines, pattern color=gray](0,1) rectangle (1,6)(4,5) rectangle (5,9)rectangle (6,6) (8,9) rectangle (9,10);
		\fill[gray, opacity=.4] (1,2) rectangle (2,4)(2,3) rectangle (3,5)(3,4) rectangle (4,6);
		\draw[ultra thick, dashed,opacity=0.7](0.1,0)|-(6,5.9)|-(9,9)|-(10,10);
		\end{tikzpicture}
		\end{minipage}

		\caption{Illustration of the proof of Theorem~\ref{thm:reco_first_bounce_R_S}} \label{fig R and S}
	\end{figure}

\end{proof}

\begin{theorem} \label{thm:reco_first_bounce_R_S}
	Let $a,b,n,k\in \mathbb N\cup\{0\}$, with $n\geq k\geq 1$. Then   
	\begin{equation}
	\mathbf{B}\widehat{R}^{(a,b)}_{n,n} =\delta_{0,a}\delta_{0,b}
	q^{n\choose 2},
	\end{equation} 
	and for $1\leq k <n$ 
	\begin{align}
		\mathbf{B}\widehat{R}^{(a,b)}_{n,k}  & =
		t^{n-k-a}
		\sum_{i=0}^a   q^{{k\choose 2}+{i+1\choose 2}}
		\sum_{h=1}^{n-k} 
		{h-1 \brack i}_q 
		{h+k-i-1\brack h}_q\\
	\notag	&\quad \times \left(\mathbf{B}\widehat{S}^{(a-i,b)}_{n-k,h} +
		\mathbf{B}\widehat{S}^{(a-i-1,b)}_{n-k,h}  \right) .
	\end{align}
\end{theorem}
\begin{proof} First suppose that $k=n$, that is, the first bounce word contains $k$ zeroes so the path must go north for $n$ steps and then east for $n$ steps. It follows that the first bounce equals zero.  It is clear that for $\widehat{R}^{(a,b)}_{n,n}$ to be nonempty $a$ and $b$ have to equal $0$ since we cannot decorate the first (and only) peak and the first $k=n$ columns may not contain a decorated fall. So there is only one path in $ \widehat{R}^{(a,b)}_{n,n}$ and its area is equal to $n\choose 2$. 
	
	Now suppose that $k<n$. We describe a procedure that uniquely describes an element of $\widehat{R}^{(a,b)}_{n,k}$.

	Start with an element $D$ of $\widehat{S}^{(a-i,b)}_{n-k,h}$ for $h\in {1,...,n-k}$. We will extend this path to obtain an element $D'$ of $\widehat{R}^{(a,b)}_{n,k}$, keeping track of the statistics during the process. We refer to Figure \ref{first bounce recursion} for a visual aid. Consider an $n\times n$ grid. Place $D$ in the top right corner of this grid such that $D$ goes from $(k,k)$ to $(n,n)$. Since the number of $0$'s in the first bounce word of $D$ equals $h$ we know that $D$ starts with $h$ vertical steps, followed by a horizontal step.  Delete these $h$ vertical steps. To obtain a path $D'$ of size $n$ we need to complete $D$ with a path from $(0,0)$ to $(k,k+h)$ which we will call $P(D)$. Note that $D'$ must have $k$ zeroes in its first bounce word, so $P(D)$ must start with $k$ vertical steps followed by a horizontal step. Hence the path that is still to be determined goes from $(1,k)$ to $(k,k+h)$. Note that we have $h\geq 1$ since $n<k$. 
	\begin{figure}[h!]
		\centering
		\begin{minipage}{0.4 \textwidth}
			\centering
			\begin{tikzpicture}[scale=0.9]
			\draw[ultra thick, dashed, opacity=.6](3,3)|-(5,5)|-(8,8)|-(10,10) ;
			\draw[] (3,3) rectangle (10,10); 
			\draw[ultra thick, blue, opacity=0.4](3,3) |-(3.5,5) |-(5.5,8)|-(10,10);  \draw (3.25,7.75)circle (2.5pt);
			\draw  (4.5, 8.25) node {$\ast$};
			\end{tikzpicture}
			
		\end{minipage}%
		\begin{minipage}{0.6 \textwidth}
			\centering
			\begin{tikzpicture}[scale=0.9]
			\draw (0,0) rectangle (10,10);
			\draw (0,0)-- (10,10);
			\draw[ultra thick, dashed, opacity=.6](0,0) |-(3,3)|-(5,5)|-(8,8)|-(10,10) ;
			\draw[thick] (2.9,2.9) rectangle (10,10); 
			\draw[ultra thick, blue, opacity=0.4](0,0) |-(.5,3)|-(2,3.5)|-(2.5,4)|- (3.5,5) |-(5.5,8)|-(10,10);
			\draw (.25, 3.25) circle (2.5pt) (3.25,7.75) circle (2.5pt) ;
			\draw (4.5, 8.25) node {$\ast$};
			\end{tikzpicture}
		\end{minipage}
		\caption{Theorem~\ref{thm:reco_first_bounce_R_S}: $D$ on the left and $D'$ on the right.}\label{first bounce recursion}
	\end{figure}
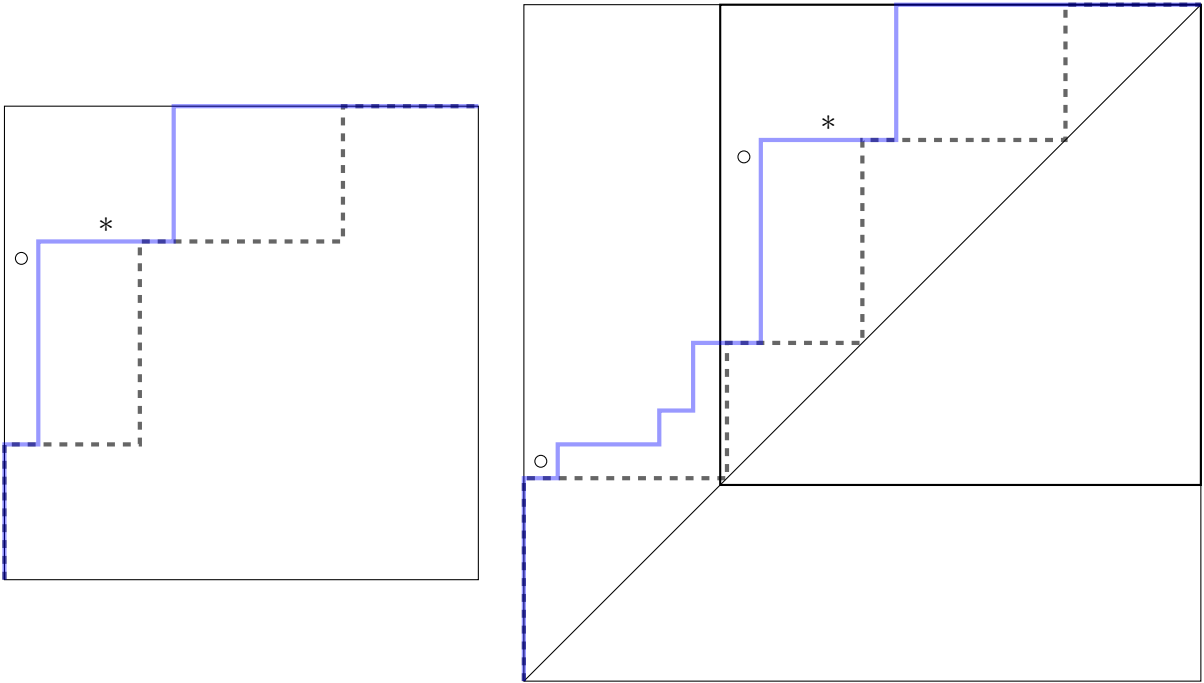
	
	At this point we can already determine the first bounce of $D'$. We have 
	\begin{equation}	
	\bounce(D')=\bounce(D)+(n-k-a).
	\end{equation} 
	Indeed its first bounce word is formed by $k$ $0$'s followed by the letters of the bounce word of $D$, all augmented by $1$. So $n-k$ letters are augmented by one, but $a$ of them will not contribute to the bounce as they will correspond to decorated peaks (all the decorated peaks are above the line $y=k$, as the first peak cannot be decorated). This explains the factor $t^{n-k-a}$.

	The area of $D$ will be the sum of the area of $D'$ and the area beneath the path $P(D)$. We construct $P(D)$ in two steps. We already observed that $P(D)$ must start with $k$ vertical steps. We construct the path from $(k,k+h)$ to $(0,k)$ keeping track of the area underneath it and above the main diagonal. We set 
	\begin{center}
		\begin{tabular}{l}
			$i$ := number of decorated peaks in $P(D)$. 
		\end{tabular}
	\end{center} The different steps are illustrated in Figure \ref{creating P(D)}.
	
	\begin{enumerate}
		\item We start by creating $i$ decorated peaks. Consider a peak as a vertical step followed by a horizontal step. \emph{We require that the top row does not contain a decorated peak.} So choose an interlacing of $i$ peaks and $h-i-1$ vertical steps. The area that lies above the bounce path and under this portion of the path is reflected by the terms $$q^{i+1\choose 2}{h-i-1+i\brack i}_q=q^{i+1\choose 2}{h-1\brack i}_q. $$

		\item Now, we have placed all the vertical steps and $i$ horizontal steps. So there are $k-i$ horizontal steps left to place. We know that the first $k$ vertical steps of $P(D)$ must be followed by a horizontal step. Imposing that we begin by a horizontal step, this leaves $k-i-1$ horizontal steps and $h$ vertical ones to interlace. The area above the bounce path that is created is counted by $${h+k-s-i-1\brack h}_q.$$ The area under the bounce path is equal to $q^{k\choose 2}$.
	\end{enumerate}
	
	Depending on the choices made in the construction of $P(D)$ we have now obtained a unique path in $\widehat{R}^{(a,b)}_{n,k}$, and kept track of its statistics. However, not all the paths of $\widehat{R}^{(a,b)}_{n,k}$ are obtained by this procedure. Indeed, in step (1), we required that the top row of $P(D)$ does not contain a decorated peak. We did this in order to be able to differentiate between the situation when the horizontal step following the decorated peak is part of $P(D)$ and the situation where it isn't. We solve this by adding another term, now starting from a path in $\widehat{S}_{n-k,h}^{(a-i-1,b)}$ and considering the last row of $P(D)$ to be always decorated (as opposed to never decorated before). The argument for the area stays exactly the same, as a decorated peak does not influence the area. For the bounce, the argument does not change either: $D$ has $a-i-1$ decorated peaks, we place $i$ decorated peaks in step (1) and consider the last row of $P(D)$ to contain a decorated peak, which totals to $a$ decorated peaks and so the bounce still goes up by $n-k-a$. 
	
	Summing over all possible values of $h,s$ and $i$, we obtain the result. 
	
	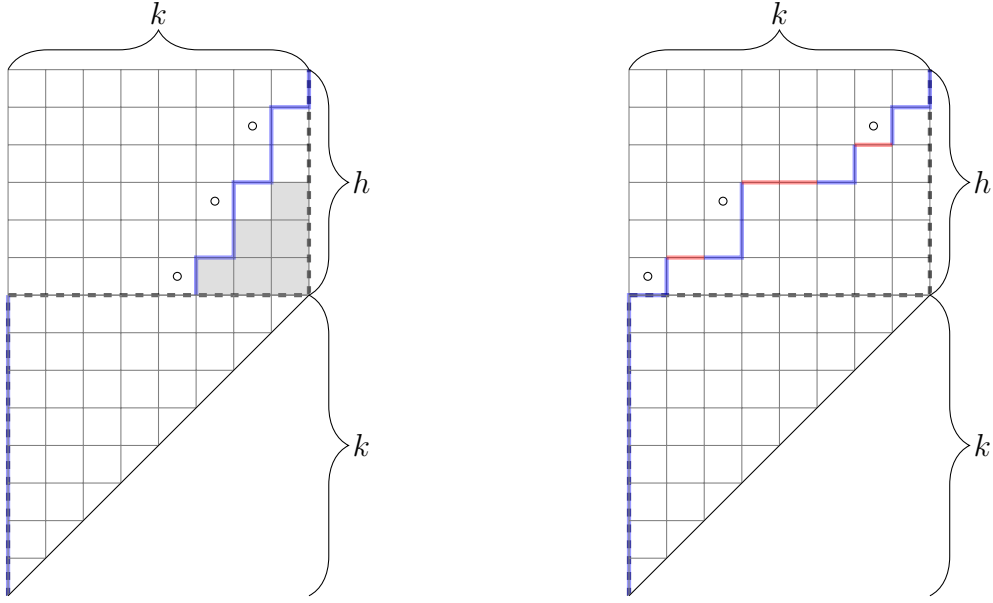
\begin{figure}
		\centering
		\begin{minipage}{0.5\textwidth}
			\centering
			\begin{tikzpicture}[scale=.5]
			\draw[gray] (0,0) grid (8,6);
			
			\draw[ultra thick, dashed, opacity=.6] (0,-8)--(0,0)-|(8,6);
			\draw (9.4,3) node {$h$} (4,7.5) node {$k$} (9.4,-4) node {$k$};
			\draw[ultra thick, blue, opacity=.4](8,6)--(8,5)--(7,5)--(7,3)--(6,3)--(6,1)--(5,1)--(5,0) (0,-8)|-(0,0);
			
			\filldraw[gray, opacity=.25](5,0)--(8,0)--(8,3)--(7,3)--(7,2)--(6,2)--(6,1)--(5,1);
			\draw(4.5,.5) circle (3pt) (5.5,2.5) circle (3pt) (6.5,4.5) circle (3pt) ;
			\draw [gray](0,-8) grid (8,0) ;
			\filldraw[white](-.2,-8.2)--(8.2,0.2)--(8.2,-8.2)--(-.2,-8.2) ;
			\draw (0,-8)--(8,0);
			\draw [decorate,decoration={brace, mirror,amplitude=15pt},xshift=0pt,yshift=0pt](8,0)--(8,6);
			\draw [decorate,decoration={brace,amplitude=15pt},xshift=0pt,yshift=0pt](0,6)--(8,6);\draw[decorate,decoration={brace,amplitude=15pt},xshift=0pt,yshift=0pt](8,0)--(8,-8);
			\end{tikzpicture}  
		\end{minipage}%
		\begin{minipage}{.5 \textwidth}
			\centering
			\begin{tikzpicture}[scale=.5]
			\draw[gray] (0,0) grid (8,6);
			\draw[ultra thick, dashed, opacity=.6] (0,-8)--(0,0)-|(8,6);
			\draw (9.4,3) node {$h$} (4,7.5) node {$k$} (9.4,-4) node {$k$};
			\draw[ultra thick, blue, opacity=.4](0,-8)|-(1,0);
			\filldraw[gray, opacity=.3];
			\draw [gray](0,-8) grid (8,0) ;
			\draw[ultra thick, blue, opacity=.4](8,6)--(8,5)--(7,5)--(7,4)(6,4)--(6,3)--(5,3) (3,3)|-(2,1)(1,1)--(1,0);
			\draw[ultra thick, red, opacity=.4](7,4)--(6,4)(5,3)--(4,3);
			\draw[ultra thick, red, opacity=.4](1,1)--(2,1)(3,3)--(4,3);
			\filldraw[white](-.2,-8.2)--(8.2,0.2)--(8.2,-8.2)--(-.2,-8.2) ;
			\draw (0,-8)--(8,0);
			\draw(.5,.5) circle (3pt) (2.5,2.5) circle (3pt) (6.5,4.5) circle (3pt) ;
			\draw [decorate,decoration={brace, mirror,amplitude=15pt},xshift=0pt,yshift=0pt](8,0)--(8,6);
			\draw[decorate,decoration={brace,amplitude=15pt},xshift=0pt,yshift=0pt](8,0)--(8,-8);
			\draw [decorate,decoration={brace,amplitude=15pt},xshift=0pt,yshift=0pt](0,6)--(8,6);
			
			\end{tikzpicture} 
		\end{minipage} 
		\caption{Creating $P(D)$, step 1 and 2}\label{creating P(D)}
	\end{figure}
\end{proof}

\medskip

Let $a,b,n,k\in \mathbb N\cup\{0\}$, with $n\geq k\geq 1$. We introduce two new sets  
\begin{equation}
\widetilde{R}_{n,k}^{(a,b)}\subseteq \widetilde{S}_{n,k}^{(a,b)}\subseteq \widetilde{\mathcal D}_{n}^{(a,b)}:
\end{equation}
	\begin{itemize}\item a path in $ \widetilde{\mathcal D}_{n}^{(a,b)}$ is a path in $ \widetilde{S}_{n,k}^{(a,b)}$ if its area word contains exactly $k$ zeroes;   
		\item a path in $\widetilde{S}_{n,k}^{(a,b)}$ is a path in $ \widetilde{R}^{(a,b)}_{n,k}$ if the peaks whose starting point lies on the main diagonal are not decorated. 
	\end{itemize}
Since $\widetilde{\mathcal D}_n^{(a,b)} = \sqcup_{k=1}^n  \widetilde{S}^{(a,b)}_{n,k}$, we have
\begin{equation}
\mathbf{B}\widetilde{\mathcal D}_n^{(a,b)} = \sum_{k=1}^n \mathbf{B}\widetilde{S}^{(a,b)}_{n,k} .
\end{equation}

Using the $\zeta$ map, we can deduce the following results for the polynomials $\mathbf{D}\widetilde{R}_{n,k}^{(a,b)}$ and $\mathbf{D}\widetilde{S}_{n,k}^{(a,b)}$.

The next proposition follows immediately from the construction of $\zeta$.
\begin{proposition} \label{prop:zeta_k}
Let $a,b,n,k\in \mathbb N\cup\{0\}$, with $n\geq k\geq 1$. Then the $\zeta$ map defined in the proof of Theorem~\ref{thm:zeta_map}, suitably restricted, gives a bijection from $\widetilde{S}_{n,k}^{(a,b)}$ to $\widehat{S}_{n,k}^{(b,a)}$ and from $\widetilde{R}_{n,k}^{(a,b)}$ to $\widehat{R}_{n,k}^{(b,a)}$.
\end{proposition}
The following corollary is an immediate consequence of Proposition~\ref{prop:zeta_k} and Theorem~\ref{thm:zeta_map}.
\begin{corollary} \label{cor:zeta_map}
Let $a,b,n,k\in \mathbb N\cup\{0\}$, with $n\geq k\geq 1$. Then
\begin{equation}
\mathbf{D}\widetilde{R}_{n,k}^{(a,b)}=\mathbf{B}\widehat{R}_{n,k}^{(b,a)}\quad \text{ and }
\quad \mathbf{D}\widetilde{S}_{n,k}^{(a,b)}=\mathbf{B}\widehat{S}_{n,k}^{(b,a)}.
\end{equation}
In particular $\mathbf{D}\widetilde{R}_{n,k}^{(a,b)}$ and $\mathbf{D}\widetilde{S}_{n,k}^{(a,b)}$ satisfy the same recursion as $\mathbf{B}\widehat{R}_{n,k}^{(b,a)}$ and $\mathbf{B}\widehat{S}_{n,k}^{(b,a)}$ do in Theorem~\ref{thm:reco_first_bounce_R_S}.
\end{corollary}

We can now write down a recursion for the polynomials $\mathbf B \widehat{S}_{n,k}^{(a,b)} $.
\begin{theorem} \label{thm:reco1_first_bounce}
	Let $a,b,n,k\in \mathbb N\cup\{0\}$, with $n\geq k\geq 1$. Then
	\begin{align} 
	\mathbf B \widehat{S}_{n,n}^{(a,b)}  & = \delta_{a,0} q^{\binom{n-b}{2}}\begin{bmatrix}
	n-1\\
	b
	\end{bmatrix}_q  ,
	\end{align}
	and for $n>k\geq 1$
	\begin{align}
	\mathbf B \widehat{S}_{n,k}^{(a,b)} & = t^{n-k-a} \sum_{s=0}^{\min(k,b)}
	\sum_{i=0}^a    
	\sum_{h=1}^{n-k} 
	q^{{k-s\choose 2}+{i+1\choose 2}} {k\brack s}_q
	{h-1 \brack i}_q 
	{h+k-s-i-1\brack h}_q\\
	\notag	&\quad \times \left(\mathbf{B}\widehat{S}^{(a-i,b-s)}_{n-k,h} +
	\mathbf{B}\widehat{S}^{(a-i-1,b-s)}_{n-k,h} \right),
	\end{align}
	with initial conditions
	\begin{align}
	\mathbf{B}\widehat{S}_{0,k}^{(a,b)}  & =  \mathbf{B}\widehat{S}_{n,0}^{(a,b)}  =0.
	\end{align}
\end{theorem}
\begin{proof}
The initial conditions are easy to check. The recursive step follows immediately by combining Theorem~\ref{thm:reco_first_bounce_S_R} and Theorem~\ref{thm:reco_first_bounce_R_S}.
\end{proof}

The following corollary follows immediately from Theorem~\ref{thm:reco1_first_bounce} and Proposition~\ref{prop:zeta_k}. This recursion for $\mathbf{D}\widetilde{S}_{n,k}^{(a,b)}$ is essentially due to Wilson (cf. \cite[Proposition~5.3.1.1]{wilsonPhD}).
\begin{corollary}[Wilson]
The polynomials $\mathbf{D}\widetilde{S}_{n,k}^{(b,a)}$ satisfy the same recursion as the polynomials $\mathbf{B}\widehat{S}_{n,k}^{(a,b)}$ in Theorem~\ref{thm:reco1_first_bounce}.
\end{corollary}
We can now write down a recursion for the polynomials $\mathbf B \widehat{R}_{n,k}^{(a,b)}  $.
\begin{theorem} \label{thm:reco_R_first_bounce}
Let $a,b,n,k\in \mathbb N\cup\{0\}$, with $n\geq k\geq 1$. Then
\begin{align}
\notag	\mathbf{B}\widehat{R}^{(a,b)}_{n,k}  & =
	t^{n-k-a}
	\sum_{i=0}^a   q^{{k\choose 2}+{i+1\choose 2}}
	\sum_{h=1}^{n-k} \sum_{s=0}^{\min(h,b)} {h \brack s}_q
	{h-1 \brack i}_q 
	{h+k-i-1\brack h}_q\\
	\notag	&\quad \times \left(\mathbf{B}\widehat{R}^{(a-i,b-s)}_{n-k-s,h-s} +\mathbf{B}\widehat{R}^{(a-i-1,b-s)}_{n-k-s,h-s} \right) \\
	 & =
	t^{n-k-a}
	q^{{k\choose 2}+{a\choose 2}}
	{k \brack a}_q 
	{n-a-1\brack n-k-a}_q 	q^{\binom{n-k-b}{2}}{n-k-1 \brack b}_q\\
	& \quad +
	t^{n-k-a}
	\sum_{i=0}^a   q^{{k\choose 2}+{i\choose 2}}
	\sum_{h=1}^{n-k-1} \sum_{s=0}^{\min(h,b)} {h \brack s}_q
	{k \brack i}_q 
	{h+k-i-1\brack h-i}_q \mathbf{B}\widehat{R}^{(a-i,b-s)}_{n-k-s,h-s} 
\end{align}
with initial conditions
\begin{equation}
\mathbf{B}\widehat{R}^{(a,b)}_{n,n} =\delta_{0,a}\delta_{0,b}
q^{n\choose 2}.
\end{equation}
\end{theorem}
\begin{proof}
For the first equality, just combine Theorem~\ref{thm:reco_first_bounce_R_S} with Theorem~\ref{thm:reco_first_bounce_S_R}.

For the second equality, rearrange the terms and use the following elementary lemma, whose proof is in the appendix of this article.
\begin{lemma} \label{lem:elementary4}
	For $h\geq 1$ and $k,s\geq 0$
	\begin{equation}
	q^{s}\begin{bmatrix}
	h-1\\
	s\\
	\end{bmatrix}_q \begin{bmatrix}
	h+k-s-1\\
	h\\
	\end{bmatrix}_q +\begin{bmatrix}
	h-1\\
	s-1\\
	\end{bmatrix}_q \begin{bmatrix}
	h+k-s\\
	h\\
	\end{bmatrix}_q= \begin{bmatrix}
	k\\
	s\\
	\end{bmatrix}_q \begin{bmatrix}
	h+k-s-1\\
	h-s\\
	\end{bmatrix}_q.
	\end{equation}
\end{lemma}
\end{proof}
The following corollary follows immediately from Theorem~\ref{thm:reco_R_first_bounce} and Proposition~\ref{prop:zeta_k}. 
\begin{corollary}
	The polynomials $\mathbf{D}\widetilde{R}_{n,k}^{(b,a)}$ satisfy the same recursion as the polynomials $\mathbf{B}\widehat{R}_{n,k}^{(a,b)}$ in Theorem~\ref{thm:reco_R_first_bounce}.
\end{corollary}

\subsubsection{Second bounce} 

Let $a,b,n,k,i\in \mathbb N\cup\{0\}$, with $n\geq k\geq i\geq 1$. We introduce two sets  
\begin{equation}
\overline{ T}_{n,k,i}^{(a,b)}\subseteq \overline{ T}_{n,k}^{(a,b)}\subseteq \overline{\mathcal D}_n^{(a,b)}:
\end{equation}
	\begin{itemize}\item a path in $ \overline{\mathcal D}_{n}^{(a,b)}$ is a path in $\overline{ T}_{n,k}^{(a,b)}$ if its second bounce word contains exactly $k$ zeroes;   
	\item a path in $\overline{ T}_{n,k}^{(a,b)}$ is a path in $\overline{ T}_{n,k,i}^{(a,b)}$ if among the $k$ first vertical steps of $D$, $i$ are decorated rises. 
\end{itemize}
Observe that $\widetilde{\mathcal D}_{n}^{(a,b)}=\sqcup_{k=1}^{n}\overline{ T}_{n,k}^{(a,b)}$ and $\overline{ T}_{n,k}^{(a,b)}=\sqcup_{i=0}^{\min(b,k-1)}\overline{ T}_{n,k,i}^{(a,b)}$, so that

\begin{align}
\mathbf{B'}\overline{T}^{(a,b)}_{n,k}  & =\sum_{i=0}^{\min(b,k-1)} \mathbf{B'}\overline{T}^{(a,b)}_{n,k,i} ,\quad \text{ and }\\ \mathbf{B'}\overline{\mathcal D}_n^{(a,b)}  & = \sum_{k=1}^n \mathbf{B'}\overline{T}^{(a,b)}_{n,k} = \sum_{k=1}^n\sum_{i=0}^{\min(b,k-1)} \mathbf{B'}\overline{T}^{(a,b)}_{n,k,i} . 
\end{align}

The following proposition is an immediate consequence of the definition of the map $\psi$ in Theorem~\ref{thm:psi_map}.
\begin{proposition} \label{prop:psi_k}
The map $\psi$ defined in the proof of Theorem~\ref{thm:psi_map} restricts to a bijection between $\overline{T}^{(a,b)}_{n,k}$ and $\overline{T}^{(b,a)}_{n,k}$.
\end{proposition}
The next corollary follows immediately from Proposition~\ref{prop:psi_k} and Theorem~\ref{thm:psi_map}
\begin{corollary} \label{cor:comb_symmetry}
Let $a,b,n,k\in \mathbb N\cup\{0\}$, with $n\geq k \geq 1$. We have
\begin{equation}
\mathbf{B'}\overline{T}^{(a,b)}_{n,k} = \mathbf{B'}\overline{T}^{(b,a)}_{n,k} ,
\end{equation}
so that
\begin{equation}
\mathbf{B'}\overline{D}^{(a,b)}_{n} = \mathbf{B'}\overline{D}^{(b,a)}_{n} .
\end{equation}
\end{corollary}

The following theorem provides a recursion for the polynomials $\mathbf{B'}\overline{T}_{n,k,i}^{(a,b)}$.

\begin{theorem} \label{thm:reco_old_bounce}
	Let $a,b,n,k,i\in \mathbb N$ with $n\geq 1$ and $n\geq k\geq 1$  then 
	\begin{equation*} 
	\mathbf{B'}\overline{T}_{n,n,i}^{(a,b)}=\delta_{i,b} q^{n-a-b \choose 2}{n-b-1\brack a}_q{n-1\brack b}_q,
	\end{equation*}
		and for $n>k\geq 1$
	\begin{equation*}
		\mathbf{B'}\overline{T}_{n,k,i}^{(a,b)} =t^{n-k} \sum_{p=1}^k\sum_{j=0}^{n-k}\sum_{f=0}^{\min(b-i,j)}  q^{p-i \choose 2}{k-i\brack p-i}_q{k-1\brack i}_q{p-1+j-f\brack j-f}_q \mathbf{B'} \overline{T}_{n-k,j,f}^{(a-k+p,b-i)} 
	\end{equation*}

	with the initial conditions \begin{align*}
		\mathbf{B'}\overline{T}_{n,0,i}^{(a,b)} &= \mathbf{B'}\overline{T}_{0,k,i}^{(a,b)}= 0.
	\end{align*} 
	\end{theorem}
\begin{proof}
	In this proof, when we refer to a \emph{peak}, we mean the vertical step \emph{and} the horizontal step that follows it, as this simplifies the argument.
	
	We describe a procedure to obtain a unique path in $\overline{T}_{n,k,i}^{(a,b)}$, keeping track of the statistics $\bounce'$ and $\area$. See Figure~\ref{fig:second_bounce_recursion} for a visual aid.
	
	We start with the case $k<n$. Consider a path $D$ in $\overline{T}_{n-k,j,f}^{(a-k+p,b-i)}$. Call $P(D)$ the portion of $D$ that starts at its lower left corner and ends at the endpoint of its $j$-th vertical step. Since the second bounce word of $D$ contains $j$ zeroes, $P(D)$ consists of vertical steps and decorated peaks and is followed by a horizontal step.  Furthermore, $P(D)$ has $f$ decorated rises. We extend $D$ to a path of size $n$. Place $D$ in the top right corner of an $n\times n$ grid.
	\begin{enumerate}
		\item Let $p$ be an integer between $1$ and $k$. We add $p$ horizontal steps to $P(D)$ as follows. Start by deleting the $f$ vertical steps of $P(D)$ that follow a decorated rise, but keep the decorations. Next, prepend one horizontal step to the path. Then choose an interlacing between the $j-f$ vertical steps of the path and the $p-1$ remaining horizontal steps. Each time a vertical step precedes a horizontal step the area under $P(D)$ and above the main diagonal and the line $y=k$ goes up by 1. So the terms of 
		$${j-f+p-1\brack j-f}_q$$ represent the added area in this zone. Conclude by putting back the $f$ vertical steps after the decorated rises. This does not modify the area as the squares in the rows following decorated rises do not contribute to the area. The resulting $P(D)$ starts at $(k-p,k)$ and ends at the same point as before. 
		\item Next, we construct a path from $(0,0)$ to $(k-p,p-i+k-p)=(k-p,k-i)$. We do this by choosing a sequence of $k-p$ decorated peaks (recall what peak means here) and $p-i$ vertical steps. The squares east of this path and west of the main diagonal contributing to the area can be seen as of two kinds: those east and those west of the line $x=k-p$. Clearly there are $p-i\choose 2$ squares of the first kind. The number of squares of the second kind depends on the interlacing of the peaks and vertical steps: every time that a vertical step precedes a peak, a square of area is created. This explains the factor $${k-p+p-i\brack p-i}_q={k-i\brack p-i}_q.$$
		\item Finally, we add $i$ decorated rises. The path from $(0,0)$ to $(k-p,k-i)$ that we constructed contains $k-p$ peaks and $p-i$ vertical steps. The highest of them cannot be a decorated rise, since it has to be followed by a horizontal step (i.e. the first one of $P(D)$). So we choose an interlacing between $k-i-1$ vertical steps of the existing path and $i$ decorated rises. The decorated rises are prepended to the existing vertical step that follows in the interlacing. This operation does not add area west of the line $x=k-p$ but does add on square of area east of this line every time a decorated rise precedes a vertical step that is not a decorated rise. Thus, we have a term $${k-i-1+i\brack i}_q={k-1\brack i}_q.$$
	\end{enumerate}
	Depending on the choice of the mentioned interlacings, we obtain a unique path $D'$ of size $n$. Its second bounce word contains $k$ zeroes (indeed, the first horizontal step that is not part of a decorated peak occurs after the $k$-th row), and $i$ of the first $k$ vertical steps are decorated rises. So $D'\in \overline{T}_{n,k,i}^{(a,b)}$. We already kept track of the area during the construction and it is not difficult to see that $$\bounce'(D')=\bounce'(D)+n-k.$$ Indeed the second bounce word of $D'$ is $k$ zeroes followed by the second bounce word of $D$ whose letters are all augmented by 1. 
	
	Summing over the possible values of $p$, $f$ and $j$, we obtain the right equation. 
	
	Now consider the case when $k=n$. Since there is only one ``bounce", it is clear that $i=b$, hence the factor $\delta_{i,b}$. The rest of the reasoning is quite similar to the previous one except for step (2). Indeed since $k=n$ the last step in this interlacing may not be a decorated peak. So we fix one undecorated vertical step at the end and obtain a term $${k-p+p-i-1\brack p-i-1}_q={k-i-1 \brack p-i-1}_q.$$ The result easily follows.
	
	The initial conditions of this recursion are easy to check.

	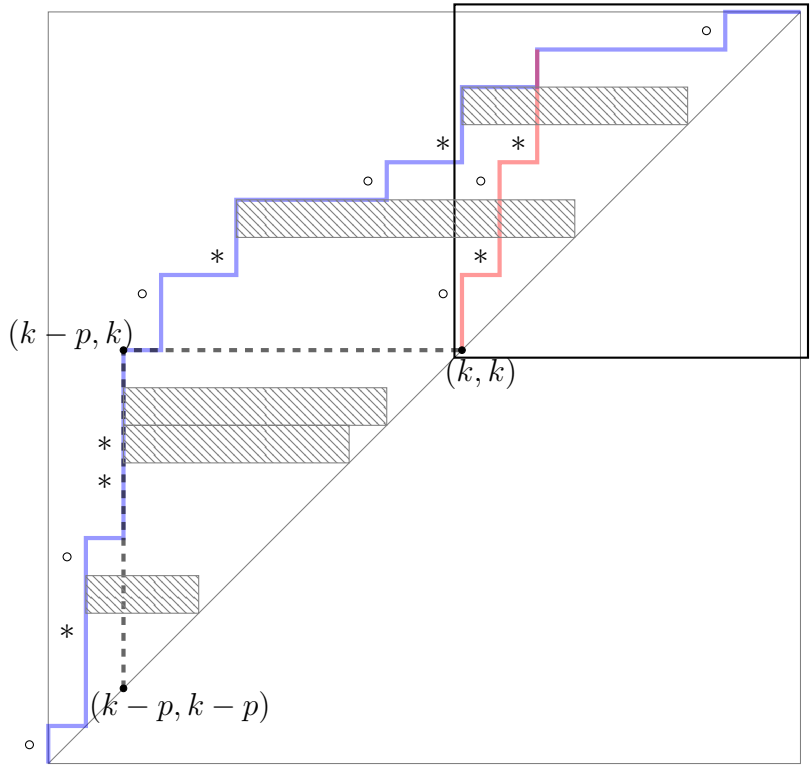
\begin{figure}[h!]
		\centering
		\begin{tikzpicture}[scale=.5]
		\draw[gray](0,0)--(20,20) (0,0) rectangle (20,20);
		\draw[ultra thick, blue, opacity=.4](0,0)|-(1,1)--(1,6)--(2,6)|-(3,11)--(3,13)--(5,13)--(5,15)--(9,15)--(9,16)--(11,16)|-(13,18)|-(18,19)|-(20,20);
		\draw (-.5,.5) circle (3pt)(.5,5.5) circle (3pt) (2.5,12.5) circle (3pt) (8.5,15.5) circle (3pt) (17.5,19.5) circle (3pt);
		\draw (.5,3.5) node {$\ast$} (1.5,7.5) node {$\ast$} (1.5,8.5) node {$\ast$} (4.5,13.5) node {$\ast$} (11.5,13.5) node {$\ast$} (12.5,16.5) node {$\ast$} (10.5,16.5) node {$\ast$} ;
		\draw[ultra thick, dashed,opacity=.6](2,2)|-(11,11); 
		\draw[ultra thick, red, opacity=.4](11,11)--(11,13)--(12,13)--(12,16)--(13,16)--(13,19);
		\draw (10.5,12.5) circle (3pt) (11.5,15.5) circle (3pt);
		\draw[thick] (10.8,10.8) rectangle (20.2,20.2);
		\filldraw[gray,pattern=north west lines, pattern color=gray] (1,4) rectangle (4,5) (2,8) rectangle (8,9) (2,9) rectangle (9,10) (5,14) rectangle (14,15) (11,17) rectangle (17,18);
		\filldraw (2,11) circle (2.5pt) (2,2) circle (2.5 pt) (11,11) circle (2.5pt);
		\draw (.6,11.4) node {$(k-p,k)$} (3.5,1.5) node {$(k-p,k-p)$} (11.5,10.4) node{$(k,k)$} ;
		\end{tikzpicture}
		\caption{Second bounce recursion}\label{fig:second_bounce_recursion}
	\end{figure}
\end{proof}

\begin{remark} \label{rmk:no_restrictions}
We define new sets of decorated Dyck paths
\begin{equation}
R^{(a,b)}_{n,k}\subseteq S^{(a,b)}_{n,k} \subseteq \mathcal{D}_{n}^{(a,b)}
\end{equation}
by setting
\begin{equation}
R^{(a,b)}_{n,k}:=\widehat{R}^{(a,b)}_{n,k}\sqcup \widehat{R}^{(a,b-1)}_{n,k},\quad \text{ and }\quad  S^{(a,b)}_{n,k}:=\widehat{S}^{(a,b)}_{n,k}\sqcup \widehat{S}^{(a,b-1)}_{n,k}.
\end{equation}
We can interpret $S^{(a,b)}_{n,k}$ as the set of decorated Dyck paths of size $n$, $k$ $0$'s in the first bounce word, $b$ decorated falls and $a$ decorated peaks where we allow the decoration of the last horizontal step as it were a fall (of course this decoration does not touch the statistics). So $\widehat{S}^{(a,b-1)}_{n,k}$ can be identified with those where the last horizontal step gets decorated. Similarly for $R^{(a,b)}_{n,k}$.

In a similar way we can define 
\begin{equation}
{R'}^{(a,b)}_{n,k} \subseteq {S'}^{(a,b)}_{n,k}  \subseteq \mathcal{D}_{n}^{(a,b)}
\end{equation}
by setting
\begin{equation}
{R'}^{(a,b)}_{n,k} :=\widetilde{R}^{(a,b)}_{n,k}\sqcup \widetilde{R}^{(a-1,b)}_{n,k},\quad \text{ and }\quad  {S'}^{(a,b)}_{n,k} :=\widetilde{S}^{(a,b)}_{n,k}\sqcup \widetilde{S}^{(a-1,b)}_{n,k}.
\end{equation}
Similarly, we can interpret ${S'}^{(a,b)}_{n,k}$ as the set of decorated Dyck paths of size $n$, $k$ $0$'s in the area word, $b$ decorated rises, $a$ decorated peaks, where we allow also the rightmost highest peak to be decorated (of course this decoration does not touch the statistics). So $\widetilde{S}^{(a,b-1)}_{n,k}$ can be identified with those where the rightmost highest peak gets decorated. Similarly for ${R'}^{(a,b)}_{n,k}$

Finally, we can do the same with 
\begin{equation}
{T}^{(a,b)}_{n,k,i} \subseteq {T}^{(a,b)}_{n,k}  \subseteq \mathcal{D}_{n}^{(a,b)}
\end{equation}
by setting
\begin{equation}
{T}^{(a,b)}_{n,k,i} :=\overline{T}^{(a,b)}_{n,k,i}\sqcup \overline{T}^{(a-1,b)}_{n,k,i},\quad \text{ and }\quad  {T}^{(a,b)}_{n,k} :=\overline{T}^{(a,b)}_{n,k}\sqcup \overline{T}^{(a-1,b)}_{n,k}.
\end{equation}
\end{remark}

\section{A summation formula}

This section is devoted to prove the following theorem.
\begin{theorem}
For $m,k\geq 1$ and $\ell\geq 0$, we have
\begin{equation} \label{eq:mastereq}
 \sum_{\gamma\vdash m}\frac{\widetilde{H}_\gamma[X]}{w_\gamma} h_k[(1-t)B_\gamma]e_\ell[B_\gamma]=\qquad \qquad \qquad \qquad  \qquad \qquad  \qquad \qquad 
\end{equation}
\begin{align*} 
 =\sum_{j=0}^{\ell} t^{\ell-j}\sum_{s=0}^{k}q^{\binom{s}{2}} \begin{bmatrix}
s+j\\
s
\end{bmatrix}_q \begin{bmatrix}
k+j-1\\
s+j-1
\end{bmatrix}_qh_{s+j}\left[\frac{X}{1-q}\right] h_{\ell-j}\left[\frac{X}{M}\right] e_{m-s-\ell}\left[\frac{X}{M}\right].
\end{align*}
\end{theorem}

\subsection{Two special cases}

We start by establishing two special cases of this formula: the cases $\ell=0$ and $\ell=m$.

The case $\ell=0$ is implicit in the work of Haglund (e.g. see equation (2.38) in \cite{haglundschroeder}). Though we don't need this special case in our argument, we outline here its proof (referring to \cite{haglundschroeder} for some details) since it shows the strategy that we will use also in the general case. 

\begin{theorem}[Haglund] \label{thm:Haglund_summation}
For $m,k\geq 1$ we have
\begin{equation} \label{eq:Haglund_summation}
	\sum_{\mu\vdash m} \frac{\widetilde{H}_\mu[X]}{w_\mu}h_k[(1-t)B_\mu] = \sum_{s=1}^{m}
	q^{\binom{s}{2}} 
	\begin{bmatrix}
	k-1\\
	s-1\\
	\end{bmatrix}_q h_s \left[ \frac{X}{1-q}\right] e_{m-s}\left[ \frac{X}{M}\right].
\end{equation}	
\end{theorem}
\begin{proof}
Let us set
\begin{equation}
g[X]:=\sum_{\mu\vdash m} \frac{ \widetilde{H}_\mu[X]}{w_\mu}h_k[(1-t)B_{\mu}].
\end{equation}
The idea is to use the expansion 
\begin{equation}
\sum_{\mu\vdash m} \frac{ \widetilde{H}_\mu[X]}{w_\mu}h_k[(1-t)B_{\mu}]=g[X]= \sum_{\mu\vdash m} \frac{ \widetilde{H}_\mu[X]}{w_\mu} \langle g[X],\widetilde{H}_\mu[X]\rangle_*.
\end{equation}
We suppose that $g[X]$ is of the form 
$$
g[X]=\sum_{r\geq 0}e_r^* f[X]
$$
for some $f[X]\in \Lambda$, so that
\begin{align*}
\langle g[X],\widetilde{H}_\mu[X]\rangle_* & = \langle \sum_{r\geq 0}e_r^*f[X],\widetilde{H}_\mu[X]\rangle_*\\ 
\text{(using \eqref{eq:hperp_estar_adjoint})}& = \langle f[X], \sum_{r\geq 0}h_r^\perp \widetilde{H}_\mu[X]\rangle_*\\ 
& = \langle f[X], \widetilde{H}_\mu[X+1]\rangle_*.
\end{align*}
Then we can use the identity \eqref{eq:glenn_formula} to get
\begin{equation}
\langle g[X],\widetilde{H}_\mu[X]\rangle_* = \langle f[X], \widetilde{H}_\mu[X+1]\rangle_*=  \left.\nabla^{-1}\tau_{-\epsilon} f[X]\right|_{X=D_\mu}.
\end{equation}
So we look for an $f[X]$ such that
\begin{equation}
\left.\nabla^{-1}\tau_{-\epsilon}f[X]\right|_{X=D_\mu}=h_k[(1-t)B_{\mu}].
\end{equation}
Clearly, we can use
\begin{equation}
\nabla^{-1}\tau_{-\epsilon}f[X]  =  h_{k}\left[\frac{X+1}{1-q}\right],
\end{equation}
so that
\begin{equation}
f[X]  =  \tau_{\epsilon}\nabla h_{k}\left[\frac{X+1}{1-q}\right].
\end{equation}
This is precisely equation (2.76) in \cite{haglundschroeder}. 

Now notice that
\begin{equation}
\sum_{\mu\vdash m} \frac{\widetilde{H}_\mu[X]}{w_\mu}\langle g[X],\widetilde{H}_\mu[X]\rangle_*=(g[X])_m,
\end{equation}
where $(g[X])_d$ denotes the homogeneous component of $g[X]$ of degree $d$.

Equation (2.83) in \cite{haglundschroeder} shows that for all $d\geq 0$
\begin{equation}
(f[X])_d  = h_{d}\left[\frac{X}{1-q}\right]q^{\binom{d}{2}} \begin{bmatrix}
k-1\\
k-d\\
\end{bmatrix}_q,
\end{equation}
so that
\begin{align*}
\sum_{\mu\vdash m} \frac{ \widetilde{H}_\mu[X]}{w_\mu}h_k[(1-t)B_{\mu}] & = \sum_{\mu\vdash m} \frac{ \widetilde{H}_\mu[X]}{w_\mu} \langle g[X],\widetilde{H}_\mu[X]\rangle_*\\
& =(g[X])_m\\
& =\sum_{d=0}^{m}e_{m-d}\left[\frac{X}{M}\right](f[X])_d\\
& =\sum_{d=0}^{m}e_{m-d}\left[\frac{X}{M}\right]h_{d}\left[\frac{X}{1-q}\right]q^{\binom{d}{2}} \begin{bmatrix}
k-1\\
k-d\\
\end{bmatrix}_q\\
& =\sum_{d=0}^{m}q^{\binom{d}{2}} \begin{bmatrix}
k-1\\
d-1\\
\end{bmatrix}_qh_{d}\left[\frac{X}{1-q}\right]e_{m-d}\left[\frac{X}{M}\right],
\end{align*}
as we wanted.
\end{proof}

We now prove the case $\ell=m$ (recall \eqref{eq:Bmu_Tmu}), as we need it to establish the general case.
\begin{theorem}
For $m\geq 1$ and $k\geq 1$ we have
\begin{align} \label{eq:basic_summation_2}
\sum_{\mu\vdash m} \frac{T_\mu \widetilde{H}_\mu[X]}{w_\mu}h_k[(1-t)B_\mu]  & =  \sum_{j=1}^m \begin{bmatrix}
k+j-1\\
j-1
\end{bmatrix}_q h_j\left[\frac{X}{1-q}\right] h_{m-j}\left[\frac{tX}{M}\right] .
\end{align}	
\end{theorem}
\begin{proof}
Using \eqref{eq:en_q_sum_Enk}, we have
\begin{equation}
\nabla e_m\left[X\frac{1-q^k}{1-q}\right]=\sum_{j=1}^m \begin{bmatrix}
k+j-1\\
j
\end{bmatrix}_q\nabla E_{mj} .
\end{equation}
Also, Haglund showed in equation (7.86) of \cite{haglundbook} that
\begin{equation}
\nabla E_{mj}=t^{m-j}(1-q^j)\mathbf{\Pi} \left(h_{m-j}\left[\frac{X}{M}\right] h_{j}\left[\frac{X}{1-q}\right]\right),
\end{equation}
where $\mathbf{\Pi}$ is the linear operator defined by
\begin{equation}
\mathbf{\Pi} \widetilde{H}_\mu[X]=\Pi_\mu \widetilde{H}_\mu[X]\qquad \text{for all }\mu.
\end{equation}
So we get
\begin{align} \label{eq:enq_aux1}
\notag \nabla e_n\left[X\frac{1-q^k}{1-q}\right] & =\sum_{k=1}^n \begin{bmatrix}
k+j-1\\
j
\end{bmatrix}_qt^{m-j}(1-q^j)\mathbf{\Pi} \left(h_{m-j}\left[\frac{X}{M}\right] h_{j}\left[\frac{X}{1-q}\right]\right)\\ 
& =(1-q^k)\sum_{k=1}^n \begin{bmatrix}
k+j-1\\
j-1
\end{bmatrix}_qt^{m-j}\mathbf{\Pi} \left(h_{m-j}\left[\frac{X}{M}\right] h_{j}\left[\frac{X}{1-q}\right]\right)  ,
\end{align}
where we used the obvious
\begin{equation}
(1-q^j)\begin{bmatrix}
k+j-1\\
j
\end{bmatrix}_q=(1-q^k)\begin{bmatrix}
k+j-1\\
j-1
\end{bmatrix}_q.
\end{equation}

But, using \eqref{eq:qn_q_Macexp}, we also have
\begin{align} \label{eq:enq_aux2}
\notag \nabla e_m\left[X\frac{1-q^k}{1-q}\right] & =(1-q^k)\sum_{\mu\vdash m}\frac{\Pi_\mu T_\mu \widetilde{H}_\mu[X]h_k[(1-t)B_\mu]}{w_\mu}\\
& =(1-q^k)\mathbf{\Pi}\sum_{\mu\vdash m}\frac{T_\mu \widetilde{H}_\mu[X]h_k[(1-t)B_\mu]}{w_\mu}.
\end{align}
Comparing \eqref{eq:enq_aux1} and \eqref{eq:enq_aux2}, and noticing that the operator $\mathbf{\Pi}$ is invertible, we get the result.
\end{proof}

\subsection{A key identity}

Now we want to prove the following key identity.
\begin{proposition} \label{prop:key_identity}
For $j\geq 0$ and $i\geq 1$ we have
\begin{equation} \label{eq:key_identity}
\nabla \left( h_i\left[\frac{X}{1-q}\right] e_j\left[\frac{X}{M}\right] \right) = 
	\sum_{s =1}^{i+j} t^{i+j-s}q^{\binom{i}{2}} \begin{bmatrix}
	s\\
	i
	\end{bmatrix}_q  h_s\left[\frac{X}{1-q}\right] h_{i+j-s}\left[\frac{X}{M}\right].
\end{equation}
\end{proposition}

In order to prove this proposition, we need a proposition and a lemma.

The following proposition is \cite[Proposition~2.6]{garsiahicksstout}.
\begin{proposition} 
For $i\geq 1$ and $j\geq 0$ we have
\begin{equation} \label{eq:GHS_2_6}
h_i\left[\frac{X}{1-q}\right] e_j\left[\frac{X}{M}\right]  = \sum_{\mu\vdash i+j} \frac{\widetilde{H}_\mu[X]}{w_\mu} \sum_{r=1}^i \begin{bmatrix}
	i-1\\
	r-1
\end{bmatrix}_q q^{\binom{r}{2}+r-ir}(-1)^{i-r}h_r[(1-t)B_\mu].
\end{equation}
\end{proposition}

The following elementary lemma is proved in the appendix of the present article.
\begin{lemma} \label{lem:elementary1}
For $a\geq 0$, $i\geq 1$ and $s\geq 0$ we have
\begin{equation} \label{eq:first_qlemma}
\sum_{r=1}^{i}\begin{bmatrix}
i-1\\
r-1
\end{bmatrix}_q \begin{bmatrix}
r+s+a-1\\
s-1
\end{bmatrix}_q q^{\binom{r}{2}+r-ir}(-1)^{i-r} =q^{\binom{i}{2}+(i-1)a} \begin{bmatrix}
s+a\\
i+a
\end{bmatrix}_q .
\end{equation}
\end{lemma}

\begin{proof}[Proof of Proposition~\ref{prop:key_identity}]
We have
\begin{equation*}
\nabla \left( h_i\left[\frac{X}{1-q}\right] e_j\left[\frac{X}{M}\right]\right)= \qquad\qquad\qquad\qquad \qquad\qquad
\end{equation*}
\begin{align*}
\text{(using \eqref{eq:GHS_2_6})}	& = \nabla \sum_{\mu\vdash i+j} \frac{\widetilde{H}_\mu[X]}{w_\mu} \sum_{r=1}^i \begin{bmatrix}
	i-1\\
	r-1
	\end{bmatrix}_q q^{\binom{r}{2}+r-ir}(-1)^{i-r}h_r[(1-t)B_\mu]\\
	& = \sum_{r=1}^i \begin{bmatrix}
	i-1\\
	r-1
	\end{bmatrix}_q q^{\binom{r}{2}+r-ir}(-1)^{i-r} \sum_{\mu\vdash i+j} \frac{T_\mu \widetilde{H}_\mu[X]}{w_\mu} h_r[(1-t)B_\mu]\\
\text{(using \eqref{eq:basic_summation_2})}	& = \sum_{r=1}^i \begin{bmatrix}
	i-1\\
	r-1
	\end{bmatrix}_q q^{\binom{r}{2}+r-ir}(-1)^{i-r}   \sum_{s =1}^{i+j} \begin{bmatrix}
	r+s-1\\
	s-1
	\end{bmatrix}_q h_s\left[\frac{X}{1-q}\right] h_{i+j-s}\left[\frac{tX}{M}\right]  \\
	& = \sum_{s =1}^{i+j} \left(\sum_{r=1}^i \begin{bmatrix}
	i-1\\
	r-1
	\end{bmatrix}_q q^{\binom{r}{2}+r-ir}(-1)^{i-r}   \begin{bmatrix}
	r+s-1\\
	s-1
	\end{bmatrix}_q \right) h_s\left[\frac{X}{1-q}\right] h_{i+j-s}\left[\frac{tX}{M}\right]  \\
	\text{(using \eqref{eq:first_qlemma})}	& = \sum_{s =1}^{i+j} q^{\binom{i}{2}} \begin{bmatrix}
	s\\
	i
	\end{bmatrix}_q  h_s\left[\frac{X}{1-q}\right] h_{i+j-s}\left[\frac{tX}{M}\right] \\
	& = \sum_{s =1}^{i+j} t^{i+j-s}q^{\binom{i}{2}} \begin{bmatrix}
	s\\
	i
	\end{bmatrix}_q  h_s\left[\frac{X}{1-q}\right] h_{i+j-s}\left[\frac{X}{M}\right]  .
\end{align*}
\end{proof}

\subsection{The main argument}

The main argument follows the main strategy of the proof of Theorem~\ref{thm:Haglund_summation}: we look for an $f[X]\in \Lambda$ such that
\begin{equation}
\left.\nabla^{-1}\tau_{-\epsilon}f[X]\right|_{X=D_\mu}=h_k[(1-t)B_{\mu}]e_a[B_\mu],
\end{equation}
so that
\begin{equation}
\sum_{\mu\vdash m} \frac{ \widetilde{H}_\mu[MB_\gamma]}{w_\mu}h_k[(1-t)B_{\mu}]e_{\ell}[B_\mu]=\sum_{d=0}^{m}e_{m-d}\left[\frac{X}{M}\right](f[X])_d.
\end{equation}

Clearly, we can use
\begin{align*}
\nabla^{-1} \tau_{-\epsilon} f[X] & =  h_{k}\left[\frac{X+1}{1-q}\right] e_{\ell}\left[\frac{X+1}{M}\right]\\
\text{(using \eqref{eq:e_h_sum_alphabets})}& = \sum_{i=0}^{k} \sum_{j=0}^{\ell}h_{k-i}\left[\frac{1}{1-q}\right]e_{\ell-j}\left[\frac{1}{M}\right]h_{i}\left[\frac{X}{1-q}\right] e_{j}\left[\frac{X}{M}\right]
\end{align*}
so
\begin{align*}
\tau_{-\epsilon}f[X] & = \sum_{i=0}^{k} \sum_{j=0}^{\ell}h_{k-i}\left[\frac{1}{1-q}\right]e_{\ell-j}\left[\frac{1}{M}\right]\nabla \left( h_{i}\left[\frac{X}{1-q}\right] e_{j}\left[\frac{X}{M}\right]\right)\\
\text{(using \eqref{eq:key_identity})}& =  \sum_{j=0}^{\ell}h_{k}\left[\frac{1}{1-q}\right]e_{\ell-j}\left[\frac{1}{M}\right]
h_{j}\left[\frac{X}{M}\right]+\\
& + \sum_{i=1}^{k} \sum_{j=0}^{\ell}h_{k-i}\left[\frac{1}{1-q}\right]e_{\ell-j}\left[\frac{1}{M}\right]
\sum_{s =1}^{i+j}t^{i+j-s} q^{\binom{i}{2}} \begin{bmatrix}
s\\
i
\end{bmatrix}_q  h_s\left[\frac{X}{1-q}\right] h_{i+j-s}\left[\frac{X}{M}\right]\\
\text{(using \eqref{eq:e_h_sum_alphabets})} & =  h_{k}\left[\frac{1}{1-q}\right]  h_{\ell}\left[\frac{X-\epsilon}{M}\right]+\\
& + \sum_{i=1}^{k} \sum_{j=0}^{\ell}h_{k-i}\left[\frac{1}{1-q}\right]e_{\ell-j}\left[\frac{1}{M}\right]
\sum_{s =1}^{i+j}t^{i+j-s} q^{\binom{i}{2}} \begin{bmatrix}
s\\
i
\end{bmatrix}_q  h_s\left[\frac{X}{1-q}\right] h_{i+j-s}\left[\frac{X}{M}\right]
\end{align*}
and therefore
\begin{align*}
f[X]
& =  h_{k}\left[\frac{1}{1-q}\right]  h_{\ell}\left[\frac{X+\epsilon -\epsilon}{M}\right]+\\
& + \sum_{i=1}^{k} \sum_{j=0}^{\ell}h_{k-i}\left[\frac{1}{1-q}\right]e_{\ell-j}\left[\frac{1}{M}\right]
\sum_{s =1}^{i+j}t^{i+j-s} q^{\binom{i}{2}} \begin{bmatrix}
s\\
i
\end{bmatrix}_q  h_s\left[\frac{X+\epsilon}{1-q}\right] h_{i+j-s}\left[\frac{X+\epsilon}{M}\right]\\
& =  h_{k}\left[\frac{1}{1-q}\right]  h_{\ell}\left[\frac{X}{M}\right]+ \sum_{i=1}^{k} \sum_{j=0}^{\ell}h_{k-i}\left[\frac{1}{1-q}\right]e_{\ell-j}\left[\frac{1}{M}\right]
\sum_{s =1}^{i+j}t^{i+j-s} q^{\binom{i}{2}} \begin{bmatrix}
s\\
i
\end{bmatrix}_q\\
& \qquad \times \sum_{b=0}^s(-1)^{s-b} h_{s-b}\left[\frac{1}{1-q}\right] \sum_{c=0}^{i+j-s}(-1)^{i+j-s-c} h_{i+j-s-c}\left[\frac{1}{M}\right] h_b\left[\frac{X}{1-q}\right] h_{c}\left[\frac{X}{M}\right],
\end{align*}
where in the last equality we used \eqref{eq:e_h_sum_alphabets}.

From this formula, it is easy to extract the homogeneous component of $f[X]$ of degree $d\geq 0$:
\begin{align*}
(f[X])_d & = \chi(d=\ell) h_{k}\left[\frac{1}{1-q}\right]  h_{\ell}\left[\frac{X}{M}\right]\\
& + \sum_{b=0}^{\min(k+\ell,d)} 
\sum_{i=\max(1,d-\ell)}^{k} 
\sum_{j=\max(0,d-i)}^{\ell}
\sum_{s =\max(1,b)}^{i+j-d+b} 
h_{k-i}\left[\frac{1}{1-q}\right]e_{\ell-j}\left[\frac{1}{M}\right]\\
& \qquad \times t^{i+j-s} q^{\binom{i}{2}} \begin{bmatrix}
s\\
i
\end{bmatrix}_q h_{s-b}\left[\frac{1}{1-q}\right] (-1)^{i+j-d} h_{i+j-s-d+b}\left[\frac{1}{M}\right] h_b\left[\frac{X}{1-q}\right] h_{d-b}\left[\frac{X}{M}\right],
\end{align*}
where $\chi(\mathcal{P})=1$ if the statement $\mathcal{P}$ is true, while $\chi(\mathcal{P})=0$ if $\mathcal{P}$ is false.

Finally
\begin{equation*}
\sum_{d=0}^{\min(m,k+\ell)}e_{m-d}\left[\frac{X}{M}\right](f[X])_d=\qquad \qquad \qquad \qquad \qquad \qquad \qquad \qquad \qquad \qquad \qquad \qquad 
\end{equation*}
\begin{align}
& = e_{m-\ell}\left[\frac{X}{M}\right]h_k\left[\frac{1}{1-q}\right] h_{\ell}\left[\frac{X}{M}\right]+\\
\notag & + \sum_{d=0}^{\min(m,k+\ell)}\sum_{b=0}^{\min(k+\ell,d)} 
\sum_{i=\max(0,d-\ell)}^{k} 
\sum_{j=\max(0,d-i)}^{\ell}
\sum_{s =\max(1,b)}^{i+j-d+b} 
h_{k-i}\left[\frac{1}{1-q}\right]e_{\ell-j}\left[\frac{1}{M}\right]\\
\notag & \qquad \times t^{i+j-s} q^{\binom{i}{2}} \begin{bmatrix}
s\\
i
\end{bmatrix}_q h_{s-b}\left[\frac{1}{1-q}\right] (-1)^{i+j-d} h_{i+j-s-d+b}\left[\frac{1}{M}\right] \\
\notag & \qquad \times e_{m-d}\left[\frac{X}{M}\right]h_b\left[\frac{X}{1-q}\right] h_{d-b}\left[\frac{X}{M}\right].
\end{align}
If we set
\begin{align*}
g(\ell,b,d,k)& := 
\sum_{i=\max(1,d-\ell)}^{k} 
\sum_{j=\max(0,d-i)}^{\ell}
\sum_{s =\max(1,b)}^{i+j-d+b} 
h_{k-i}\left[\frac{1}{1-q}\right]e_{\ell-j}\left[\frac{1}{M}\right]\\
& \qquad \times t^{i+j-s} q^{\binom{i}{2}} \begin{bmatrix}
s\\
i
\end{bmatrix}_q h_{s-b}\left[\frac{1}{1-q}\right] (-1)^{i+j-d} h_{i+j-s-d+b}\left[\frac{1}{M}\right],
\end{align*}
then we proved the identity
\begin{equation} \label{eq:crucial_sum_lemma}
\sum_{\mu\vdash m} \frac{ \widetilde{H}_\mu[X]}{w_\mu}h_k[(1-t)B_{\mu}]e_\ell[B_\mu]=\qquad\qquad\qquad\qquad\qquad \qquad\qquad\qquad\qquad
\end{equation}
\begin{align*}
& = h_k\left[\frac{1}{1-q}\right] e_{m-\ell}\left[\frac{X}{M}\right] h_{\ell}\left[\frac{X}{M}\right]+\\
& + \sum_{d=0}^{\min(m,k+\ell)}\sum_{b=0}^{\min(k+\ell,d)} g(\ell,b,d,k) e_{m-d}\left[\frac{X}{M}\right]h_b\left[\frac{X}{1-q}\right] h_{d-b}\left[\frac{X}{M}\right] .
\end{align*} 
Now we can rewrite this as
\begin{align*}
&  h_k\left[\frac{1}{1-q}\right] e_{m-\ell}\left[\frac{X}{M}\right] h_{\ell}\left[\frac{X}{M}\right]+\\
& + \sum_{d=0}^{\min(m,k+\ell)}\sum_{b=0}^{\min(k+\ell,d)} g(\ell,b,d,k) e_{m-d}\left[\frac{X}{M}\right]h_b\left[\frac{X}{1-q}\right] h_{d-b}\left[\frac{X}{M}\right]\\
& =  h_k\left[\frac{1}{1-q}\right] e_{m-\ell}\left[\frac{X}{M}\right] h_{\ell}\left[\frac{X}{M}\right]+\\
& + \sum_{s=0}^{\min(m-\ell,k)}\sum_{j=0}^{\min(k+\ell-s,\ell)} g(\ell,s+j,\ell+s,k) e_{m-s-\ell}\left[\frac{X}{M}\right]h_{s+j}\left[\frac{X}{1-q}\right] h_{\ell-j}\left[\frac{X}{M}\right] .
\end{align*}

\subsection{Simplifying coefficients}

Now we want to simplify the coefficients. 

We have
\begin{equation}
g(\ell,s+j,\ell+s,k)= \qquad \qquad \qquad \qquad \qquad \qquad \qquad \qquad \qquad \qquad 
\end{equation}
\begin{align*}
\qquad  & =\sum_{\tilde{i}=\max(1,s)}^{k} \sum_{\tilde{j}=\max(0,\ell+s-\tilde{i})}^{\ell} \sum_{\tilde{s}=\max(1,s+j)}^{\tilde{i}+\tilde{j}-\ell+j}h_{k-\tilde{i}} \left[\frac{1}{1-q}\right]e_{\ell-\tilde{j}} \left[\frac{1}{M}\right]\\
& \qquad \times t^{\tilde{i}+\tilde{j}-\tilde{s}} q^{\binom{\tilde{i}}{2}} \begin{bmatrix}
\tilde{s}\\
\tilde{i}
\end{bmatrix}_q h_{\tilde{s}-s-j} \left[\frac{1}{1-q}\right](-1)^{\tilde{i}+\tilde{j}-s-\ell} h_{\tilde{i}+\tilde{j}-\tilde{s}-\ell+j} \left[\frac{1}{M}\right]\\
& =\sum_{\tilde{i}=\max(1,s)}^{k} \sum_{\tilde{s}= \max(1,s+j)}^{\tilde{i}+j}\sum_{\tilde{j}=\max(0,\ell+\tilde{s}-j-\tilde{i})}^{\ell} h_{k-\tilde{i}} \left[\frac{1}{1-q}\right]e_{\ell-\tilde{j}} \left[\frac{1}{M}\right]\\
& \qquad \times t^{\tilde{i}+\tilde{j}-\tilde{s}} q^{\binom{\tilde{i}}{2}} \begin{bmatrix}
\tilde{s}\\
\tilde{i}
\end{bmatrix}_q h_{\tilde{s}-s-j} \left[\frac{1}{1-q}\right](-1)^{\tilde{i}+\tilde{j}-s-\ell} h_{\tilde{i}+\tilde{j}-\tilde{s}-\ell+j} \left[\frac{1}{M}\right]\\
& =\sum_{\tilde{i}=\max(1,s)}^{k} \sum_{\tilde{s}= \max(1,s+j) }^{\tilde{i}+j}h_{k-\tilde{i}} \left[\frac{1}{1-q}\right]q^{\binom{\tilde{i}}{2}} \begin{bmatrix}
\tilde{s}\\
\tilde{i}
\end{bmatrix}_q e_{\tilde{s}-s-j} \left[-\frac{1}{1-q}\right] t^{\ell-j}\\
& \qquad \times \sum_{\tilde{j}=\max(0,\ell+\tilde{s}-j-\tilde{i})}^{\ell} e_{\ell-\tilde{j}} \left[\frac{1}{M}\right]  e_{\tilde{i}+\tilde{j}-\tilde{s}-\ell+j} \left[\frac{-t}{M}\right]\\
& =t^{\ell-j}\sum_{\tilde{i}=\max(1,s)}^{k} \sum_{\tilde{s}= \max(1,s+j) }^{\tilde{i}+j}h_{k-\tilde{i}} \left[\frac{1}{1-q}\right]q^{\binom{\tilde{i}}{2}} \begin{bmatrix}
\tilde{s}\\
\tilde{i}
\end{bmatrix}_q e_{\tilde{s}-s-j} \left[-\frac{1}{1-q}\right] e_{\tilde{i}+j-\tilde{s}} \left[\frac{1}{1-q}\right] \\
& =t^{\ell-j}\sum_{a=\max(1-s,0)}^{k-s} \sum_{\tilde{s}=\max(1,s+j) }^{s+a+j}h_{k-s-a} \left[\frac{1}{1-q}\right]q^{\binom{s+a}{2}} \begin{bmatrix}
\tilde{s}\\
s+a
\end{bmatrix}_q e_{\tilde{s}-s-j} \left[-\frac{1}{1-q}\right]  e_{s+a+j-\tilde{s}} \left[\frac{1}{1-q}\right] \\
& =t^{\ell-j}\sum_{a=\max(1-s,0)}^{k-s} \sum_{b= \max(1-s-j,1) }^{a}h_{k-s-a} \left[\frac{1}{1-q}\right]q^{\binom{s+a}{2}} \begin{bmatrix}
s+j+b\\
s+a
\end{bmatrix}_q e_{b} \left[-\frac{1}{1-q}\right]  e_{a-b} \left[\frac{1}{1-q}\right]\\
& =t^{\ell-j}\sum_{a=0}^{k-s} \sum_{b= 0 }^{a}h_{k-s-a} \left[\frac{1}{1-q}\right]q^{\binom{s+a}{2}} \begin{bmatrix}
s+j+b\\
s+a
\end{bmatrix}_q e_{b} \left[-\frac{1}{1-q}\right]  e_{a-b} \left[\frac{1}{1-q}\right] \\
& \qquad -\chi(s=0)t^{\ell-j}h_k\left[\frac{1}{1-q}\right] .
\end{align*}

At this point we need the following elementary lemma, which we prove in the appendix of the present article.
\begin{lemma} \label{lem:elementary2}
For $k,s,j\geq 0$ with $(k,s,j)\neq (0,0,0)$, we have
\begin{equation*}
\begin{bmatrix}
s+j\\
s
\end{bmatrix}_q \begin{bmatrix}
k+j-1\\
k-s
\end{bmatrix}_q =\qquad \qquad \qquad \qquad \qquad \qquad \qquad \qquad \qquad \qquad \qquad \qquad 
\end{equation*}
\begin{equation} \label{eq:qlemma3}
=\sum_{a=0}^{k-s} \sum_{b= 0 }^{a}h_{k-s-a} \left[\frac{1}{1-q}\right]q^{\binom{a}{2}+s\cdot a} \begin{bmatrix}
s+j+b\\
s+a
\end{bmatrix}_q e_{b} \left[-\frac{1}{1-q}\right]  e_{a-b} \left[\frac{1}{1-q}\right].
\end{equation}
\end{lemma}

So, using the obvious
\begin{equation}
\binom{s+a}{2}=\binom{s}{2}+\binom{a}{2}+s\cdot a,
\end{equation} 
and Lemma~\ref{lem:elementary2}, we finally get
\begin{equation}
g(\ell,s+j,\ell+s,k)=t^{\ell-j} q^{\binom{s}{2}} \begin{bmatrix}
s+j\\
s
\end{bmatrix}_q \begin{bmatrix}
k+j-1\\
k-s
\end{bmatrix}_q-\chi(s=0)t^{\ell-j}h_k\left[\frac{1}{1-q}\right].
\end{equation}
Therefore
\begin{align*}
	&  h_k\left[\frac{1}{1-q}\right] e_{m-\ell}\left[\frac{X}{M}\right] h_{\ell}\left[\frac{X}{M}\right]+\\
	& + \sum_{s=-\ell}^{\min(m-\ell,k)}\sum_{j=-s}^{\min(k+\ell-s,\ell)} g(\ell,s+j,\ell+s,k) e_{m-s-\ell}\left[\frac{X}{M}\right]h_{s+j}\left[\frac{X}{1-q}\right] h_{\ell-j}\left[\frac{X}{M}\right] \\
	& = h_k\left[\frac{1}{1-q}\right] e_{m-\ell}\left[\frac{X}{M}\right] h_{\ell}\left[\frac{X}{M}\right]+\\
	& \qquad +  \sum_{j=0}^{\ell} t^{\ell-j}\sum_{s=0}^{k}q^{\binom{s}{2}} \begin{bmatrix}
	s+j\\
	s
	\end{bmatrix}_q \begin{bmatrix}
	k+j-1\\
	s+j-1
	\end{bmatrix}_qh_{s+j}\left[\frac{X}{1-q}\right] h_{\ell-j}\left[\frac{X}{M}\right] e_{m-s-\ell}\left[\frac{X}{M}\right] \\
	 & \qquad -\sum_{j=0}^{\ell}t^{\ell-j}h_k\left[\frac{1}{1-q}\right]h_{j}\left[\frac{X}{1-q}\right] h_{\ell-j}\left[\frac{X}{M}\right] e_{m-\ell}\left[\frac{X}{M}\right]\\
	 & = h_k\left[\frac{1}{1-q}\right] e_{m-\ell}\left[\frac{X}{M}\right] h_{\ell}\left[\frac{X}{M}\right]+\\
	 & \qquad +  \sum_{j=0}^{\ell} t^{\ell-j}\sum_{s=0}^{k}q^{\binom{s}{2}} \begin{bmatrix}
	 s+j\\
	 s
	 \end{bmatrix}_q \begin{bmatrix}
	 k+j-1\\
	 s+j-1
	 \end{bmatrix}_qh_{s+j}\left[\frac{X}{1-q}\right] h_{\ell-j}\left[\frac{X}{M}\right] e_{m-s-\ell}\left[\frac{X}{M}\right] \\
	 & \qquad - h_k\left[\frac{1}{1-q}\right]  h_{\ell}\left[X\left(\frac{1}{1-q}+\frac{t}{M}\right)\right] e_{m-\ell}\left[\frac{X}{M}\right]\\
	 & = \sum_{j=0}^{\ell} t^{\ell-j}\sum_{s=0}^{k}q^{\binom{s}{2}} \begin{bmatrix}
	 s+j\\
	 s
	 \end{bmatrix}_q \begin{bmatrix}
	 k+j-1\\
	 s+j-1
	 \end{bmatrix}_qh_{s+j}\left[\frac{X}{1-q}\right] h_{\ell-j}\left[\frac{X}{M}\right] e_{m-s-\ell}\left[\frac{X}{M}\right] ,
\end{align*} 
as we wanted.

\section{Three families of plethystic formulae}

We introduce here three (pairs of) families of polynomials.

\medskip

\emph{By convention, all our polynomials are equal to $0$ if any of their indices is negative.}

\medskip

We set 
\begin{equation}
\widetilde{F}_{n,k}^{(d,\ell)}:=\langle \Delta_{h_{\ell}} \Delta_{e_{n-\ell-d-1}}'E_{n-\ell,k},e_{n-\ell}\rangle.
\end{equation}
and
\begin{align} \label{eq:rel_F_Ftilde}
\notag F_{n,k}^{(d,\ell)} & :=\langle \Delta_{h_{\ell}} \Delta_{e_{n-\ell-d}}E_{n-\ell,k},e_{n-\ell}\rangle\\
\text{(using \eqref{eq:deltaprime})}& =\widetilde{F}_{n,k}^{(d,\ell)}+\widetilde{F}_{n,k}^{(d-1,\ell)}.
\end{align}
Observe that, using Theorem~\ref{thm:Haglund_formula} and Lemma~\ref{lem:Mac_hook_coeff}, we have
\begin{align}
\notag F_{n,k}^{(d,\ell)} & =\langle \Delta_{h_{\ell}} \Delta_{e_{n-d-\ell}}E_{n-\ell,k},e_{n-\ell}\rangle\\ 
& =t^{n-\ell-k}\left\langle \Delta_{h_{n-\ell-k}} e_{n-d}\left[X\frac{1-q^k}{1-q}\right] ,e_{\ell}h_{n-d-\ell}\right\rangle \\
\notag & =t^{n-\ell-k}\left\langle \Delta_{h_{n-\ell-k}}\Delta_{e_{\ell}} e_{n-d}\left[X\frac{1-q^k}{1-q}\right] ,h_{n-d}\right\rangle .
\end{align}
Also, using \eqref{eq:rel_F_Ftilde}, we have 
\begin{align} \label{eq:rel_Ftilde_F}
\notag \widetilde{F}_{n,k}^{(d,\ell)} & =\langle \Delta_{h_{\ell}} \Delta_{e_{n-d-\ell-1}}'E_{n-\ell,k},e_{n-\ell}\rangle\\ 
& =\sum_{u\geq 1}(-1)^{u-1}\langle \Delta_{h_{\ell}} \Delta_{e_{n-(d+u)-\ell}}E_{n-\ell,k},e_{n-\ell}\rangle \\
\notag & =\sum_{u\geq 1}(-1)^{u-1} F_{n,k}^{(d+u,\ell)}.
\end{align}

We set 
\begin{equation}
\widetilde{G}_{n,k}^{(d,\ell)} := t^{n-\ell}q^{\binom{k}{2}}\left\langle  \Delta_{e_{n-\ell}}\Delta_{e_{n-d-1}}' e_{n}\left[X\frac{1-q^k}{1-q}\right] ,h_{n}\right\rangle
\end{equation}
and
\begin{align} \label{eq:rel_G_Gtilde}
\notag G_{n,k}^{(d,\ell)} &:= t^{n-\ell}q^{\binom{k}{2}}\left\langle  \Delta_{e_{n-\ell}}\Delta_{e_{n-d}} e_{n}\left[X\frac{1-q^k}{1-q}\right] ,h_{n}\right\rangle\\
\text{(using \eqref{eq:deltaprime})} & =\widetilde{G}_{n,k}^{(d,\ell)}+\widetilde{G}_{n,k}^{(d-1,\ell)}.
\end{align}
Observe that, using Lemma~\ref{lem:Mac_hook_coeff}, we have
\begin{align*}
G_{n,k}^{(d,\ell)} & := t^{n-\ell}q^{\binom{k}{2}}\left\langle  \Delta_{e_{n-\ell}}\Delta_{e_{n-d}} e_{n}\left[X\frac{1-q^k}{1-q}\right] ,h_{n}\right\rangle\\ 
& = t^{n-\ell}q^{\binom{k}{2}}\left\langle  \Delta_{e_{n-\ell}} e_{n}\left[X\frac{1-q^k}{1-q}\right] ,e_{n-d}h_{d}\right\rangle .
\end{align*}
Also, using \eqref{eq:rel_G_Gtilde}, we have
\begin{align} \label{eq:rel_Gtilde_G}
\notag \widetilde{G}_{n,k}^{(d,\ell)} & = t^{n-\ell}q^{\binom{k}{2}}\left\langle  \Delta_{e_{n-\ell}}\Delta_{e_{n-d-1}}' e_{n}\left[X\frac{1-q^k}{1-q}\right] ,h_{n}\right\rangle\\ 
& =\sum_{u\geq 1}(-1)^{u-1}t^{n-\ell}q^{\binom{k}{2}}\left\langle  \Delta_{e_{n-\ell}}\Delta_{e_{n-(d+u)}} e_{n}\left[X\frac{1-q^k}{1-q}\right] ,h_{n}\right\rangle\\
\notag  & =\sum_{u\geq 1}(-1)^{u-1}G_{n,k}^{(d+u,\ell)}.
\end{align}

We set
\begin{align}
\overline{H}_{n,n,i}^{(d,\ell)} & :=\delta_{i,\ell} q^{\binom{n-\ell-d}{2}}\begin{bmatrix}
n-1\\
\ell
\end{bmatrix}_q\begin{bmatrix}
n-\ell-1\\
d
\end{bmatrix}_q
\end{align}
and for $1\leq k<n$
\begin{align} \label{eq:rel_Htilde_Ftilde}
\notag \overline{H}_{n,k,i}^{(d,\ell)} & :=  t^{n-k}\sum_{s=0}^{k-i}\sum_{h=0}^{n-k}q^{\binom{s}{2}} \begin{bmatrix}
k-i\\
s
\end{bmatrix}_q \begin{bmatrix}
k-1\\
i
\end{bmatrix}_q \begin{bmatrix}
s+i-1+h\\
h
\end{bmatrix}_q \widetilde{F}_{n-k,h}^{(d-k+i+s,\ell-i)}.
\end{align}
Similarly
\begin{align}
H_{n,n,i}^{(d,\ell)} & :=\overline{H}_{n,n,i}^{(d,\ell)}+\overline{H}_{n,n,i}^{(d-1,\ell)}=\delta_{i,\ell} q^{\binom{n-\ell-d}{2}}\begin{bmatrix}
	n-1\\
	\ell
	\end{bmatrix}_q\begin{bmatrix}
	n-\ell\\
	d
	\end{bmatrix}_q
\end{align}
and for $1\leq k<n$
\begin{align}   \label{eq:rel_H_F}
H_{n,k,i}^{(d,\ell)} & := \overline{H}_{n,k,i}^{(d,\ell)}+\overline{H}_{n,k,i}^{(d-1,\ell)}\\
\text{(using \eqref{eq:rel_F_Ftilde})}& =  t^{n-k}\sum_{s=0}^{k-i}\sum_{h=0}^{n-k}q^{\binom{s}{2}} \begin{bmatrix}
k-i\\
s
\end{bmatrix}_q \begin{bmatrix}
k-1\\
i
\end{bmatrix}_q \begin{bmatrix}
s+i-1+h\\
h
\end{bmatrix}_q F_{n-k,h}^{(d-k+i+s,\ell-i)}.
\end{align}
Notice that, using the definitions, we have
\begin{align} \label{eq:rel_Htilde_H}
\overline{H}_{n,k,i}^{(d,\ell)} & =\sum_{u\geq 1}(-1)^{u-1}H_{n,k,i}^{(d+u,\ell)}.
\end{align}

\subsection{Recursive relations among the families}

In this section we establish some recursive relations among our families.

\begin{theorem}
	For $k,\ell,d\geq 0$, $n\geq k+\ell$ and $n\geq d$, we have
	\begin{align} \label{eq:Fnn_formula}
	F_{n,n}^{(d,\ell)} & = \delta_{\ell,0} q^{\binom{n-d}{2}}\begin{bmatrix}
	n\\
	d
	\end{bmatrix}_q  ,
	\end{align}
	and for $n>k$
	\begin{align} \label{eq:rel_F_G}
	F_{n,k}^{(d,\ell)} & =\sum_{s=1}^{\min(k,n-d)} 
	\begin{bmatrix}
	k\\
	s\\
	\end{bmatrix}_q G_{n-k,s}^{(d-k+s,\ell)}.
	\end{align}
	Similarly
	\begin{align} \label{eq:Ftildenn_formula}
	\widetilde{F}_{n,n}^{(d,\ell)} & = \delta_{\ell,0} q^{\binom{n-d}{2}}\begin{bmatrix}
	n-1\\
	d
	\end{bmatrix}_q  ,
	\end{align}
	and for $n>k$
	\begin{align} \label{eq:rel_Ftilde_Gtilde}
	\widetilde{F}_{n,k}^{(d,\ell)} & =\sum_{s=1}^{\min(k,n-d)} 
	\begin{bmatrix}
	k\\
	s\\
	\end{bmatrix}_q \widetilde{G}_{n-k,s}^{(d-k+s,\ell)}.
	\end{align}
\end{theorem}
\begin{proof}
Notice that
\begin{equation}
F_{n,n}^{(d,\ell)}=\langle \Delta_{h_{\ell}} \Delta_{e_{n-\ell-d}}E_{n-\ell,n},e_{n-\ell}\rangle
\end{equation}
so clearly it is equal to $0$ when $\ell\neq 0$. For $\ell=0$ we get
\begin{align}
F_{n,n}^{(d,0)} & =\langle \Delta_{e_{n-d}}E_{n ,n},e_{n }\rangle\\
\text{(using \eqref{eq:Mac_hook_coeff})} & =\langle \Delta_{e_{n-d}}\nabla E_{n ,n},h_{n }\rangle.
\end{align}
So using \cite[Proposition~2.3]{haglundschroeder}, we get
\begin{align}
F_{n,n}^{(d,0)} & =e_{n-d}[[n]_q]\\
\text{(using \eqref{eq:e_q_binomial})} & = q^{\binom{n-d}{2}}\begin{bmatrix}
n\\
d
\end{bmatrix}_q.
\end{align}
If we set
\begin{equation}
f(n,d):=q^{\binom{n-d}{2}}\begin{bmatrix}
n-1\\
d
\end{bmatrix}_q,
\end{equation}
then
\begin{align}
f(n,d)+f(n,d-1)& =q^{\binom{n-d}{2}}\begin{bmatrix}
n-1\\
d
\end{bmatrix}_q+q^{\binom{n-d+1}{2}}\begin{bmatrix}
n-1\\
d-1
\end{bmatrix}_q\\
 & =q^{\binom{n-d}{2}}\left(\begin{bmatrix}
n-1\\
d
\end{bmatrix}_q+q^{n-d}\begin{bmatrix}
n-1\\
d-1
\end{bmatrix}_q\right)\\
& = q^{\binom{n-d}{2}}\begin{bmatrix}
n\\
d
\end{bmatrix}_q,
\end{align}
so that
\begin{align}
q^{\binom{n-d}{2}}\begin{bmatrix}
n-1\\
d
\end{bmatrix}_q & =f(n,d)\\
 & =\sum_{u\geq 1}(-1)^{u-1}q^{\binom{n-(d+u)}{2}}\begin{bmatrix}
n\\
d+u
\end{bmatrix}_q\\
 & =\sum_{u\geq 1}(-1)^{u-1}F_{n,n}^{(d,0)}.
\end{align}
	
Therefore \eqref{eq:Ftildenn_formula} follows from \eqref{eq:Fnn_formula} using \eqref{eq:rel_Ftilde_F}. 

Similarly, \eqref{eq:rel_Ftilde_Gtilde} follows from \eqref{eq:rel_F_G} using \eqref{eq:rel_Ftilde_F} and \eqref{eq:rel_Gtilde_G}. Hence it remains to prove only \eqref{eq:rel_F_G}.

Using \eqref{eq:qn_q_Macexp} and Lemma~\ref{lem:Mac_hook_coeff}, we have
\begin{align*}
F_{n,k}^{(d,\ell)} & = t^{n-\ell-k}\left\langle \Delta_{h_{n-\ell-k}}\Delta_{e_{\ell}} e_{n-d}\left[X\frac{1-q^k}{1-q}\right] ,h_{n-d}\right\rangle \\
& =t^{n-\ell-k} (1-q^k)\sum_{\gamma\vdash n-d}\frac{\Pi_\gamma}{w_\gamma} h_k[(1-t)B_\gamma] h_{n-\ell-k}[B_\gamma]e_{\ell}[B_\gamma]
\end{align*}
\begin{align*}
\text{(using \eqref{eq:e_h_expansion})}& =t^{n-\ell-k} (1-q^k)\sum_{\gamma\vdash n-d}\frac{\Pi_\gamma}{w_\gamma} h_k[(1-t)B_\gamma] \sum_{\nu\vdash n-k} e_{n-k-\ell}[B_\nu]\frac{\widetilde{H}_\nu[MB_\gamma]}{w_\nu} \\
\text{(using \eqref{eq:Macdonald_reciprocity})}& =t^{n-\ell-k} (1-q^k)\sum_{\nu\vdash n-k} e_{n-k-\ell}[B_\nu] \frac{\Pi_\nu}{w_\nu}\sum_{\gamma\vdash n-d} h_k[(1-t)B_\gamma]  \frac{\widetilde{H}_\gamma[MB_\nu]}{w_\gamma}\\
\text{(using \eqref{eq:Haglund_summation})}& =t^{n-\ell-k} (1-q^k)\sum_{\nu\vdash n-k} e_{n-k-\ell}[B_\nu] \frac{\Pi_\nu}{w_\nu}\\
& \qquad \times \sum_{s=1}^{\min(k,n-d)}
q^{\binom{s}{2}} 
\begin{bmatrix}
k-1\\
s-1\\
\end{bmatrix}_q e_{n-d-s}[B_\nu]h_s[(1-t)B_\nu]\\
& =t^{n-\ell-k}  \sum_{s=1}^{\min(k,n-d)}
q^{\binom{s}{2}} 
\begin{bmatrix}
k\\
s\\
\end{bmatrix}_q\\
& \qquad \times (1-q^s)\sum_{\nu\vdash n-k}  \frac{\Pi_\nu}{w_\nu} e_{n-k-\ell}[B_\nu] e_{n-d-s}[B_\nu]h_s[(1-t)B_\nu] \\
\text{(using \eqref{eq:qn_q_Macexp})}& =t^{n-\ell-k}  \sum_{s=1}^{\min(k,n-d)}
q^{\binom{s}{2}} 
\begin{bmatrix}
k\\
s\\
\end{bmatrix}_q \left\langle  \Delta_{e_{n-k-\ell}}\Delta_{e_{n-d-s}} e_{n-k}\left[X\frac{1-q^s}{1-q}\right] ,h_{n-k}\right\rangle \\
& =  \sum_{s=1}^{\min(k,n-d)} 
\begin{bmatrix}
k\\
s\\
\end{bmatrix}_q   G_{n-k,s}^{(d-k+s,\ell)} .
\end{align*}	
\end{proof}

\begin{theorem} 
	For $k,\ell,d\geq 0$, $n\geq k+\ell$ and $n\geq d$, we have
	\begin{align} \label{eq:rel_G_F}
	G_{n,k}^{(d,\ell)}& =\sum_{h=1}^{n} \sum_{s=0}^{\ell} t^{n-\ell}q^{\binom{k}{2}+\binom{s}{2}}\begin{bmatrix}
	k\\
	s\\
	\end{bmatrix}_q \begin{bmatrix}
	h+k-s-1\\
	h-s\\
	\end{bmatrix}_q  F_{n,h}^{(d,\ell-s)}\\
\notag	& =\sum_{h=1}^{n}\sum_{s=0}^{\ell}t^{n-\ell}q^{\binom{k}{2}+\binom{s+1}{2}}\begin{bmatrix}
	h-1\\
	s\\
	\end{bmatrix}_q \begin{bmatrix}
	h+k-s-1\\
	h\\
	\end{bmatrix}_q(F_{n,h}^{(d,\ell-s)}+F_{n,h}^{(d,\ell-s-1)}),
	\end{align}
	and
	\begin{align} \label{eq:rel_Gtilde_Ftilde}
	\widetilde{G}_{n,k}^{(d,\ell)}& =\sum_{h=1}^{n} \sum_{s=0}^{\ell} t^{n-\ell}q^{\binom{k}{2}+\binom{s}{2}}\begin{bmatrix}
	k\\
	s\\
	\end{bmatrix}_q \begin{bmatrix}
	h+k-s-1\\
	h-s\\
	\end{bmatrix}_q  \widetilde{F}_{n,h}^{(d,\ell-s)}\\
\notag	& =\sum_{h=1}^{n}\sum_{s=0}^{\ell}t^{n-\ell}q^{\binom{k}{2}+\binom{s+1}{2}}\begin{bmatrix}
	h-1\\
	s\\
	\end{bmatrix}_q \begin{bmatrix}
	h+k-s-1\\
	h\\
	\end{bmatrix}_q(\widetilde{F}_{n,h}^{(d,\ell-s)}+\widetilde{F}_{n,h}^{(d,\ell-s-1)}).
	\end{align}
\end{theorem}
\begin{proof}
Clearly equations \eqref{eq:rel_Gtilde_Ftilde} follow from equations \eqref{eq:rel_G_F} using \eqref{eq:rel_F_Ftilde} and \eqref{eq:rel_G_Gtilde}. Hence we will prove only equations \eqref{eq:rel_G_F}.

Using \eqref{eq:qn_q_Macexp} and Lemma~\ref{lem:Mac_hook_coeff}, we have
	\begin{align*}
	G_{n,k}^{(d,\ell)} & =t^{n-\ell}q^{\binom{k}{2}}\left\langle  \Delta_{e_{n-\ell}}\Delta_{e_{n-d}} e_{n}\left[X\frac{1-q^k}{1-q}\right] ,h_{n}\right\rangle	
	\end{align*}
	\begin{align*}
	&= t^{n-\ell}q^{\binom{k}{2}}\sum_{\gamma\vdash n-d} \frac{\Pi_\gamma}{w_\gamma} (1-q^k)\sum_{\nu\vdash n} \frac{\widetilde{H}_\nu[MB_\gamma]}{w_\nu}h_k[(1-t)B_\nu]e_{n-\ell}[B_\nu]\\
	\text{(using \eqref{eq:mastereq})} & =\sum_{\gamma\vdash n-d} \frac{\Pi_\gamma}{w_\gamma} (1-q^k)\sum_{j=0}^{n-\ell} t^{n-\ell-j}\sum_{s=0}^{\ell}q^{\binom{k}{2}+\binom{s}{2}} \begin{bmatrix}
	 s+j\\
	 s
	 \end{bmatrix}_q \begin{bmatrix}
	 k+j-1\\
	 s+j-1
	 \end{bmatrix}_q\\
	  & \qquad \times h_{s+j}\left[(1-t)B_\gamma\right] h_{n-\ell-j}\left[B_\gamma\right] e_{\ell-s}\left[B_\gamma\right]\\
	  & =t^{n-\ell}\sum_{\gamma\vdash n-d} \frac{\Pi_\gamma}{w_\gamma} \sum_{j=0}^{n-\ell} t^{n-\ell-j}\sum_{s=0}^{\ell}q^{\binom{k}{2}+\binom{s}{2}} \begin{bmatrix}
	  k\\
	  s
	  \end{bmatrix}_q \begin{bmatrix}
	  k+j-1\\
	  j
	  \end{bmatrix}_q\\
	  & \qquad \times (1-q^{s+j}) h_{s+j}\left[(1-t)B_\gamma\right] h_{n-\ell-j}\left[B_\gamma\right] e_{\ell-s}\left[B_\gamma\right],
	\end{align*}
	where in the last equality we used the following elementary lemma, whose proof is in the appendix.
	\begin{lemma} \label{lem:elementary3}
		For all $k,s,j\geq 0$
		\begin{equation}
		\begin{bmatrix}
		k\\
		s
		\end{bmatrix}_q \begin{bmatrix}
		k+j-1\\
		j
		\end{bmatrix}_q (1-q^{s+j}) =(1-q^k) \begin{bmatrix}
		s+j\\
		s
		\end{bmatrix}_q \begin{bmatrix}
		k+j-1\\
		k-s
		\end{bmatrix}_q .
		\end{equation}
	\end{lemma}
	Rearranging the terms, and using \eqref{eq:qn_q_Macexp} and Lemma~\ref{lem:Mac_hook_coeff}, we get	
	\begin{align*}
	G_{n,k}^{(d,\ell)}&= \sum_{j=0}^{n-\ell} \sum_{s=0}^{\ell}t^{n-\ell}q^{\binom{k}{2}+\binom{s}{2}} \begin{bmatrix}
	k\\
	s
	\end{bmatrix}_q \begin{bmatrix}
	k+j-1\\
	j
	\end{bmatrix}_qt^{n-\ell-j}\\
	& \qquad \times \sum_{\gamma\vdash n-d} \frac{\Pi_\gamma}{w_\gamma}(1-q^{s+j}) h_{s+j}\left[(1-t)B_\gamma\right] h_{n-\ell-j}\left[B_\gamma\right] e_{\ell-s}\left[B_\gamma\right]\\
	& =\sum_{j=0}^{n-\ell} \sum_{s=0}^{\ell}t^{n-\ell} q^{\binom{k}{2}+\binom{s}{2}}\begin{bmatrix}
	k\\
	s\\
	\end{bmatrix}_q \begin{bmatrix}
	k+j-1\\
	j\\
	\end{bmatrix}_q  t^{n-\ell-j}\\
	& \qquad \times \left\langle \Delta_{h_{n-\ell-j}} \Delta_{e_{\ell-s}} e_{n-d}\left[X[s+j]_q\right],h_{n-d}\right\rangle\\
	& =\sum_{j=0}^{n-\ell} \sum_{s=0}^{\ell}t^{n-\ell} q^{\binom{k}{2}+\binom{s}{2}}\begin{bmatrix}
	k\\
	s\\
	\end{bmatrix}_q \begin{bmatrix}
	k+j-1\\
	j\\
	\end{bmatrix}_q  F_{n,s+j}^{(d,\ell-s)}\\
	& =\sum_{h=1}^{n} \sum_{s=0}^{\ell} t^{n-\ell}q^{\binom{k}{2}+\binom{s}{2}}\begin{bmatrix}
	k\\
	s\\
	\end{bmatrix}_q \begin{bmatrix}
	h+k-s-1\\
	h-s\\
	\end{bmatrix}_q  F_{n,h}^{(d,\ell-s)}.
	\end{align*}
For the last equality in \eqref{eq:rel_G_F} we can use Lemma~\ref{lem:elementary4}: we have
	\begin{align*}
	G_{n,k}^{(d,\ell)}& =\sum_{h=1}^{n} \sum_{s=0}^{\ell} t^{n-\ell}q^{\binom{k}{2}+\binom{s}{2}}\begin{bmatrix}
	k\\
	s\\
	\end{bmatrix}_q \begin{bmatrix}
	h+k-s-1\\
	h-s\\
	\end{bmatrix}_q  F_{n,h}^{(d,\ell-s)}\qquad \qquad \qquad \qquad 
		\end{align*}
		\begin{align*}
	& =\sum_{h=1}^{n} \sum_{s=0}^{\ell}t^{n-\ell}q^{\binom{k}{2}} \left(q^{\binom{s+1}{2}}\begin{bmatrix}
	h-1\\
	s\\
	\end{bmatrix}_q \begin{bmatrix}
	h+k-s-1\\
	h\\
	\end{bmatrix}_q +q^{\binom{s}{2}}\begin{bmatrix}
	h-1\\
	s-1\\
	\end{bmatrix}_q \begin{bmatrix}
	h+k-s\\
	h\\
	\end{bmatrix}_q\right)\\
	& \qquad \times F_{n,h}^{(d,\ell-s)}\\
	& =\sum_{h=1}^{n}\sum_{s=0}^{\ell}t^{n-\ell}q^{\binom{k}{2}+\binom{s+1}{2}}\begin{bmatrix}
	h-1\\
	s\\
	\end{bmatrix}_q \begin{bmatrix}
	h+k-s-1\\
	h\\
	\end{bmatrix}_q(F_{n,h}^{(d,\ell-s)}+F_{n,h}^{(d,\ell-s-1)}),
	\end{align*}
	as we wanted.
\end{proof}
\begin{theorem} \label{thm:rel_F_H}
	For $k,\ell,d\geq 0$, $n\geq k+\ell$ and $n\geq d$, we have
	\begin{equation} \label{eq:rel_F_and_H}
	F_{n,k}^{(d,\ell)}= \sum_{j=0}^{\min(n-k,\ell)}H_{n,k+j,j}^{(d,\ell)},
	\end{equation}
	and
	\begin{equation} \label{eq:rel_Ftilde_and_Hbar}
	\widetilde{F}_{n,k}^{(d,\ell)}= \sum_{j=0}^{\min(n-k,\ell)}\overline{H}_{n,k+j,j}^{(d,\ell)}.
	\end{equation}
\end{theorem}
\begin{proof}
Clearly \eqref{eq:rel_Ftilde_and_Hbar} follows from \eqref{eq:rel_F_and_H}  using \eqref{eq:rel_F_Ftilde} and \eqref{eq:rel_G_Gtilde}. Hence we will prove only \eqref{eq:rel_F_and_H}.

For $n=k+\ell$, using \eqref{eq:qn_q_Macexp} and Lemma~\ref{lem:Mac_hook_coeff}, we have
\begin{align*}
F_{n,k}^{(d,\ell)} & =\langle \Delta_{h_{\ell}} \Delta_{e_{n-\ell-d}}E_{n-\ell,k},e_{n-\ell}\rangle\\
& =t^{n-\ell-k}\left\langle \Delta_{h_{n-\ell-k}}\Delta_{e_{\ell}} e_{n-d}\left[X\frac{1-q^k}{1-q}\right] ,h_{n-d}\right\rangle \\
& = t^{n-\ell-k}(1-q^k)\sum_{\gamma\vdash n-d}\frac{\Pi_\gamma}{w_\gamma}h_k[(1-t)B_\gamma]h_{n-\ell-k}[B_\gamma]e_\ell[B_\gamma]\\
& = (1-q^k)\sum_{\gamma\vdash n-d}\frac{\Pi_\gamma}{w_\gamma}h_k[(1-t)B_\gamma]e_\ell[B_\gamma]\\
\text{(using \eqref{eq:garsia_haglund_eval})}& = \sum_{\gamma\vdash n-d}\frac{\widetilde{H}_\gamma[M[k]_q]}{w_\gamma}e_\ell[B_\gamma]\\
\text{(using \eqref{eq:e_h_expansion})}& = h_\ell[[k]_q]e_{n-\ell-d}[[k]]_q\\
\text{(using \eqref{eq:h_q_binomial})}& = \begin{bmatrix}
k+\ell-1\\
\ell
\end{bmatrix}_qe_{k-d}[[k]]_q\\
\text{(using \eqref{eq:e_q_binomial})}& = q^{\binom{n-d}{2}}\begin{bmatrix}
k\\
d
\end{bmatrix}_q\begin{bmatrix}
k+\ell-1\\
\ell
\end{bmatrix}_q.
\end{align*}
On the other hand, using \eqref{eq:en_q_sum_Enk},
\begin{align*}
\sum_{j=0}^{\min(n-k,\ell)}H_{n,k+j,j}^{(d,\ell)}& = \sum_{j=0}^{\min(n-k,\ell)} t^{n-k-j}\sum_{s=0}^{k}q^{\binom{s}{2}} \begin{bmatrix}
k\\
s
\end{bmatrix}_q \begin{bmatrix}
k+j-1\\
j
\end{bmatrix}_q \qquad \qquad \qquad 
\end{align*}
\begin{align*}
& \qquad \times \left\langle \Delta_{h_{\ell-j}} \Delta_{e_{n-d-s-\ell}}e_{n-k-\ell}\left[X\frac{1-q^{s+j}}{1-q}\right],e_{n-k-\ell} \right\rangle\\
& = \sum_{j=0}^{\ell} t^{n-k-j}\sum_{s=0}^{k}q^{\binom{s}{2}} \begin{bmatrix}
k\\
s
\end{bmatrix}_q \begin{bmatrix}
k+j-1\\
j
\end{bmatrix}_q\\
& \qquad \times \left\langle \Delta_{h_{\ell-j}} \Delta_{e_{k-d-s}}e_{0}\left[X\frac{1-q^{s+j}}{1-q}\right],e_{0} \right\rangle\\
 & = q^{\binom{n-d}{2}}\begin{bmatrix}
 k\\
 d
 \end{bmatrix}_q\begin{bmatrix}
 k+\ell-1\\
 \ell
 \end{bmatrix}_q,
\end{align*}
as we wanted.
	
For $n>k+\ell$, using \eqref{eq:qn_q_Macexp} and Lemma~\ref{lem:Mac_hook_coeff}, we have
\begin{align*}
	F_{n,k}^{(d,\ell)} & =\langle \Delta_{h_{\ell}} \Delta_{e_{n-\ell-d}}E_{n-\ell,k},e_{n-\ell}\rangle\\
	& =t^{n-\ell-k}\left\langle \Delta_{h_{n-\ell-k}}\Delta_{e_{\ell}} e_{n-d}\left[X\frac{1-q^k}{1-q}\right] ,h_{n-d}\right\rangle \\
	& = t^{n-\ell-k}(1-q^k)\sum_{\gamma\vdash n-d}\frac{\Pi_\gamma}{w_\gamma}h_k[(1-t)B_\gamma]h_{n-\ell-k}[B_\gamma]e_\ell[B_\gamma]\\
	\text{(using \eqref{eq:e_h_expansion})}& = t^{n-\ell-k}(1-q^k)\sum_{\gamma\vdash n-d}\frac{\Pi_\gamma}{w_\gamma}h_k[(1-t)B_\gamma]\sum_{\mu\vdash n-\ell-k}T_\mu\frac{\widetilde{H}_\mu[MB_\gamma]}{w_\mu} e_\ell[B_\gamma]\\
	\text{(using \eqref{eq:Macdonald_reciprocity})}& = t^{n-\ell-k}\sum_{\mu\vdash n-\ell-k} T_\mu\frac{\Pi_\mu}{w_\mu} (1-q^k)\sum_{\gamma\vdash n-d}\frac{\widetilde{H}_\gamma[MB_\mu]}{w_\gamma} h_k[(1-t)B_\gamma]e_\ell[B_\gamma].
	\end{align*}
On the other hand, using \eqref{eq:en_q_sum_Enk},
	\begin{align*}
	\sum_{j=0}^{\min(n-k,\ell)}H_{n,k+j,j}^{(d,\ell)}& = \sum_{j=0}^{\min(n-k,\ell)} t^{n-k-j}\sum_{s=0}^{k}q^{\binom{s}{2}} \begin{bmatrix}
	k\\
	s
	\end{bmatrix}_q \begin{bmatrix}
	k+j-1\\
	j
	\end{bmatrix}_q\\
	& \qquad \times \left\langle \Delta_{h_{\ell-j}} \Delta_{e_{n-d-s-\ell}}e_{n-k-\ell}\left[X\frac{1-q^{s+j}}{1-q}\right],e_{n-k-\ell} \right\rangle\\
	\text{(using \eqref{eq:qn_q_Macexp})}& = t^{n-k-\ell} \sum_{j=0}^{\min(n-k,\ell)} t^{\ell-j}\sum_{s=0}^{k}q^{\binom{s}{2}} \begin{bmatrix}
	k\\
	s
	\end{bmatrix}_q \begin{bmatrix}
	k+j-1\\
	j
	\end{bmatrix}_q\\
	& \qquad \times (1-q^{s+j}) \sum_{\mu\vdash n-\ell-k} \frac{\Pi_\mu}{w_\mu}h_{s+j}[(1-t)B_\mu]T_\mu h_{\ell-j}[B_\mu] e_{n-d-s-\ell}[B_\mu] \\
\qquad 	& = t^{n-k-\ell} \sum_{\mu\vdash n-\ell-k} T_\mu \frac{\Pi_\mu}{w_\mu} \sum_{j=0}^{\min(n-k,\ell)} t^{\ell-j}\sum_{s=0}^{k}q^{\binom{s}{2}} \begin{bmatrix}
		k\\
		s
		\end{bmatrix}_q \begin{bmatrix}
		k+j-1\\
		j
		\end{bmatrix}_q\\
	& \qquad \times (1-q^{s+j})  h_{s+j}[(1-t)B_\mu] h_{\ell-j}[B_\mu] e_{n-d-s-\ell}[B_\mu]
	\end{align*}
	\begin{align*}
	\qquad \quad & = t^{n-k-\ell} \sum_{\mu\vdash n-\ell-k} T_\mu \frac{\Pi_\mu}{w_\mu}(1-q^k) \sum_{j=0}^{\min(n-k,\ell)} t^{\ell-j}\sum_{s=0}^{k}q^{\binom{s}{2}} \begin{bmatrix}
	s+j\\
	s
	\end{bmatrix}_q \begin{bmatrix}
	k+j-1\\
	k-s
	\end{bmatrix}_q\\
	& \qquad \times   h_{s+j}[(1-t)B_\mu] h_{\ell-j}[B_\mu] e_{n-d-s-\ell}[B_\mu],
	\end{align*}
	where in the last equality we used Lemma~\ref{lem:elementary3}.
	
Now comparing the result of the last two computations and using \eqref{eq:mastereq}, we complete our proof.
\end{proof}

\subsection{Two recursions for $F_{n,k}^{(d,\ell)}$ and $\widetilde{F}_{n,k}^{(d,\ell)}$}

We are now able to establish two recursions for $F_{n,k}^{(d,\ell)}$ and $\widetilde{F}_{n,k}^{(d,\ell)}$.

\begin{theorem} \label{thm:reco1_F_and_Ftilde}
	For $k,\ell,d\geq 0$, $n\geq k+\ell$ and $n\geq d$, the $F_{n,k}^{(d,\ell)}$ satisfy the following recursion: for $n\geq 1$
		\begin{align}
	F_{n,n}^{(d,\ell)} & =\delta_{\ell,0}  q^{\binom{n-d}{2}}\begin{bmatrix}
	n\\
	d
	\end{bmatrix}_q  ,
	\end{align}
	and, for $n\geq 1$ and $1\leq k<n$,
	\begin{align} \label{eq:reco1_F_F}
	F_{n,k}^{(d,\ell)} & = \sum_{s=0}^{k}\sum_{i=0}^{\ell}q^{\binom{s}{2}+\binom{i+1}{2}}t^{n-k-\ell}\begin{bmatrix}
	k\\
	s
	\end{bmatrix}_q\sum_{h=1}^{n-k} \begin{bmatrix}
	h-1\\
	i
	\end{bmatrix}_q \\
	\notag & \quad \times \begin{bmatrix}
	h+s-i-1\\
	h
	\end{bmatrix}_q (F_{n-k,h}^{(d-k+s,\ell-i)}+F_{n-k,h}^{(d-k+s,\ell-i-1)}),
	\end{align}
	with initial conditions
	\begin{align}
	F_{0,k}^{(d,\ell)} & = \delta_{k,0}\delta_{\ell,0}\delta_{d,0}, \qquad F_{n,0}^{(d,\ell)} = \delta_{n,0}\delta_{\ell,0}\delta_{d,0}.
	\end{align}
	
	Similarly, the $\widetilde{F}_{n,k}^{(d,\ell)}$ satisfy the following recursion: for $n\geq 1$
		\begin{align}
	\widetilde{F}_{n,n}^{(d,\ell)} & = \delta_{\ell,0}  q^{\binom{n-d}{2}}\begin{bmatrix}
	n-1\\
	d
	\end{bmatrix}_q  ,
	\end{align}
	and, for $n\geq 1$ and $1\leq k<n$,
	\begin{align} \label{eq:reco1_Ftilde_Ftilde}
	\widetilde{F}_{n,k}^{(d,\ell)} & = \sum_{s=0}^{k}\sum_{i=0}^{\ell}q^{\binom{s}{2}+\binom{i+1}{2}}t^{n-k-\ell}\begin{bmatrix}
	k\\
	s
	\end{bmatrix}_q\sum_{h=1}^{n-k} \begin{bmatrix}
	h-1\\
	i
	\end{bmatrix}_q \\
	\notag & \quad \times \begin{bmatrix}
	h+s-i-1\\
	h
	\end{bmatrix}_q (\widetilde{F}_{n-k,h}^{(d-k+s,\ell-i)}+\widetilde{F}_{n-k,h}^{(d-k+s,\ell-i-1)}),
	\end{align}
	with initial conditions
	\begin{align}
	\widetilde{F}_{0,k}^{(d,\ell)} & =  \widetilde{F}_{n,0}^{(d,\ell)} =0.
	\end{align}
\end{theorem}
\begin{proof}
Combining \eqref{eq:rel_F_G} and \eqref{eq:rel_G_F} we immediately get \eqref{eq:reco1_F_F}. Similarly, combining \eqref{eq:rel_Ftilde_Gtilde} and \eqref{eq:rel_Gtilde_Ftilde} we immediately get \eqref{eq:reco1_Ftilde_Ftilde}. The remaining initial conditions are straightforward to check.
\end{proof}

\begin{theorem}
For $k,\ell,d\geq 0$, $n\geq k+\ell$ and $n\geq d$, the $F_{n,k}^{(d,\ell)}$ satisfy the following recursion: for $n\geq 1$
\begin{align}
F_{n,n}^{(d,\ell)} & =\delta_{\ell,0} q^{\binom{n-d}{2}}\begin{bmatrix}
n\\
d
\end{bmatrix}_q,
\end{align}
\begin{align*}
F_{n,k}^{(d,\ell)}
& = \chi(n=k+\ell) q^{\binom{k-d}{2}}\begin{bmatrix}
n-1\\
\ell
\end{bmatrix}_q\begin{bmatrix}
k\\
d
\end{bmatrix}_q\\
& \quad + \sum_{j=0}^{n-k}t^{n-k-j}\sum_{s=0}^{k}\sum_{h=1}^{n-k-j}q^{\binom{s}{2}} \begin{bmatrix}
k\\
s
\end{bmatrix}_q \begin{bmatrix}
k+j-1\\
j
\end{bmatrix}_q \begin{bmatrix}
s+j-1+h\\
h
\end{bmatrix}_q F_{n-k-j,h}^{(d-k+s,\ell-j)}.
\end{align*}
with initial conditions
\begin{align}
F_{0,k}^{(d,\ell)} & = \delta_{k,0}\delta_{\ell,0}\delta_{d,0}, \qquad F_{n,0}^{(d,\ell)} = \delta_{n,0}\delta_{\ell,0}\delta_{d,0}.
\end{align}
Similarly, the $\widetilde{F}_{n,k}^{(d,\ell)}$ satisfy the following recursion: for $n\geq 1$
\begin{align}
\widetilde{F}_{n,n}^{(d,\ell)} & = \delta_{\ell,0} q^{\binom{n-d}{2}}\begin{bmatrix}
n-1\\
d
\end{bmatrix}_q   ,
\end{align}
and, for $n\geq 1$ and $1\leq k<n$, 
\begin{align*}
\widetilde{F}_{n,k}^{(d,\ell)} & = \chi(n=k+\ell) q^{\binom{k-d}{2}}\begin{bmatrix}
n-1\\
\ell
\end{bmatrix}_q\begin{bmatrix}
k-1\\
d
\end{bmatrix}_q\\
 & \quad + \sum_{j=0}^{\min(n-k,\ell)}t^{n-k-j}\sum_{s=0}^{k}\sum_{h=0}^{n-k-j}q^{\binom{s}{2}} \begin{bmatrix}
k\\
s
\end{bmatrix}_q \begin{bmatrix}
k+j-1\\
j
\end{bmatrix}_q \begin{bmatrix}
s+j-1+h\\
h
\end{bmatrix}_q \widetilde{F}_{n-k-j,h}^{(d-k+s,\ell-j)},
\end{align*}
with initial conditions
\begin{align}
\widetilde{F}_{0,k}^{(d,\ell)} & = \widetilde{F}_{n,0}^{(d,\ell)} =0.
\end{align}
\end{theorem}
\begin{proof}
Combining Theorem~\ref{thm:rel_F_H} and the definitions of the polynomials $\overline{H}_{n,k,i}^{(d,\ell)}$ and $H_{n,k,i}^{(d,\ell)}$, the result follows.
\end{proof}

\subsection{A recursion for $G_{n,k}^{(d,\ell)}$ and $\widetilde{G}_{n,k}^{(d,\ell)}$}

In this section we establish a recursion for $G_{n,k}^{(d,\ell)}$ and $\widetilde{G}_{n,k}^{(d,\ell)}$.

\begin{theorem} \label{thm:recoG_and_Gtilde}
	For $k,\ell,d\geq 0$, $n\geq k+\ell$ and $n\geq d$, the $G_{n,k}^{(d,\ell)}$ satisfy the following recursion: 
	\begin{align} \label{eq:reco_G}
	G_{n,k}^{(d,\ell)} & = t^{n-\ell}q^{\binom{k}{2}+\binom{\ell}{2}}\begin{bmatrix}
	k\\
	\ell
	\end{bmatrix}_q \begin{bmatrix}
	n+k-\ell-1\\
	n-\ell
	\end{bmatrix}_q q^{\binom{n-d}{2}}\begin{bmatrix}
	n\\
	d
	\end{bmatrix}_q\\
	& \quad +\sum_{h=1}^{n-1}\sum_{s=0}^{\ell}\sum_{j=1}^{\min(n-d,h)}t^{n-\ell}q^{\binom{k}{2}+\binom{s}{2}}\begin{bmatrix}
	k\\
	s
	\end{bmatrix}_q \begin{bmatrix}
	h\\
	j
	\end{bmatrix}_q \\
	\notag & \quad \times \begin{bmatrix}
	h+k-s-1\\
	h-s
	\end{bmatrix}_q G_{n-h,j}^{(d-h+j,\ell-s)},
	\end{align}
	with initial conditions
	\begin{equation}
	G_{0,k}^{(d,\ell)}=\delta_{d,0}\delta_{\ell,0}q^{\binom{k}{2}}.
	\end{equation}
	
	Similarly
	\begin{align} \label{eq:reco_Gtilde}
	\widetilde{G}_{n,k}^{(d,\ell)} & = t^{n-\ell}q^{\binom{k}{2}+\binom{\ell}{2}}\begin{bmatrix}
	k\\
	\ell
	\end{bmatrix}_q \begin{bmatrix}
	n+k-\ell-1\\
	n-\ell
	\end{bmatrix}_q q^{\binom{n-d}{2}}\begin{bmatrix}
	n-1\\
	d
	\end{bmatrix}_q\\
	& \quad +\sum_{h=1}^{n-1}\sum_{s=0}^{\ell}\sum_{j=1}^{\min(n-d,h)}t^{n-\ell}q^{\binom{k}{2}+\binom{s}{2}}\begin{bmatrix}
	k\\
	s
	\end{bmatrix}_q \begin{bmatrix}
	h\\
	j
	\end{bmatrix}_q \\
	\notag & \quad \times \begin{bmatrix}
	h+k-s-1\\
	h-s
	\end{bmatrix}_q \widetilde{G}_{n-h,j}^{(d-h+j,\ell-s)},
	\end{align}
	with initial conditions
	\begin{equation}
	\widetilde{G}_{0,k}^{(d,\ell)}=\delta_{d,0}\delta_{\ell,0}q^{\binom{k}{2}}.
	\end{equation}
\end{theorem}
\begin{proof}
Using \eqref{eq:rel_G_F}, we have
\begin{align}
G_{n,k}^{(d,\ell)}& =\sum_{h=1}^{n} \sum_{s=0}^{\ell} t^{n-\ell}q^{\binom{k}{2}+\binom{s}{2}}\begin{bmatrix}
k\\
s\\
\end{bmatrix}_q \begin{bmatrix}
h+k-s-1\\
h-s\\
\end{bmatrix}_q  F_{n,h}^{(d,\ell-s)}\\
& = \sum_{s=0}^{\ell} t^{n-\ell}q^{\binom{k}{2}+\binom{s}{2}}\begin{bmatrix}
k\\
s\\
\end{bmatrix}_q \begin{bmatrix}
n+k-s-1\\
n-s\\
\end{bmatrix}_q  F_{n,n}^{(d,\ell-s)}\\
& \quad + \sum_{h=1}^{n-1} \sum_{s=0}^{\ell} t^{n-\ell}q^{\binom{k}{2}+\binom{s}{2}}\begin{bmatrix}
k\\
s\\
\end{bmatrix}_q \begin{bmatrix}
h+k-s-1\\
h-s\\
\end{bmatrix}_q  F_{n,h}^{(d,\ell-s)}\\
\text{(using \eqref{eq:Fnn_formula})}& = t^{n-\ell}q^{\binom{k}{2}+\binom{\ell}{2}}\begin{bmatrix}
k\\
\ell
\end{bmatrix}_q \begin{bmatrix}
n+k-\ell-1\\
n-\ell
\end{bmatrix}_q q^{\binom{n-d}{2}}\begin{bmatrix}
n\\
d
\end{bmatrix}_q\\
& \quad + \sum_{h=1}^{n-1} \sum_{s=0}^{\ell} t^{n-\ell}q^{\binom{k}{2}+\binom{s}{2}}\begin{bmatrix}
k\\
s\\
\end{bmatrix}_q \begin{bmatrix}
h+k-s-1\\
h-s\\
\end{bmatrix}_q  F_{n,h}^{(d,\ell-s)}
\end{align}
\begin{align}
\text{(using \eqref{eq:rel_F_G})}& = t^{n-\ell}q^{\binom{k}{2}+\binom{\ell}{2}}\begin{bmatrix}
k\\
\ell
\end{bmatrix}_q \begin{bmatrix}
n+k-\ell-1\\
n-\ell
\end{bmatrix}_q q^{\binom{n-d}{2}}\begin{bmatrix}
n\\
d
\end{bmatrix}_q\\
&  \quad +\sum_{h=1}^{n-1}\sum_{s=0}^{\ell}\sum_{j=1}^{\min(n-d,h)}t^{n-\ell}q^{\binom{k}{2}+\binom{s}{2}}\begin{bmatrix}
k\\
s
\end{bmatrix}_q \begin{bmatrix}
h\\
j
\end{bmatrix}_q \\
\notag & \quad \times \begin{bmatrix}
h+k-s-1\\
h-s
\end{bmatrix}_q \widetilde{G}_{n-h,j}^{(d-h+j,\ell-s)},
\end{align}
as we wanted.

For \eqref{eq:reco_Gtilde} we get the same argument using \eqref{eq:rel_Gtilde_Ftilde}, \eqref{eq:Ftildenn_formula} and then \eqref{eq:rel_Ftilde_Gtilde}.

The initial conditions are easy to check. This completes the proof.
\end{proof}

\subsection{A recursion for $H_{n,k,i}^{(d,\ell)}$ and $\overline{H}_{n,k,i}^{(d,\ell)}$}

In this section we establish a recursion for $H_{n,k,i}^{(d,\ell)}$ and $\overline{H}_{n,k,i}^{(d,\ell)}$.

\begin{theorem} \label{thm:reco_H_and_Htilde}
For $k,\ell,d\geq 0$, $n\geq k+\ell$ and $n\geq d$, the $H_{n,k,i}^{(d,\ell)}$ satisfy the following recursion:
\begin{align}
H_{n,n,i}^{(d,\ell)} & =\delta_{i,\ell} q^{\binom{n-\ell-d}{2}}\begin{bmatrix}
n-1\\
\ell
\end{bmatrix}_q\begin{bmatrix}
n-\ell\\
d
\end{bmatrix}_q,
\end{align}
and for $1\leq k<n$
\begin{align} \label{eq:reco_H}
 H_{n,k,i}^{(d,\ell)} & = t^{n-k}\sum_{s=0}^{k-i}\sum_{h=0}^{n-k}q^{\binom{s}{2}} \begin{bmatrix}
k-i\\
s
\end{bmatrix}_q \begin{bmatrix}
k-1\\
i
\end{bmatrix}_q \begin{bmatrix}
s+i-1+h\\
h
\end{bmatrix}_q \sum_{j=0}^{\min(n-k-h,\ell-i)}H_{n-k,h+j,j}^{(d-k+i+s,\ell-i)},
\end{align}
with initial conditions
\begin{equation}
H_{0,k,i}^{(d,\ell)}=H_{n,0,i}^{(d,\ell)}=0.
\end{equation}
Similarly
\begin{align}
\overline{H}_{n,n,i}^{(d,\ell)} & =\delta_{i,\ell} q^{\binom{n-\ell-d}{2}}\begin{bmatrix}
n-1\\
\ell
\end{bmatrix}_q\begin{bmatrix}
n-\ell-1\\
d
\end{bmatrix}_q,
\end{align}
and for $1\leq k<n$
\begin{align} \label{eq:reco_Hbar}
 \overline{H}_{n,k,i}^{(d,\ell)} & = t^{n-k}\sum_{s=0}^{k-i}\sum_{h=0}^{n-k}q^{\binom{s}{2}} \begin{bmatrix}
k-i\\
s
\end{bmatrix}_q \begin{bmatrix}
k-1\\
i
\end{bmatrix}_q \begin{bmatrix}
s+i-1+h\\
h
\end{bmatrix}_q \sum_{j=0}^{\min(n-k-h,\ell-i)}\overline{H}_{n-k,h+j,j}^{(d-k+i+s,\ell-i)},
\end{align}
with initial conditions
\begin{equation}
\overline{H}_{0,k,i}^{(d,\ell)}=\overline{H}_{n,0,i}^{(d,\ell)}=0.
\end{equation}
\end{theorem}
\begin{proof}
Clearly \eqref{eq:reco_Hbar} follows from \eqref{eq:reco_H} using \eqref{eq:rel_Htilde_H}. Hence we will only prove \eqref{eq:reco_H}.

To prove \eqref{eq:reco_H}, just combine the definition \eqref{eq:rel_Htilde_H} of $H_{n,k,i}^{(d,\ell)}$ and \eqref{eq:rel_F_and_H}.

The initial conditions are easy to check.
\end{proof}
\begin{remark} \label{rem:rewriting_reco_H}
Using some change of variables, we can rewrite the recursions in Theorem~\ref{thm:reco_H_and_Htilde} in the following way:
\begin{align*} 
\notag \overline{H}_{n,k,i}^{(d,\ell)} & = t^{n-k}\sum_{s=0}^{k-i}\sum_{h=0}^{n-k}q^{\binom{s}{2}} \begin{bmatrix}
k-i\\
s
\end{bmatrix}_q \begin{bmatrix}
k-1\\
i
\end{bmatrix}_q \begin{bmatrix}
s+i-1+h\\
h
\end{bmatrix}_q \sum_{j=0}^{\min(n-k-h,\ell-i)}\overline{H}_{n-k,h+j,j}^{(d-k+i+s,\ell-i)}\\
\notag & = t^{n-k}\sum_{s=0}^{k-i}\sum_{j=0}^{\ell-i}\sum_{f=j}^{n-k+j}q^{\binom{s}{2}} \begin{bmatrix}
k-i\\
s
\end{bmatrix}_q \begin{bmatrix}
k-1\\
i
\end{bmatrix}_q \begin{bmatrix}
s+i-1+f-j\\
f-j
\end{bmatrix}_q \overline{H}_{n-k,f,j}^{(d-k+i+s,\ell-i)}\\
& = t^{n-k}\sum_{p=i}^{k}\sum_{j=0}^{\ell-i}\sum_{f=j}^{n-k+j}q^{\binom{p-i}{2}} \begin{bmatrix}
k-i\\
p-i
\end{bmatrix}_q \begin{bmatrix}
k-1\\
i
\end{bmatrix}_q \begin{bmatrix}
p-1+f-j\\
f-j
\end{bmatrix}_q \overline{H}_{n-k,f,j}^{(d-k+p,\ell-i)}.
\end{align*}
Similarly for $H_{n,k,i}^{(d,\ell)}$.
\end{remark}

\section{Another symmetric function identity}
The goal of this section is to prove the following theorem and deduce some of its consequences.

\begin{theorem} \label{thm:schroeder_identity}
	For all $a,b,k\in \mathbb{N}$, with $a\geq 1$, $b\geq 1$ and $1\leq k\leq a$, we have
	\begin{equation} \label{eq:Schroeder_identity}
	\langle \Delta_{e_a}'\Delta_{e_{a+b-k-1}}'e_{a+b},h_{a+b}\rangle = \langle \Delta_{h_k}\Delta_{e_{a-k}}' e_{a+b-k},e_{a+b-k}\rangle.
	\end{equation}
\end{theorem}

\subsection{Proof of Theorem~\ref{thm:schroeder_identity}}

We start working on the right hand side of \eqref{eq:Schroeder_identity}: using \eqref{eq:en_expansion}, we have
\begin{align*}
\Delta_{h_k}\Delta_{e_{a-k}}' e_{a+b-k} & = \sum_{\gamma\vdash a+b-k} h_k[B_\gamma] e_{a-k}[B_\gamma-1] \widetilde{H}_{\gamma}[X] \frac{MB_{\gamma}\Pi_{\gamma}}{w_{\gamma}}\\
\text{(using \eqref{eq:e_h_sum_alphabets})} & = \sum_{\gamma\vdash a+b-k} h_k[B_\gamma] \sum_{i=0}^{a-k}(-1)^ie_{a-k-i}[B_\gamma] \widetilde{H}_{\gamma}[X] \frac{MB_{\gamma}\Pi_{\gamma}}{w_{\gamma}}\\
& = \sum_{i=0}^{a-k}(-1)^i\sum_{\gamma\vdash a+b-k} h_k\left[\frac{MB_\gamma}{M}\right] e_{a-k-i}\left[\frac{MB_\gamma}{M}\right] \widetilde{H}_{\gamma}[X] \frac{MB_{\gamma}\Pi_{\gamma}}{w_{\gamma}}\\
\text{(using \eqref{eq:e_h_expansion})} & = \sum_{i=0}^{a-k}(-1)^i\sum_{\gamma\vdash a+b-k} \sum_{\mu\vdash a-i}\frac{e_k[B_{\mu}]\widetilde{H}_{\mu}[MB_{\gamma}]}{w_{\mu}} \cdot  \widetilde{H}_{\gamma}[X] \frac{MB_{\gamma}\Pi_{\gamma}}{w_{\gamma}}\\
\text{(as $e_k[B_{\mu}]=0$ for $k>|\mu|$)} & = \sum_{i=0}^{a}(-1)^i\sum_{\gamma\vdash a+b-k} \sum_{\mu\vdash a-i}\frac{e_k[B_{\mu}]\widetilde{H}_{\mu}[MB_{\gamma}]}{w_{\mu}} \cdot  \widetilde{H}_{\gamma}[X] \frac{MB_{\gamma}\Pi_{\gamma}}{w_{\gamma}}\\
\text{(using \eqref{eq:Macdonald_reciprocity})} & = \sum_{i=0}^{a}(-1)^i\sum_{\gamma\vdash a+b-k} \sum_{\mu\vdash a-i}\frac{e_k[B_{\mu}]}{w_{\mu}}\frac{\widetilde{H}_{\gamma}[MB_{\mu}]\Pi_{\mu}}{\Pi_{\gamma}} \cdot  \widetilde{H}_{\gamma}[X] \frac{MB_{\gamma}\Pi_{\gamma}}{w_{\gamma}}\\
& = \sum_{i=0}^{a}(-1)^i\sum_{\gamma\vdash a+b-k} \sum_{\mu\vdash a-i} e_k\left[\frac{MB_\mu}{M}\right] \widetilde{H}_{\gamma}[MB_{\mu}] \cdot  \widetilde{H}_{\gamma}[X] \frac{\Pi_{\mu} MB_{\gamma}}{w_{\mu}w_{\gamma}}\\
\text{(using \eqref{eq:def_dmunu})} & = \sum_{i=0}^{a}(-1)^i\sum_{\gamma\vdash a+b-k} \sum_{\mu\vdash a-i} \sum_{\beta \supset_k \gamma} d_{\beta \gamma}^{(k)} \widetilde{H}_{\beta}[MB_{\mu}] \cdot  \widetilde{H}_{\gamma}[X] \frac{\Pi_{\mu} MB_{\gamma}}{w_{\mu}w_{\gamma}}\\
\text{(using \eqref{eq:rel_cmunu_dmunu})} & = \sum_{i=0}^{a}(-1)^i\sum_{\gamma\vdash a+b-k} \sum_{\mu\vdash a-i} \sum_{\beta \supset_k \gamma} c_{\beta \gamma}^{(k)}\frac{w_{\gamma}}{w_{\beta}} \widetilde{H}_{\beta}[MB_{\mu}] \cdot  \widetilde{H}_{\gamma}[X] \frac{\Pi_{\mu} MB_{\gamma}}{w_{\mu}w_{\gamma}}\\
& =\sum_{i=0}^{a}(-1)^i \sum_{\mu\vdash a-i}  \sum_{\beta\vdash a+b} \sum_{\gamma \subset_k \beta} c_{\beta \gamma}^{(k)}B_{\gamma} \widetilde{H}_{\gamma}[X] \cdot \widetilde{H}_{\beta}[MB_{\mu}] \frac{\Pi_{\mu} M}{w_{\mu}w_{\beta}}\\
\text{(using \eqref{eq:Macdonald_reciprocity})} & =  \sum_{i=0}^{a}(-1)^i \sum_{\beta\vdash a+b} \sum_{\gamma \subset_k \beta} c_{\beta \gamma}^{(k)}B_{\gamma} \widetilde{H}_{\gamma}[X] \cdot \sum_{\mu\vdash a-i}  \frac{\widetilde{H}_{\mu}[MB_{\beta}]}{w_{\mu}}\frac{\Pi_{\beta} M}{w_{\beta}} 
\end{align*}

\begin{align*}
\text{(using \eqref{eq:e_h_expansion})}  & =  \sum_{\beta\vdash a+b} \sum_{\gamma \subset_k \beta} c_{\beta \gamma}^{(k)}B_{\gamma} \widetilde{H}_{\gamma}[X] \cdot \sum_{i=0}^{a}(-1)^ie_{a-i}[B_{\beta}]\frac{\Pi_{\beta} M}{w_{\beta}}\\
\text{(using \eqref{eq:e_h_sum_alphabets})} & =  \sum_{\beta\vdash a+b} \sum_{\gamma \subset_k \beta} c_{\beta \gamma}^{(k)}B_{\gamma} \widetilde{H}_{\gamma}[X] \cdot e_{a}[B_{\beta}-1]\frac{\Pi_{\beta} M}{w_{\beta}}.
\end{align*}
Taking the scalar product with $e_{a+b-k}$ we get
\begin{align*}
\langle \Delta_{h_k}\Delta_{e_{a-k}}' e_{a+b-k} ,e_{a+b-k}\rangle &  =  \sum_{\beta\vdash a+b} \sum_{\gamma \subset_k \beta} c_{\beta \gamma}^{(k)}B_{\gamma} \langle \widetilde{H}_{\gamma}[X],e_{a+b-k}\rangle \cdot e_a[B_{\beta}-1]\frac{\Pi_{\beta} M}{w_{\beta}}\\
\text{(using \eqref{eq:Mac_hook_coeff})} &  =  \sum_{\beta\vdash a+b} \sum_{\gamma \subset_k \beta} c_{\beta \gamma}^{(k)}B_{\gamma} e_{a+b-k}[B_{\gamma}] \cdot e_a[B_{\beta}-1]\frac{\Pi_{\beta} M}{w_{\beta}}\\
\text{(using \eqref{eq:Bmu_Tmu})} &  =  \sum_{\beta\vdash a+b} \sum_{\gamma \subset_k \beta} c_{\beta \gamma}^{(k)}B_{\gamma} T_{\gamma} \cdot e_a[B_{\beta}-1]\frac{\Pi_{\beta} M}{w_{\beta}}.
\end{align*}
For the left hand side we have
\begin{align*}
\langle \Delta_{e_a}'\Delta_{e_{a+b-k-1}}'e_{a+b},h_{a+b}\rangle 
&  =  \sum_{\beta\vdash a+b} e_{a+b-k-1}[B_{\beta}-1]B_{\beta} \langle \widetilde{H}_{\beta}[X],h_{a+b}\rangle \cdot e_a[B_{\beta}-1]\frac{\Pi_{\beta} M}{w_{\beta}}\\
\text{(using \eqref{eq:Bmu_Tmu})} &  =  \sum_{\beta\vdash a+b} e_{a+b-k-1}[B_{\beta}-1]B_{\beta}   \cdot e_a[B_{\beta}-1]\frac{\Pi_{\beta} M}{w_{\beta}}.
\end{align*}
So, in order to conclude the proof of \eqref{eq:Schroeder_identity}, it will be enough to prove the following lemma.
\begin{lemma}
For every $n,k\in \mathbb{N}$, with $n> k\geq 1$, and for every $\beta\vdash n$, we have
\begin{equation}
e_{n-k-1}[B_{\beta}-1]B_{\beta}=\sum_{\gamma \subset_k \beta} c_{\beta \gamma}^{(k)}B_{\gamma} T_{\gamma} .
\end{equation}
\end{lemma}
\begin{proof}
For $\gamma\vdash n-k$, applying \eqref{eq:Pieri_sum1}, we have 
\begin{equation}
B_{\gamma}=\sum_{\delta\subset_1 \gamma}c_{\gamma \delta}^{(1)}.
\end{equation}
	
Therefore
\begin{align*}
\sum_{\gamma \subset_k \beta} c_{\beta \gamma}^{(k)}B_{\gamma} T_{\gamma} & = \sum_{\gamma \subset_k \beta} c_{\beta \gamma}^{(k)}\sum_{\delta\subset_1 \gamma}c_{\gamma \delta}^{(1)} T_{\gamma}\\
& = \sum_{\delta\subset_{k+1} \beta}\sum_{\delta\subset_1 \gamma \subset_k \beta} c_{\beta \gamma}^{(k)}c_{\gamma \delta}^{(1)} T_{\gamma}\\
& = \sum_{\delta\subset_{k+1} \beta}T_{\delta}B_{\beta/\delta}\frac{1}{B_{\beta/\delta}}\sum_{\delta\subset_1 \gamma \subset_k \beta} c_{\beta \gamma}^{(k)}c_{\gamma \delta}^{(1)} \frac{T_{\gamma}}{T_{\delta}}\\
\text{(using \eqref{eq:cmunu_recursion})}& = \sum_{\delta\subset_{k+1} \beta}T_{\delta}B_{\beta/\delta} c_{\beta \delta}^{(k+1)}\\
& = B_{\beta}\sum_{\delta\subset_{k+1} \beta} c_{\beta \delta}^{(k+1)}T_{\delta} -\sum_{\delta\subset_{k+1} \beta} c_{\beta \delta}^{(k+1)}B_{\delta}T_{\delta},
\end{align*}
which gives
\begin{equation}
\sum_{\gamma \subset_k \beta} c_{\beta \gamma}^{(k)}B_{\gamma} T_{\gamma}=-\sum_{\delta\subset_{k+1} \beta} c_{\beta \delta}^{(k+1)}B_{\delta}T_{\delta} +B_{\beta}\sum_{\delta\subset_{k+1} \beta} c_{\beta \delta}^{(k+1)}T_{\delta}.
\end{equation}
Observe that by \eqref{eq:Bmu_Tmu} we have
\begin{align*}
\sum_{\delta\subset_{k+1} \beta} c_{\beta \delta}^{(k+1)}T_{\delta} 
 & = \sum_{\delta\subset_{k+1} \beta} c_{\beta \delta}^{(k+1)}e_{n-k-1}[B_{\delta}]\\
\text{(using \eqref{eq:Mac_hook_coeff})}& = \sum_{\delta\subset_{k+1} \beta} c_{\beta \delta}^{(k+1)}\langle \widetilde{H}_{\delta},e_{n-k-1}\rangle \\
\text{(using \eqref{eq:def_cmunu})}& =\langle h_{k+1}^{\perp}\widetilde{H}_{\beta},e_{n-k-1}\rangle\\
& =\langle \widetilde{H}_{\beta},e_{n-k-1}h_{k+1}\rangle\\
\text{(using \eqref{eq:Mac_hook_coeff})}  & = e_{n-k-1}[B_{\beta}],
\end{align*}
so
\begin{equation} \label{eq:inductive_step}
\sum_{\gamma \subset_k \beta} c_{\beta \gamma}^{(k)}B_{\gamma} T_{\gamma}=-\sum_{\delta\subset_{k+1} \beta} c_{\beta \delta}^{(k+1)}B_{\delta}T_{\delta} +B_{\beta}e_{n-k-1}[B_{\beta}].
\end{equation}
Therefore we can argue by induction on $n-k$: for $n-k=1$, we have $k=n-1$, so
\begin{equation}
\sum_{\gamma \subset_{n-1} \beta} c_{\beta \gamma}^{(n-1)}B_{\gamma} T_{\gamma}=c_{\beta (1)}^{(n-1)} B_{(1)}T_{(1)}=c_{\beta (1)}^{(n-1)},
\end{equation}
but
\begin{align*}
c_{\beta (1)}^{(n-1)} & =c_{\beta (1)}^{(n-1)}\langle e_{1},e_1\rangle\\
& =c_{\beta (1)}^{(n-1)}\langle \widetilde{H}_{(1)},e_1\rangle\\
& = \langle h_{n-1}^{\perp}\widetilde{H}_{\beta},e_1\rangle\\
& = \langle \widetilde{H}_{\beta},e_1 h_{n-1}\rangle\\
\text{(using \eqref{eq:Mac_hook_coeff})} & = e_1[B_{\beta}]= B_{\beta} = e_0[B_{\beta}-1]B_{\beta}.
\end{align*}
If $n-k\geq 2$, by induction
\begin{equation}
\sum_{\delta\subset_{k+1} \beta} c_{\beta \delta}^{(k+1)}B_{\delta}T_{\delta}=e_{n-k-2}[B_{\beta}-1]B_{\beta},
\end{equation}
so, using \eqref{eq:inductive_step}, we have
\begin{align*}
\sum_{\gamma \subset_k \beta} c_{\beta \gamma}^{(k)}B_{\gamma} T_{\gamma} & =-\sum_{\delta\subset_{k+1} \beta} c_{\beta \delta}^{(k+1)}B_{\delta}T_{\delta} +B_{\beta}e_{n-k-1}[B_{\beta}]\\
& = -e_{n-k-2}[B_{\beta}-1]B_{\beta}+B_{\beta}e_{n-k-1}[B_{\beta}]\\
& = B_{\beta}(e_{n-k-1}[B_{\beta}]-e_{n-k-2}[B_{\beta}-1])\\
& = B_{\beta}e_{n-k-1}[B_{\beta}-1].
\end{align*}
This concludes the proof of the lemma.
\end{proof}
We already observed that the lemma finishes the proof of Theorem~\ref{thm:schroeder_identity}.

\subsection{Some consequences of Theorem~\ref{thm:schroeder_identity}}

We deduce here a few consequences of Theorem~\ref{thm:schroeder_identity}.

First of all, we have the following corollary.
\begin{corollary}
\begin{equation} \label{eq:Schroeder_corol}
\langle \Delta_{e_a}\Delta_{e_{a+b-k-1}}'e_{a+b},h_{a+b}\rangle = \langle \Delta_{h_k}\Delta_{e_{a-k}} e_{a+b-k},e_{a+b-k}\rangle.
\end{equation}
\end{corollary}
\begin{proof}
The same identity \eqref{eq:Schroeder_identity} with $a$ and $b$ replaced by $a-1$ and $b+1$ respectively, gives 
\begin{equation} \label{eq:Schroeder_id2}
\langle \Delta_{e_{a-1}}'\Delta_{e_{a+b-k-1}}'e_{a+b},h_{a+b}\rangle = \langle \Delta_{h_k}\Delta_{e_{a-k-1}}' e_{a+b-k},e_{a+b-k}\rangle.
\end{equation}
Adding this one to \eqref{eq:Schroeder_identity} and using $\Delta_{e_m}=\Delta_{e_m}'+\Delta_{e_{m-1}}'$ on $\Lambda^{(n)}$ with $1\leq m\leq n$, we get the result.
\end{proof}
As an application of \eqref{eq:Schroeder_identity}, we give a first algebraic proof of the symmetry \eqref{eq:HRW_symmetry} predicted by Conjecture 7.1 in \cite{haglundremmelwilson}, which is left open in \cite{zabrocki} as Conjecture 16. A combinatorial proof will be provided in Section~6.
\begin{theorem} \label{thm:symmetry}
	For $n>k+\ell$, the following expression is symmetric in $k$ and $\ell$:
	\begin{equation}
	\langle \Delta_{h_k}\Delta_{e_{n-k-\ell-1}}' e_{n-k},e_{n-k}\rangle.
	\end{equation}
\end{theorem}
\begin{proof}[Algebraic proof of Theorem~\ref{thm:symmetry}]
	Using \eqref{eq:Schroeder_identity}, with $n=a+b$ and $\ell=b-1$, we have
	\begin{equation}
	\langle \Delta_{h_k}\Delta_{e_{n-k-\ell-1}}' e_{n-k},e_{n-k}\rangle=\langle \Delta_{e_{n-\ell-1}}'\Delta_{e_{n-k-1}}'e_{n},h_{n}\rangle ,
	\end{equation}
	which is obviously symmetric in $k$ and $\ell$.
\end{proof}
\begin{proposition} \label{prop:SF_sums_F_H}
For $n,d,\ell\geq 0$, $n\geq k+\ell$ and $n\geq d$, we have
\begin{equation} \label{eq:Fnk_sum}
\langle \Delta_{e_{n-\ell-1}}'e_{n},e_{n-d}h_{d}\rangle=\sum_{k=1}^n F_{n,k}^{(d,\ell)}=\sum_{k=1}^n \sum_{i=1}^{k-1} H_{n,k,i}^{(d,\ell)},
\end{equation}
and
\begin{equation} \label{eq:Ftildenk_sum}
\langle \Delta_{e_{n-\ell-1}}'e_{n},s_{d+1,1^{n-d-1}}\rangle=\sum_{k=1}^n \widetilde{F}_{n,k}^{(d,\ell)}=\sum_{k=1}^n \sum_{i=1}^{k-1} \overline{H}_{n,k,i}^{(d,\ell)}.
\end{equation}
\end{proposition}
\begin{proof}
To prove the first equality in \eqref{eq:Fnk_sum}, observe that
\begin{align*}
	\sum_{k=1}^n\notag F_{n,k}^{(d,\ell)} & =\langle \Delta_{h_{\ell}} \Delta_{e_{n-d-\ell}}\sum_{k=1}^{n-\ell}E_{n-\ell,k},e_{n-\ell}\rangle\\ 
\text{(using \eqref{eq:en_sum_Enk})}	& =\langle \Delta_{h_{\ell}} \Delta_{e_{n-d-\ell}}e_{n-\ell},e_{n-\ell}\rangle \\
\text{(using \eqref{eq:Schroeder_corol})}& =\langle \Delta_{e_{n-d}}\Delta_{e_{n-\ell-1}}'e_{n},h_{n}\rangle\\
\text{(using \eqref{eq:Mac_hook_coeff})}& =\langle \Delta_{e_{n-\ell-1}}'e_{n},e_{n-d}h_{d}\rangle.
\end{align*}
The second equality follows simply from \eqref{eq:rel_F_and_H}.

A similar computations using \eqref{eq:Schroeder_identity}, \eqref{eq:Mac_hook_coeff_ss} and \eqref{eq:rel_Ftilde_and_Hbar} gives \eqref{eq:Ftildenk_sum}.
\end{proof}

\section{Combinatorial interpretations of plethystic formulae}	

In this section we combine our previous results to provide combinatorial interpretations of several plethystic formulae. We are going to use the notations introduced in Remark~\ref{rmk:no_restrictions}.

\medskip
In the following theorem, the first equalities of \eqref{eq:Ftildenk_qt} and \eqref{eq:Fnk_qt} are due to Zabrocki, as they are immediate consequences of \cite[Theorem~10]{zabrocki}. 
\begin{theorem} \label{thm:qt_enumerators_formulae}
Let $a,b,n,k,i\in \mathbb N$ with $n\geq 1$, $n\geq k\geq 1$, $\ell \geq i$. Then
	\begin{align} 
\label{eq:Ftildenk_qt}	\widetilde{F}_{n,k}^{(d,\ell)} & =\mathbf{D}\widetilde{S}_{n,k}^{(d,\ell)}=\mathbf{B}\widehat{S}_{n,k}^{(\ell,d)},\\
\label{eq:Fnk_qt}	F_{n,k}^{(d,\ell)} & =\mathbf{D}{S'}_{n,k}^{(d,\ell)}=\mathbf{B}S_{n,k}^{(\ell,d)},\\
\label{eq:Gtildenk_qt}	\widetilde{G}_{n-k,k}^{(d,\ell)} & =\mathbf{D}\widetilde{R}_{n,k}^{(d,\ell)}=\mathbf{B}\widehat{R}_{n,k}^{(\ell,d)},\\
\label{eq:Gnk_qt}	G_{n-k,k}^{(d,\ell)} & =\mathbf{D} {R'}_{n,k}^{(d,\ell)}=\mathbf{B}{R}_{n,k}^{(\ell,d)},\\
\label{eq:Hbarnk_qt}	\overline{H}_{n,k,i}^{(d,\ell)} & =\mathbf{B'}\overline{T}_{n,k,i}^{(d,\ell)},\\
\label{eq:Hnk_qt}	H_{n,k,i}^{(d,\ell)} & =\mathbf{B'}T_{n,k,i}^{(d,\ell)}.
	\end{align}
\end{theorem}
\begin{proof}
Identity \eqref{eq:Ftildenk_qt} follows by comparing Theorem~\ref{thm:reco1_first_bounce} with Theorem~\ref{thm:reco1_F_and_Ftilde}, and Corollary~\ref{cor:zeta_map}.

Identity \eqref{eq:Fnk_qt} follows from \eqref{eq:Ftildenk_qt}, and the definitions of the polynomials involved.

Identity \eqref{eq:Gtildenk_qt} follows by comparing Theorem~\ref{thm:reco_R_first_bounce} with Theorem~\ref{thm:recoG_and_Gtilde}, and Corollary~\ref{cor:zeta_map}.

Identity \eqref{eq:Gnk_qt} follows from \eqref{eq:Gtildenk_qt}, and the definitions of the polynomials involved.

Identity \eqref{eq:Hbarnk_qt} follows by comparing Theorem~\ref{thm:reco_old_bounce} with Theorem~\ref{thm:reco_H_and_Htilde} and Remark~\ref{rem:rewriting_reco_H}.

Identity \eqref{eq:Hnk_qt} follows from \eqref{eq:Hbarnk_qt}, and the definitions of the polynomials involved.
\end{proof}

\subsection{Decorated $q,t$-Schr\"{o}der}

Now we solve half of Problem 8.1 in \cite{haglundremmelwilson}, by proving Conjecture 5.2.1.1 in \cite{wilsonPhD}. This gives the decorated analogue of Haglund's $q,t$-Schr\"{o}der theorem \cite{haglundschroeder} for $\Delta_{e_k}e_n$.
\begin{theorem} \label{thm:decoqtSchroeder}
For $a,b,k\in \mathbb{N}\cup\{0\}$, $a\geq 1$, $a+b\geq k+1$, we have
\begin{equation} \label{eq:qtSchroder1}
\langle \Delta_{e_{a+b-k-1}}'e_{a+b},s_{b+1,1^{a-1}}\rangle= \mathbf{D}\widetilde{\mathcal{D}}_{a+b}^{(b,k)}=\mathbf{B'}\overline{\mathcal{D}}_{a+b}^{(b,k)}=\mathbf{B}\widehat{\mathcal{D}}_{a+b}^{(k,b)}.
\end{equation}
so that, for $a+b\geq k+1$, we have
\begin{equation} \label{eq:qtSchroder2}
\langle \Delta_{e_{a+b-k-1}}'e_{a+b},e_a h_{b}\rangle= \mathbf{D} \mathcal{D}_{a+b}^{(b,k)}=\mathbf{B'}\mathcal{D}_{a+b}^{(b,k)}=\mathbf{B}\mathcal{D}_{a+b}^{(b,k)}.
\end{equation}
\end{theorem}
\begin{proof}
To prove \eqref{eq:qtSchroder1}, just combine Proposition~\ref{prop:SF_sums_F_H} and Theorem~\ref{thm:qt_enumerators_formulae}.

To prove the first two equalities in \eqref{eq:qtSchroder2}, just add \eqref{eq:qtSchroder1} with itself where we replace $b$ by $b-1$ and $a$ by $a+1$. For the last equality in \eqref{eq:qtSchroder2}, doing the same operation we get
\begin{equation}
\langle \Delta_{e_{a+b-k-1}}'e_{a+b},e_a h_{b}\rangle=\mathbf{B}\widehat{\mathcal{D}}_{a+b}^{(k,b)}+\mathbf{B}\widehat{\mathcal{D}}_{a+b}^{(k,b-1)}.
\end{equation}
But using again \eqref{eq:qtSchroder1} and Corollary~\ref{cor:comb_symmetry}, we have
\begin{align}
\mathbf{B}\widehat{\mathcal{D}}_{a+b}^{(k,b)}+\mathbf{B}\widehat{\mathcal{D}}_{a+b}^{(k,b-1)} & = \mathbf{B}'\overline{\mathcal{D}}_{a+b}^{(b,k)}+\mathbf{B}'\overline{\mathcal{D}}_{a+b}^{(b-1,k)}\\
 & = \mathbf{B}'\overline{\mathcal{D}}_{a+b}^{(k,b)}+\mathbf{B}'\overline{\mathcal{D}}_{a+b}^{(k,b-1)}\\
 & = \mathbf{B}\widehat{\mathcal{D}}_{a+b}^{(b,k)}+\mathbf{B}\widehat{\mathcal{D}}_{a+b}^{(b-1,k)}\\
 & =  \mathbf{B} \mathcal{D}_{a+b}^{(b,k)},
\end{align}
as claimed.
\end{proof}
We are now able to give a combinatorial proof of Theorem~\ref{thm:symmetry}.
\begin{proof}[Combinatorial proof of Theorem~\ref{thm:symmetry}]
Using \eqref{eq:Schroeder_identity}, with $n=a+b$ and $\ell=b-1$,  \eqref{eq:Mac_hook_coeff_ss}, and Theorem~\ref{thm:decoqtSchroeder}, we have
\begin{align}
\langle \Delta_{h_k}\Delta_{e_{n-k-\ell-1}}' e_{n-k},e_{n-k}\rangle & =\langle \Delta_{e_{n-\ell-1}}'e_{n},s_{k+1,1^{n-k-1}}\rangle = \mathbf{B'}\overline{\mathcal{D}}_{n}^{(k,\ell)}.
\end{align}
Now the symmetry in $k$ and $\ell$ is explained combinatorially by the bijection $\psi$ of Theorem~\ref{thm:psi_map} (cf. Corollary~\ref{cor:comb_symmetry}).
\end{proof}

\section{Square $q,t$-lattice paths and $\Delta_{e_{n-1}}e_n$}

In this section we explain how a simple combinatorial transformation translates the Delta conjecture for $\Delta_{e_{n-1}}e_n$ into a new square conjecture, similar to the one for $\nabla \omega(p_n)$ proposed in \cite{loehrwarringtonqtsquare} and proved in \cite{leven} after the breakthrough in \cite{carlssonmellit}. In particular, we get a new $q,t$-square out of it.

The interest in these observations relies into a more surprising one: the two symmetric functions $\Delta_{e_{n-1}}e_n$ and $\nabla \omega(p_n)$ have the same evaluation at $t=1/q$.

\subsection{A new square conjecture}

Consider the set $\mathcal{SQ}_n^E$ of \emph{square paths} ending east, i.e. lattice paths going from $(0,0)$ to $(n,n)$ consisting of unit east or north steps, and ending with an east step. The set $\mathcal{SQ}_n^N$ of square paths ending north is defined similarly.

Let us denote by $\mathcal{LSQ}_{n}^{E}$ the set of \emph{labelled square paths} ending east, i.e. the set of pairs $(P,r)$ where $P\in \mathcal{SQ}_n^E$ and $r\in \mathfrak{S}_n$, $r=r_1r_2\cdots r_n$, is such that $r_i<r_{i+1}$ when $w_{i+1}(P)=w_i(P)+1$, where $w_i(P)$ are the letters of the area word of $P$ (i.e. the $i$-th vertical step of $P$ lies on the diagonal $y=x+w_i(P)$). As it is custom with labelled Dyck paths, we identify the $(P,r)$ with square paths with labels next to the right of their vertical steps: see Figure \ref{SQpath} for an example.

\begin{figure}[h!]
	\begin{center}
		\begin{tikzpicture}[scale=1.4]
			\draw[gray] (0,0) grid[step=0.5 cm](5.5,5.5);
			\draw[gray] (2,0) to (7.5,5.5);
			\filldraw (3.5,1.5) circle(2pt);
			\draw[blue, ultra thick, opacity=0.4](3.5,1.5)|-(4,2.5)|-(5,4.5) |-(5.5,5.5) (0,0)|-(1,1)|-(3.5,1.5);
			\draw (.25, 0.25) node {4} circle(.2 cm)
			(0.25, 0.75) node {6} circle(.2 cm)
			(1.25, 1.25) node {11} circle(.2 cm)
			(3.75, 1.75) node {7} circle(.2 cm)
			(3.75, 2.25) node {10} circle(.2 cm)
			(4.25, 2.75) node {1} circle(.2 cm)
			(4.25, 3.25) node {2} circle(.2 cm)
			(4.25, 3.75) node {5} circle(.2 cm)
			(4.25, 4.25) node {8} circle(.2 cm)
			(5.25, 4.75) node {3} circle(.2 cm)
			(5.25, 5.25) node {9} circle(.2 cm);
			\fill[gray,opacity=.2](0,0)--(0,1)--(1,1)--(1,1.5)--(3,1.5)--(3,1)--(2.5,1)--(2.5,.5)--(2,.5)--(2,0) (3.5,2) rectangle (4,2.5) (4,2.5) rectangle (4.5, 4) (4.5,3) rectangle (5,4)(5,3.5) rectangle (5.5,5.5) (5,4.5)rectangle (4,4) (5.5,4.5) rectangle (6.5,5.5) rectangle (7,5) (5.5,4.5) rectangle (6,4);
			\draw [decorate,decoration={brace, mirror,amplitude=8pt},xshift=0pt,yshift=0pt](0,0) --(2,0);
			\draw [decorate,decoration={brace, mirror,amplitude=8pt},xshift=0pt,yshift=0pt](5.5,0)--(5.5,1.5);
			\draw (1,-.5) node{$c$} (6,.75) node {$j$};
			\draw [gray](5.5,4.5) grid[step=.5](6.5,5.5) rectangle (7,5) (5.5,4.5) rectangle (6,4) (6,4.5) rectangle (6.5,5) ;
			\end{tikzpicture}
	\end{center}
	\caption{Example of a labelled square path}\label{SQpath}
\end{figure}
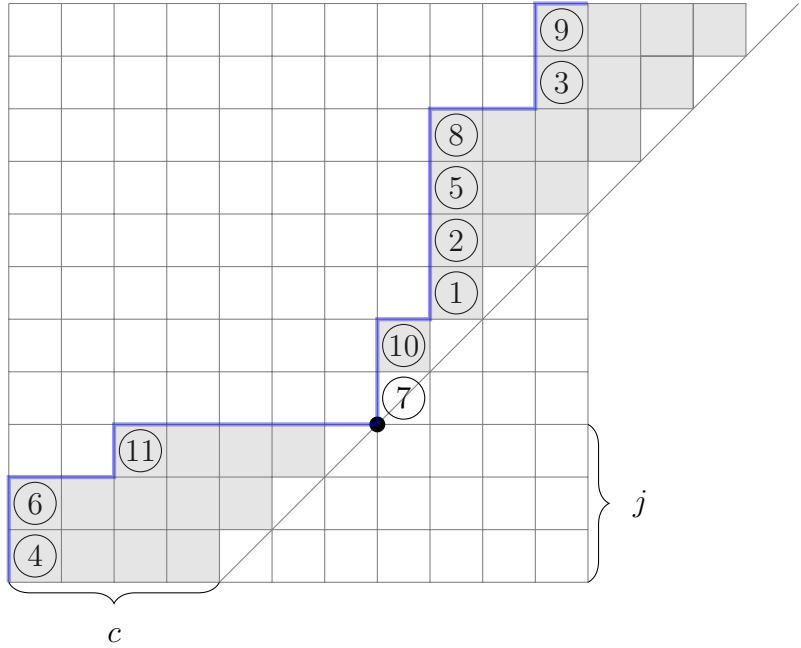

Set $c:=-\min \{w_i(P)\vert i=1,2,\dots n\}$ and $ j:=\min\{i\vert w_i(P)=-c\}-1$ and define $$\mathsf{area}((P,r))=\mathsf{area}(P):=\sum_{i=1}^n w_i(P)+nc-j-c.$$ For example, the path in Figure \ref{SQpath} has $c=4$, $j=3$ and area $31-3-4=24$.

Next, we set
\begin{equation}
\mathsf{dinv}(P,r):=\sum_{1\leq i<j\leq n}\chi(w_i(P)=w_j(P)\text{ and }r_i<r_j,\, \text{ or }w_i(P)=w_j(P)+1\text{ and }r_i>r_j).
\end{equation}
We define $\sigma(P,r)$ to be the permutation obtained by reading the labels of $(P,r)$ along the diagonals $y=x+k$ from top to bottom, from right to left. For example, if $(P,r)$ is the path in Figure~\ref{SQpath} then $\sigma(P,r)=6\,9\,8\,1\; \!\!1\, 4\,3\,5\,2\,1\,1\; \!\! 0\,7$.

We then call $\mathsf{ides}(\sigma(P,r)):=\mathsf{Des}(\sigma(P,r)^{-1})$, where, as usual, for any permutation $\tau\in \mathfrak{S}_n$, we set
\begin{equation}
\mathsf{Des}(\tau):=\{i\mid \tau(i)>\tau(i+1)\}\subseteq \{1,2,\dots,n-1\}.
\end{equation}
For every $S\subseteq \{1,2,\dots,n-1\}$, let $Q_{S,n}$ denote the \emph{Gessel fundamental quasisymmetric function} of degree $n$ indexed by $S$ , i.e.
\begin{equation}
	Q_{S,n}:=\mathop{\sum_{i_1\leq i_2\leq \cdots \leq i_n}}_{i_j<i_{j+1}\text{ if }j\in S}x_{i_1} x_{i_2}  \cdots   x_{i_n}.
\end{equation}
We can now formulate our new square conjecture.
\begin{conjecture} \label{conj:new_square}
	For every $n\geq 1$
	\begin{equation}
	\Delta_{e_{n-1}}e_n=\sum_{(P,r)\in \mathcal{LSQ}_n^E}q^{\mathsf{dinv}(P,r)}t^{\mathsf{area}(P,r)}Q_{\mathsf{ides}(\sigma(P,r)),n}.
	\end{equation}
\end{conjecture}

A similar conjecture can be formulated using the square paths ending north.

\subsection{Relation with the Delta conjecture for $\Delta_{e_{n-1}}e_n$}
Let $\mathcal{LD}_{n,k}^{\text{Fall}}$ be the set of labelled Dyck paths with $k$ decorations on falls as in \cite{haglundremmelwilson} .

We describe a bijective map 
\begin{equation}  \label{gammamap}
\gamma_E: \mathcal{LSQ}^E_n \rightarrow \mathcal{LD}_{n,0}^{Fall} \sqcup \mathcal{LD}_{n,1}^{Fall}. 
\end{equation}
We have $\mathcal{LD}_{n,0}^{Fall}\subseteq \mathcal{LSQ}^E_n$ and so we define $\gamma_{E_{\big{\vert} \mathcal{LD}_{n,0}^{Fall} }}$ to be the identity map. Now take $(P,r)\in \mathcal{LSQ}^E_n$ such that $P$ is not a Dyck path. Create a labelled Dyck path $D$ from $P$ as follows: start from $(0,0)$ by copying $P$ and its labels starting from the point $(j+c,j)$. When we arrive at the end of $P$ we add an extra horizontal step (thus creating a fall since $P$ ended with an east step). Now continue $D$ by copying $P$ and its labels starting from $(0,0)$. The path $D$ ends when we reach the point $(c+j-1, j)$ in $P$. Finally decorate the fall that was created by adding the extra horizontal step. It is easy to see how to invert $\gamma_E$. We refer to Figure \ref{gamma} for an example.  
\begin{figure}[h!]
\centering
	\begin{minipage}{.6\textwidth}
	\centering
		\begin{tikzpicture}[scale=.8]
			\draw[gray] (0,0) grid (8,8) grid (11,5);
			\fill[white](8,5)--(11,5)--(11,8)--(8,5);
			\draw[white,ultra thick] (8,5)--(11,5)--(11,8);
			\draw[gray](3,0)--(11,8) (8,8)--(8,0);
			\draw[ultra thick, blue, opacity=.4](0,0)|-(3,1)|-(4,2) (5,2)|-(6,5)|-(7,7)|-(8,8);
			\draw[ultra thick, red, opacity=1](4,2)--(5,2) ;
			\filldraw (5,2) circle (3pt);
			\draw (.5,.5) node{2} circle (.4cm)
			(3.5,1.5) node{7}circle (.4cm)
			(5.5,2.5) node{1}circle (.4cm)
			(5.5,3.5) node{6}circle (.4cm)
			(5.5,4.5) node{8}circle (.4cm)
			(6.5,5.5) node{4}circle (.4cm)
			(6.5,6.5) node{5}circle (.4cm)
			(7.5,7.5) node{3} circle (.4cm);
		\end{tikzpicture}
	\end{minipage}%
	\begin{minipage}{.4\textwidth}
	\centering
		\begin{tikzpicture}[scale=.8]
			\draw[gray] (0,0) grid (8,8) (0,0)--(8,8);
			\draw[ultra thick, blue, opacity=.4](0,0)|-(1,3)|-(2,5)|-(4,6)|-(7,7)|-(8,8) ;
			\draw[ultra thick, red] (3,6)--(4,6);
			\filldraw (0,0) circle (3pt);
			\draw (.5,.5) node{1}circle (.4cm)
			(.5,1.5) node{6}circle (.4cm)
			(.5,2.5) node{8}circle (.4cm)
			(1.5,3.5) node{4}circle (.4cm)
			(1.5,4.5) node{5}circle (.4cm)
			(2.5,5.5) node{3}circle (.4cm)
			(5.5,6.5) node{2}circle (.4cm)
			(7.5,7.5) node{7}circle (.4cm)
			(2.5,6.5) node{$\ast$};
		\end{tikzpicture}
	\end{minipage} \caption{The map $\gamma_E$. }\label{gamma}
\end{figure}
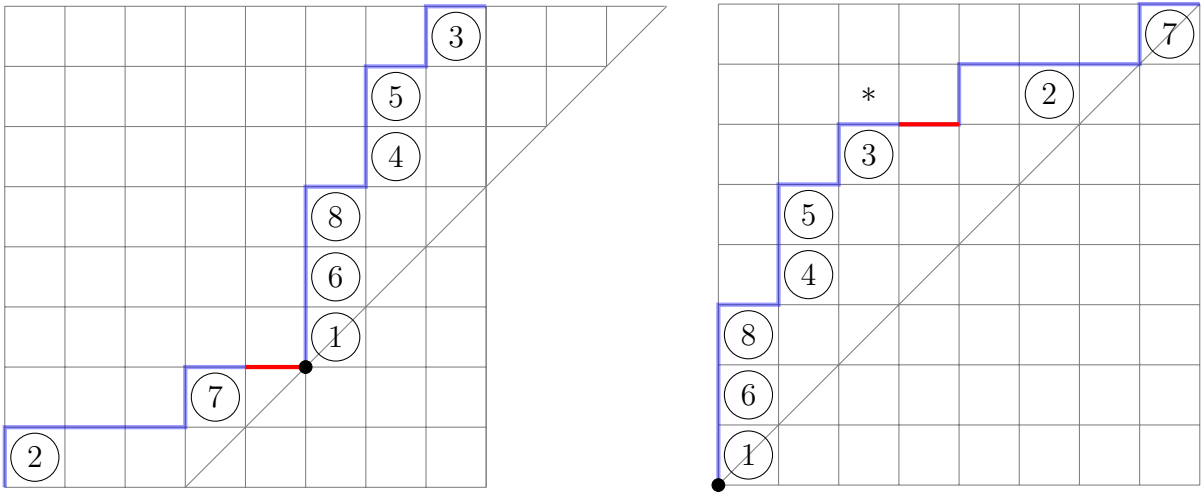

If $\area^-$ and $\dinv^-$ are the statistics on  $\mathcal{LD}_{n,k}^{\text{Fall}}$ defined in \cite[Section~3.1]{haglundremmelwilson}, then for all $(P,r) \in \mathcal{LSQ}^E_n$ \begin{align*} \area^-(\gamma_E((P,r))&=\area((P,r))\\ \dinv^-(\gamma_E((P,r))&=\dinv((P,r)).\end{align*}

\subsection{A new $q,t$-square}
Now consider the sets $\mathcal{SQ}_n^E$ and $\mathcal{SQ}_n^N$ of (unlabelled) square paths ending east and north, respectively. For $P\in \mathcal{SQ}_n^E$ define $\area(P)$ to be the same as in the labelled version and for $P'\in \mathcal{SQ}_n^N$ we define 
\begin{equation}
\area(P')= \sum_{i=1}^n w_i(P)+nc-j.
\end{equation} 

Next, for $P\in \mathcal{SQ}_n^E\sqcup \mathcal{SQ}_n^N$, we set
\begin{align}
\mathsf{dinv}(P)=\dinv(P') & :=\sum_{1\leq i<j\leq n}\chi(w_i(P)=w_j(P))+ \sum_{1\leq i<j< n}\chi(w_i(P)=w_j(P)+1)  \\
 & \quad + \sum_{1\leq i< n}\chi(w_i(P)=w_n(P)+1\text{ and }w_n(P)\geq 0).
\end{align}

Finally we define the $\bounce$ as follows. The bounce path of $P$ starts at the point $(c+j, j)$ and travels north until it hits the beginning of an east step when it turns east until it hits the diagonal $y=x-c$, where it turns north again, and so on. When it crosses the line $y=n$, at the point $(n+m,n)$ it stops and starts again at the point $(m,0)$. Here there is a slight difference between the definitions for paths ending east and paths ending north:
\begin{itemize}
	\item if the path ends east, then if $m=0$ the bounce path starts traveling east at the point $(0,0)$, otherwise, it travels north; 
	\item if the path ends north, then the bounce path always travels north starting from $(m,0)$.  
\end{itemize} 
Then the path bounces on the \emph{end} of east steps and again on the diagonal $y=x-c$. It ends when it arrives at the point $(c+j,j)$. We then label the vertical steps of the bounce starting with $0$'s and adding $1$ every time the bounce path bounces. Finally, $\bounce(P)=\bounce(P')$ is defined to be the sum of the labels of the bounce path. For an example, see Figure~\ref{squarebounceex}, whose bounce equals $15$. 
 
\begin{figure}[h!]
	\begin{center}
		\begin{tikzpicture}[scale=1.1]
		\draw[gray] (0,0) grid[step=0.5 cm](5.5,5.5) (2,0) to (7.5,5.5); 
		\filldraw (3.5,1.5) circle(2pt);
		\draw[blue, ultra thick, opacity=0.4](3.5,1.5)|-(4,2.5)|-(5,4.5) |-(5.5,5.5) (0,0)|-(1,1)|-(3.5,1.5);
		\draw[gray] (5.5,4.5) grid[step=.5](6.5,5.5) rectangle (7,5) (5.5,4.5) rectangle (6,4) ;
		\draw[ultra thick, dashed, opacity=.5] (3.5,1.5)|-(4.5,2.5)|-(6.5,4.5) to (6.5,5.5)(1,0)--(1,1)(1,1)--(3,1)|-(3.5,1.5);
		\draw (-.25, 0.25) node {2}
			(-.25, 0.75) node {2}
			(-.25, 1.25) node {3}
			(-.25, 1.75) node {0}
			(-.25, 2.25) node {0}
			(-.25, 2.75) node {1}
			(-.25, 3.25) node {1}
			(-.25, 3.75) node {1}
			(-.25, 4.25) node {1}
			(-.25, 4.75) node {2}
			(-.25, 5.25) node {2};
			\end{tikzpicture}
	\end{center}
	\caption{Example of a path in $\mathcal{SQ}_n^E$, its bounce path and its labels. } \label{squarebounceex}
\end{figure}
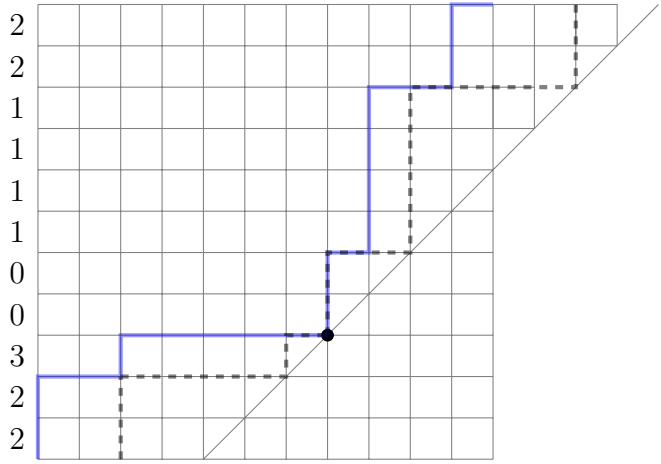
\begin{proposition} \label{prop:qtsquare_bijections}
There exist bijections
\begin{align}
\gamma_E' & : \mathcal{SQ}_n^E\rightarrow \mathcal{D}^{(0,0)}_n\sqcup \mathcal{D}_n^{(0,1)}=\widetilde{\mathcal{D}}^{(0,0)}_n\sqcup \widetilde{\mathcal{D}}_n^{(0,1)}\\
\gamma_N' & : \mathcal{SQ}_n^N\rightarrow \mathcal{D}_n^{(1,0)}=\widetilde{\mathcal{D}}^{(0,0)}_n\sqcup \widetilde{\mathcal{D}}_n^{(1,0)} =\overline{\mathcal{D}}^{(0,0)}_n\sqcup \overline{\mathcal{D}}_n^{(1,0)}
\end{align}
such that for all $P\in \mathcal{SQ}_n^E$ and $P'\in \mathcal{SQ}_n^N$ we have
\begin{align}
\area(\gamma_E'(P)) & =\area(P)  &   \area(\gamma_N'(P')) & =\area(P')  \\
\dinv(\gamma_E'(P)) & =\dinv(P) &    \dinv(\gamma_N'(P')) & =\dinv(P')  \\
\bounce(\gamma_E'(P)) & =\bounce(P) &   \bounce'(\gamma_N'(P')) & =\bounce(P').
\end{align}
\end{proposition}
\begin{proof}
We define $\gamma_E'$ in exactly the same way as the $\gamma_E$ map (\ref{gammamap}) defined in the previous section for labelled objects. 

We define a new map  $\gamma_N'$ similarly to $\gamma_E'$: the image of a path is constructed in exactly the same way except the portion of the path right before the added horizontal step is a north step, so we created a peak instead of a fall. Note that the paths that get sent into paths with a decoration on the last peak are exactly the paths whose $j=0$.

Now the statements of the proposition are straightforward to check.
\end{proof}

The following theorem with its proof provides the analogue of several results about the $q,t$-square in \cite{loehrwarringtonqtsquare} and \cite{canloehr}. For this reason we call it a \emph{new $q,t$-square} theorem.
\begin{theorem} We have
	\begin{align}
	 \langle \Delta_{e_{n-1}}e_n,e_n \rangle& =\sum_{P\in \mathcal{SQ}_n^E}q^{\area(P)}t^{\bounce(P)} = \sum_{P\in \mathcal{SQ}_n^E}q^{\dinv(P)} t^{\area(P)} \\ & =\sum_{P\in \mathcal{SQ}_n^N}q^{\area(P)}t^{\bounce(P)} =\sum_{P\in \mathcal{SQ}_n^N}q^{\dinv(P)} t^{\area(P)}. 
	\end{align}
\end{theorem}
\begin{proof}
We observe that Theorem~\ref{thm:decoqtSchroeder} gives
\begin{align}
\langle \Delta_{e_{n-1}}e_n,e_n \rangle & =\langle \Delta_{e_{n-1}}'e_n,e_n \rangle+\langle \Delta_{e_{n-2}}'e_n,e_n \rangle\\
& = \mathbf{D} \mathcal{D}_n^{(0,0)} + \mathbf{D} \mathcal{D}_n^{(0,1)}\\
& = \mathbf{D} \widetilde{\mathcal{D}}_n^{(0,0)} + \mathbf{D} \widetilde{\mathcal{D}}_n^{(0,1)}.
\end{align}
	
Using $\gamma'_E$ and $\gamma'_N$ of Proposition~\ref{prop:qtsquare_bijections}, Theorem~\ref{thm:decoqtSchroeder}, the $\zeta$ map of Theorem~\ref{thm:zeta_map} and the $\psi$ map of Theorem~\ref{thm:psi_map}, we prove the equalities in the diagram below. 

$$
\begin{tikzcd}
&&\sum_{P\in \mathcal{SQ}^E_n}q^{\dinv(P)}t^{\area(P)} \dar["\gamma'_E"] &\\ 
&&\mathbf D \widetilde{\mathcal{D}}^{(0,0)}_n+ \mathbf D \widetilde{\mathcal{D}}^{(0,1)} \dar["\zeta"]  &\\ 
&&\mathbf B \widehat{\mathcal{D}}^{(0,0)}_n+ \mathbf B \widehat{\mathcal{D}}^{(1,0)} \dar[equal] &\\  &\mathbf D \mathcal{D}^{(1,0)}\rar[equal] &\mathbf B \mathcal{D}^{(1,0)}\rar[equal] &	 \mathbf B' \mathcal{D}^{(1,0)} \\
 &\sum_{P\in \mathcal{SQ}^N_n}q^{\dinv(P)}t^{\area(P)} \uar["\gamma'_N"]&
 \mathbf B'\overline{\mathcal{D}}^{(0,0)}+\mathbf B'\overline{\mathcal{D}}^{(1,0)} \urar[equal]\dar["\psi"]
 & \sum_{P\in \mathcal{SQ}^N_n}q^{\area(P)}t^{\bounce(P)} \uar["\gamma'_N"]  \\
 &&\mathbf B'\overline{\mathcal{D}}^{(0,0)}+\mathbf B'\overline{\mathcal{D}}^{(0,1)}& \\
  &&\mathbf B\overline{\mathcal{D}}^{(0,0)}+\mathbf B\overline{\mathcal{D}}^{(0,1)}\uar[equal]& \\
  &&\sum_{P\in \mathcal{SQ}^E_n}q^{\area(P)}t^{\bounce(P)} \uar["\gamma'_E"]&
\end{tikzcd}
$$
All this proves the result.
\end{proof}

The most remarkable connection between the old $q,t$-square and the new one is the fact that they coincide at $t=1/q$. In fact, comparing \cite[Corollary~4.8.1]{haglundbook} (with $d=1$) with \cite[Theorem~10]{loehrwarringtonqtsquare}, we deduce immediately the following proposition.
\begin{proposition} \label{prop:coinc}
We have
\begin{equation} \label{eq:t_1_q_id}
\langle \Delta_{e_{n-1}}e_n,e_n \rangle\big{|}_{t=1/q} = \langle \Delta_{e_{n}}\omega (p_n),e_n \rangle\big{|}_{t=1/q}=q^{-\binom{n}{2}} \frac{1}{1+q^n}\begin{bmatrix}
2n\\
n\\
\end{bmatrix}_q.
\end{equation}
\end{proposition}

In fact, using \eqref{eq:lem_e_h_Delta}, we can rewrite \eqref{eq:t_1_q_id} as
\begin{equation}
\langle \Delta_{e_{n}}e_n,e_{n-1}e_1 \rangle\big{|}_{t=1/q} = \langle \Delta_{e_{n}}\omega (p_n),e_n \rangle\big{|}_{t=1/q}.
\end{equation}
In the next subsection we show that this is not an isolated phenomenon.

\subsection{$\Delta_f\omega(p_n)$ at $t=1/q$}

We start with the following theorem.
\begin{theorem} \label{thm:omegapn}
For any $f\in \Lambda^{(k)}$ we have
\begin{equation}
{\Delta_f \omega(p_n)}_{\big{|}_{t=1/q}}=\frac{[n]_qf[[n]_q]}{[k]_qq^{k(n-1)}}e_n[X[k]_q].
\end{equation}
\end{theorem}
\begin{proof}
By the definition of plethystic notation and Murnaghan-Nakayama rule, we have
\begin{align*}
\omega(p_n)/(1-q^n) & =(-1)^{n-1}p_n/(1-q^n)\\
& = ((-1)^{n-1}p_n)[X/(1-q)]\\
& = \left(\sum_{a=1}^n(-1)^{a-1}s_{(a,1^{n-a})}\right)[X/(1-q)]\\
& = \sum_{a=1}^n(-1)^{a-1}s_{(a,1^{n-a})}[X/(1-q)]
\end{align*}
It is well-known (see \cite{garsiahaimanqLagrange}) that
\begin{equation}
\widetilde{H}_{\mu}(X;q,1/q)=Cs_{\mu}\left[\frac{X}{1-q}\right]
\end{equation}
for all $\mu\vdash n$ and some constant $C$ (depending on $\mu$), so
\begin{align*}
{\Delta_f \omega(p_n)/(1-q^n)}_{\big{|}_{t=1/q}} & = \sum_{a=1}^n(-1)^{a-1}\Delta_f s_{(a,1^{n-a})}[X/(1-q)]\\
& = \sum_{a=1}^n(-1)^{a-1}f[B_{(a,1^{n-a})}(q,1/q)] s_{(a,1^{n-a})}[X/(1-q)].
\end{align*}
But notice that
\begin{equation}
B_{(a,1^{n-a})}(q,1/q)=q^{-(n-a)}[n]_q,
\end{equation}
so that
\begin{align*}
{\Delta_f \omega(p_n)/(1-q^n)}_{\big{|}_{t=1/q}} & = \sum_{a=1}^n(-1)^{a-1}f[B_{(a,1^{n-a})}(q,1/q)] s_{(a,1^{n-a})}[X/(1-q)]\\
& = \sum_{a=1}^n(-1)^{a-1}f[q^{-(n-a)}[n]_q] s_{(a,1^{n-a})}[X/(1-q)]\\
& = \sum_{a=1}^n(-1)^{a-1}q^{-k(n-a)}f[[n]_q] s_{(a,1^{n-a})}[X/(1-q)]\\
& =  f[[n]_q] q^{-k(n-1)}\sum_{a=1}^n(-1)^{a-1}q^{k(a-1)} s_{(a,1^{n-a})}[X/(1-q)]\\
& =  \frac{f[[n]_q]}{(1-q^k) q^{k(n-1)}}\sum_{a=1}^n(-q^k)^{a-1} (1-q^k) s_{(a,1^{n-a})}[X/(1-q)]
\end{align*}
\begin{align*}
\text{(using \eqref{eq:s_plethystic_eval})} & =  \frac{f[[n]_q]}{(1-q^k) q^{k(n-1)}}\sum_{a=1}^ns_{(n-a+1,1^{a-1})}[1-q^k] s_{(a,1^{n-a})}[X/(1-q)]\\
\text{(using \eqref{eq:s_plethystic_eval})}  & = \frac{f[[n]_q]}{(1-q^k) q^{k(n-1)}}\sum_{\mu\vdash n} s_{\mu'}[1-q^k] s_{\mu}[X/(1-q)]\\
\text{(using \eqref{eq:Cauchy_identities})}& = \frac{f[[n]_q]}{(1-q^k) q^{k(n-1)}}e[X[k]_q].
\end{align*}
This immediately implies our theorem.
\end{proof}
The following theorem, due to Haglund, Remmel and Wilson, is proved in a similar way.
\begin{theorem}[\cite{haglundremmelwilson} Theorem~5.1]
For any $f\in \Lambda^{(k)}$ we have
\begin{equation} \label{eq:HRW_lemma}
{\Delta_f e_n}_{\big{|}_{t=1/q}}=\frac{f[[n]_q]}{[k+1]_qq^{k(n-1)}}e_n[X[k+1]_q].
\end{equation}
\end{theorem}
We have an easy corollary, which is at the origin of our investigations. This corollary shows that Proposition~\ref{prop:coinc} is not an isolated coincidence.
\begin{corollary} \label{cor:omegapn}
For any $f\in \Lambda^{(k)}$ we have
\begin{equation}
{\langle\Delta_f e_n,e_{n-1}e_1\rangle}_{\big{|}_{t=1/q}}={\langle\Delta_f \omega(p_n),e_{n}\rangle}_{\big{|}_{t=1/q}}.
\end{equation}
\end{corollary}
\begin{proof}
Using \eqref{eq:h_q_binomial}, we have
\begin{equation}
h_n[[k]_q]=\begin{bmatrix}
n+k-1\\
k-1
\end{bmatrix}_q.
\end{equation}
Using Cauchy's identity \eqref{eq:Cauchy_identities}, we get
\begin{equation}
e[X[k]_q]=\sum_{\mu\vdash n}s_{\mu'}[X]s_{\mu}[[k]_q].
\end{equation}
Therefore
\begin{equation}
\langle e[X[k]_q],e_n\rangle = h_n[[k]_q]=\begin{bmatrix}
n+k-1\\
k-1
\end{bmatrix}_q,
\end{equation}
and similarly
\begin{equation}
\langle e[X[k+1]_q],e_{n-1}e_1\rangle  = (h_{n-1}h_1)[[k+1]_q]=\begin{bmatrix}
n+k-1\\
k
\end{bmatrix}_q [k+1]_q.
\end{equation}
So
\begin{align*}
{\langle\Delta_f \omega(p_n),e_{n}\rangle}_{\big{|}_{t=1/q}} & = \frac{[n]_qf[[n]_q]}{[k]_qq^{k(n-1)}}\langle e_n[X[k]_q],e_{n} \rangle\\
& = \frac{f[[n]_q]}{q^{k(n-1)}}\frac{[n]_q}{[k]_q}\begin{bmatrix}
n+k-1\\
k-1
\end{bmatrix}_q\\
& = \frac{f[[n]_q]}{q^{k(n-1)}}\begin{bmatrix}
n+k-1\\
k
\end{bmatrix}_q\\
& = \frac{f[[n]_q]}{[k+1]_qq^{k(n-1)}}\begin{bmatrix}
n+k-1\\
k
\end{bmatrix}_q[k+1]_q\\
& = \frac{f[[n]_q]}{[k+1]_qq^{k(n-1)}}\langle e[X[k+1]_q],e_{n-1}e_1\rangle\\
& = {\langle \Delta_f e_n,e_{n-1}e_1\rangle}_{\big{|}_{t=1/q}}.
\end{align*}
	
\end{proof}

The following result shows a surprisingly tight relation between the square conjecture in \cite{loehrwarringtonqtsquare} and our Conjecture~\ref{conj:new_square}.
\begin{theorem} \label{thm:square_cnjs_at_1_q}
We have
\begin{equation}
{\Delta_{e_n} \omega(p_n)}_{\big{|}_{t=1/q}}={\Delta_{e_{n-1}}e_n}_{\big{|}_{t=1/q}}.
\end{equation}
\end{theorem}
\begin{proof}
Using \eqref{eq:e_q_binomial}, we have
\begin{equation} \label{eq:aux_e_qbin}
e_n[[n]_{q}]=q^{\binom{n}{2}}\quad \text{ and }\quad
e_{n-1}[[n]_{q}]=q^{\binom{n-1}{2}}[n]_q.
\end{equation}
	
Using Theorem~\ref{thm:omegapn} we have
\begin{align*}
{\Delta_{e_n} \omega(p_n)}_{\big{|}_{t=1/q}} & =
\frac{[n]_qe_{n}[[n]_q]}{[n]_qq^{n(n-1)}}e_n[X[n]_q]\\
\text{(using \eqref{eq:aux_e_qbin})}& = \frac{[n]_qq^{\binom{n}{2}}}{[n]_qq^{n(n-1)}}e_n[X[n]_q]\\
& = \frac{[n]_qq^{\binom{n-1}{2}+(n-1)}}{[n]_qq^{n(n-1)}}e_n[X[n]_q]
\end{align*}
\begin{align*}
& = \frac{[n]_qq^{\binom{n-1}{2}}}{[n]_qq^{(n-1)^2}}e_n[X[n]_q]\\
\text{(using \eqref{eq:aux_e_qbin})}& = \frac{e_{n-1} [[n]_q]}{[n]_qq^{(n-1)^2}}e_n[X[n]_q]\\
\text{(using \eqref{eq:HRW_lemma})} & = {\Delta_{e_{n-1}} e_n}_{\big{|}_{t=1/q}}.
\end{align*}
\end{proof}

\section{Open problems}

Other than the open questions in \cite{haglundremmelwilson} that we didn't answer, our work leaves open a few natural combinatorial questions.

For example, consider the identity
\begin{equation}
\mathbf{D} \widetilde{S}_{n,k}^{(a,b)} = \sum_{j=0}^{\min(n-k,b)} \mathbf{B}' \overline{T}_{n,k+j,j}^{(a,b)},
\end{equation}
which is given by combining Theorem~\ref{thm:qt_enumerators_formulae} with Theorem~\ref{thm:rel_F_H}. There should be a bijective explanation of this identity. Notice that, summing over $k$, this would lead to a bijective explanation of the identity
\begin{equation}
\mathbf{D} \widetilde{\mathcal{D}}_{n}^{(a,b)} = \mathbf{B}' \overline{\mathcal{D}}_{n}^{(a,b)},
\end{equation}
coming from Theorem~\ref{thm:decoqtSchroeder}.

We observe here that for $b=0$ these explanations are provided by the bijection given in \cite{ehkk}, so it is conceivable to look for an extension of that bijection to objects with decorated rises.

Another natural problem is the following: assuming Conjecture~\ref{conj:new_square} and recalling the results in \cite{leven}, find a combinatorial explanation of Theorem~\ref{thm:square_cnjs_at_1_q}.

\section{Acknowledgments}

We thank Alessandro Iraci for useful conversations. The first author is pleased to thank Adriano Garsia, who, after countless hours of explanations along several years, finally converted him to this fascinating subject.

\section*{Appendix: proofs of elementary lemmas}

\subsection*{Proof of Lemma~\ref{lem:elementary4}}

Using \eqref{eq:qbin_recursion}, we have
\begin{equation*}
q^{s}\begin{bmatrix}
h-1\\
s\\
\end{bmatrix}_q \begin{bmatrix}
h+k-s-1\\
h\\
\end{bmatrix}_q + \begin{bmatrix}
h-1\\
s-1\\
\end{bmatrix}_q \begin{bmatrix}
h+k-s\\
h\\
\end{bmatrix}_q=\qquad \qquad \qquad \qquad \qquad 
\end{equation*}
\begin{align*}
& = q^{s}\begin{bmatrix}
h-1\\
s\\
\end{bmatrix}_q \begin{bmatrix}
h+k-s-1\\
h\\
\end{bmatrix}_q \\
& \quad +  \begin{bmatrix}
h-1\\
s-1\\
\end{bmatrix}_q\left( \begin{bmatrix}
h+k-s-1\\
k-s-1\\
\end{bmatrix}_q+q^{k-s} \begin{bmatrix}
h+k-s-1\\
k-s\\
\end{bmatrix}_q\right)
\end{align*}
\begin{align*}
\qquad \quad & = \begin{bmatrix}
h\\
s\\
\end{bmatrix}_q \begin{bmatrix}
h+k-s-1\\
k-s-1\\
\end{bmatrix}_q+ q^{k-s}\begin{bmatrix}
h-1\\
s-1\\
\end{bmatrix}_q \begin{bmatrix}
h+k-s-1\\
k-s\\
\end{bmatrix}_q \\
& = \frac{[h]_q!}{[s]_q![h-s]_q!}\frac{[h+k-s-1]_q!}{[h]_q![k-s-1]_q!} +q^{k-s}\frac{[h-1]_q!}{[s-1]_q![h-s]_q!}\frac{[h+k-s-1]_q!}{[h-1]_q![k-s]_q!}\\
& = \frac{[k]_q!}{[s]_q![k-s]_q!}\frac{[h+k-s-1]_q!}{[k-1]_q![h-s]_q!} \frac{1}{[k]_q} \left( [k-s]_q+q^{k-s}[s]_q\right)\\
& =\begin{bmatrix}
k\\
s\\
\end{bmatrix}_q \begin{bmatrix}
h+k-s-1\\
h-s\\
\end{bmatrix}_q.
\end{align*}
This completes the proof of the lemma.

\subsection*{Proof of Lemma~\ref{lem:elementary1}}

	Let us call $f(s,i,a)$ the left hand side of \eqref{eq:first_qlemma}. 
	First we rewrite
	\begin{align*}
	f(s,i,a) & = \sum_{r=1}^{i}\begin{bmatrix}
	i-1\\
	r-1
	\end{bmatrix}_q \begin{bmatrix}
	r+s+a-1\\
	s-1
	\end{bmatrix}_q q^{\binom{r}{2}+r-ir}(-1)^{i-r}\\
	& = \sum_{j=0}^{i-1} \begin{bmatrix}
	i-1\\
	j
	\end{bmatrix}_q  \begin{bmatrix}
	s+a+j\\
	s-1
	\end{bmatrix}_q q^{\binom{j+2}{2}-ij-i}(-1)^{i-j-1}.
	\end{align*}
	Then, using \eqref{eq:qbin_recursion}, we have
	\begin{align*}
	f(s,i,a) & = \sum_{j=0}^{i-1} \begin{bmatrix}
	i-1\\
	j
	\end{bmatrix}_q  \begin{bmatrix}
	s+a+j\\
	s-1
	\end{bmatrix}_q q^{\binom{j+2}{2}-ij-i}(-1)^{i-j-1}\\
	& = \sum_{j=0}^{i-1}\left(q^j\begin{bmatrix}
	i-2\\
	j
	\end{bmatrix}_q+\begin{bmatrix}
	i-2\\
	j-1
	\end{bmatrix}_q\right) \begin{bmatrix}
	s+a+j\\
	s-1
	\end{bmatrix}_q q^{\binom{j+2}{2}-ij-i}(-1)^{i-j-1}\\
	\\
	& = -q^{-1}\sum_{j=0}^{i-1}\begin{bmatrix}
	i-2\\
	j
	\end{bmatrix}_q \begin{bmatrix}
	s+a+j\\
	s-1
	\end{bmatrix}_q q^{\binom{j+2}{2}-(i-1)j-(i-1)}(-1)^{(i-1)-j-1}\\
	\\
	& \quad + \sum_{j=0}^{i-1} \begin{bmatrix}
	i-2\\
	j-1
	\end{bmatrix}_q  \begin{bmatrix}
	s+a+j\\
	s-1
	\end{bmatrix}_q q^{\binom{j+2}{2}-ij-i}(-1)^{i-j-1}\\
	& =-q^{-1}f(s,i-1,a)\\
	& \quad + \sum_{h=0}^{i-2} \begin{bmatrix}
	i-2\\
	h
	\end{bmatrix}_q  \begin{bmatrix}
	s+a+1+h\\
	s-1
	\end{bmatrix}_q q^{\binom{h+2}{2}+h+2-ih-i-i}(-1)^{i-h-2}\\
	&=-q^{-1}f(s,i-1,a)+q^{1-i}f(s,i-1,a+1).
	\end{align*}
	So by induction,
	\begin{align*}
	f(s,i,a)&=-q^{-1}f(s,i-1,a)+q^{1-i}f(s,i-1,a+1)\\
	& =-q^{-1}q^{\binom{i-1}{2}+(i-2)a} \begin{bmatrix}
	s+a\\
	i-1+a
	\end{bmatrix}_q+q^{1-i}q^{\binom{i-1}{2}+(i-2)(a+1)} \begin{bmatrix}
	s+a+1\\
	i+a
	\end{bmatrix}_q
		\end{align*}
			\begin{align*}
	& =q^{\binom{i-1}{2}+(i-2)a-1}\left(\begin{bmatrix}
	s+a+1\\
	i+a
	\end{bmatrix}_q- \begin{bmatrix}
	s+a\\
	i-1+a
	\end{bmatrix}_q \right)\\
	& =q^{\binom{i-1}{2}+(i-2)a-1}q^{i+a}\begin{bmatrix}
	s+a\\
	i+a
	\end{bmatrix}_q \\
	& =q^{\binom{i}{2}+(i-1)a} \begin{bmatrix}
	s+a\\
	i+a
	\end{bmatrix}_q,
	\end{align*}
as we wanted. 

The base cases are easy to check, and the lemma is proved.

\subsection*{Proof of Lemma~\ref{lem:elementary2}}

	Let us call $f(k,s,j)$ the right hand side of \eqref{eq:qlemma3}. 
	
	We proceed by induction on $k$, $s$ and $j$. 
	
	For the inductive step, using \eqref{eq:qbin_recursion}, we have
	\begin{equation*}
	f(k,s,j) =\qquad \qquad \qquad \qquad \qquad \qquad \qquad \qquad \qquad \qquad \qquad \qquad 
	\end{equation*}
	\begin{align*}
	& = \sum_{a=0}^{k-s} \sum_{b= 0 }^{a}h_{k-s-a} \left[\frac{1}{1-q}\right]q^{\binom{a}{2}+s\cdot a} \begin{bmatrix}
	s+j+b\\
	s+a
	\end{bmatrix}_q e_{b} \left[-\frac{1}{1-q}\right]  e_{a-b} \left[\frac{1}{1-q}\right]\\
	& = \sum_{a=0}^{k-s} \sum_{b= 0 }^{a}h_{k-s-a} \left[\frac{1}{1-q}\right]q^{\binom{a}{2}+s\cdot a} \begin{bmatrix}
	s+j-1+b\\
	j-1-a+b
	\end{bmatrix}_q e_{b} \left[-\frac{1}{1-q}\right]  e_{a-b} \left[\frac{1}{1-q}\right]\\
	& \quad + \sum_{a=0}^{k-s} \sum_{b= 0 }^{a}h_{k-s-a} \left[\frac{1}{1-q}\right]q^{\binom{a}{2}+s\cdot a} q^{j-a+b} \begin{bmatrix}
	s+j+b-1\\
	j-a+b
	\end{bmatrix}_q e_{b} \left[-\frac{1}{1-q}\right]  e_{a-b} \left[\frac{1}{1-q}\right]\\
	& = f(k,s,j-1)+\\
	& \quad + q^j\sum_{a=0}^{k-s} \sum_{b= 0 }^{a}h_{k-s-a} \left[\frac{1}{1-q}\right]q^{\binom{a}{2}+(s-1)\cdot a} \begin{bmatrix}
	s-1+j+b\\
	j-a+b
	\end{bmatrix}_q e_{b} \left[-\frac{1}{1-q}\right]  e_{a-b} \left[\frac{1}{1-q}\right]\\
	& \quad + q^{j} \sum_{a=0}^{k-s} \sum_{b= 0 }^{a}h_{k-s-a} \left[\frac{1}{1-q}\right]q^{\binom{a}{2}+(s-1)\cdot a} \begin{bmatrix}
	s+j+b-1\\
	j-a+b
	\end{bmatrix}_q (q^b-1)e_{b} \left[-\frac{1}{1-q}\right]  e_{a-b} \left[\frac{1}{1-q}\right]
		\end{align*}
			\begin{align*}
	& = f(k,s,j-1)+q^jf(k-1,s-1,j)+\\
	& \quad + q^{j} \sum_{a=1}^{k-s} \sum_{b= 1 }^{a}h_{k-s-a} \left[\frac{1}{1-q}\right]q^{\binom{a}{2}+(s-1)\cdot a} \begin{bmatrix}
	s+j+b-1\\
	j-a+b
	\end{bmatrix}_q e_{b-1} \left[-\frac{1}{1-q}\right]  e_{a-b} \left[\frac{1}{1-q}\right]\\
	& = f(k,s,j-1)+q^jf(k-1,s-1,j)+\\
	& \quad + q^{j} \sum_{\tilde{a}=0}^{k-1-s} \sum_{\tilde{b}= 0 }^{\tilde{a}}h_{k-1-s-\tilde{a}} \left[\frac{1}{1-q}\right]q^{\binom{\tilde{a}}{2}+\tilde{a}+(s-1)\cdot (\tilde{a}+1)} \begin{bmatrix}
	s+j+\tilde{b}\\
	j-\tilde{a}+\tilde{b}
	\end{bmatrix}_q e_{\tilde{b}} \left[-\frac{1}{1-q}\right]  e_{\tilde{a}-\tilde{b}} \left[\frac{1}{1-q}\right]\\
	& = f(k,s,j-1)+q^jf(k-1,s-1,j)+q^{j+s-1}f(k-1,s,j).
	\end{align*}
	By induction, this gives
	\begin{align*}
	f(k,s,j)& = f(k,s,j-1)+q^jf(k-1,s-1,j)+q^{j+s-1}f(k-1,s,j)\\
	& = \begin{bmatrix}
	s+j-1\\
	j-1
	\end{bmatrix}_q \begin{bmatrix}
	k+j-2\\
	k-s
	\end{bmatrix}_q
	+q^j \begin{bmatrix}
	s+j-1\\
	j
	\end{bmatrix}_q \begin{bmatrix}
	k+j-2\\
	k-s
	\end{bmatrix}_q \\
	&\quad +q^{j+s-1} \begin{bmatrix}
	s+j\\
	j
	\end{bmatrix}_q \begin{bmatrix}
	k+j-2\\
	s+j-1
	\end{bmatrix}_q\\
	& = \begin{bmatrix}
	s+j\\
	j
	\end{bmatrix}_q \begin{bmatrix}
	k+j-2\\
	s+j-2
	\end{bmatrix}_q +q^{j+s-1} \begin{bmatrix}
	s+j\\
	j
	\end{bmatrix}_q \begin{bmatrix}
	k+j-2\\
	s+j-1
	\end{bmatrix}_q\\
	& =\begin{bmatrix}
	s+j\\
	j
	\end{bmatrix}_q \begin{bmatrix}
	k+j-1\\
	s+j-1
	\end{bmatrix}_q,
	\end{align*}
	as we wanted. 
	
	The base cases are easy to check. This completes the proof of the lemma.

\subsection*{Proof of Lemma~\ref{lem:elementary3}}

We have
	\begin{align*}
	\begin{bmatrix}
	k\\
	s
	\end{bmatrix}_q \begin{bmatrix}
	k+j-1\\
	j
	\end{bmatrix}_q (1-q^{s+j}) & = (1-q)\begin{bmatrix}
	k\\
	s
	\end{bmatrix}_q \begin{bmatrix}
	k+j-1\\
	j
	\end{bmatrix}_q [s+j]_q \frac{[s+j-1]_q!}{[s+j-1]_q!}\\
	& = (1-q)\frac{[k]_q!}{[k-s]_q![s]_q} \frac{[k+j-1]_q!}{[k-1]_q![j]_q}\frac{[s+j]_q!}{[s+j-1]_q!}\\
	& = (1-q)[k]_q\begin{bmatrix}
	s+j\\
	s
	\end{bmatrix}_q \begin{bmatrix}
	k+j-1\\
	k-s
	\end{bmatrix}_q \\
	& = (1-q^k) \begin{bmatrix}
	s+j\\
	s
	\end{bmatrix}_q \begin{bmatrix}
	k+j-1\\
	k-s
	\end{bmatrix}_q .
	\end{align*}
This proves the lemma.

\end{document}